\documentclass{amsart}
\usepackage{amsthm, graphicx, graphics, latexsym, mathtools, color, amssymb}
\usepackage{psfrag,epsfig}
\usepackage{multirow, multicol}
\usepackage{subfigure}
\usepackage[latin1]{inputenc}
\usepackage{tikz}
\usetikzlibrary{shapes,snakes,arrows}
\usepackage{algorithm}
\usepackage{algorithmic}

\newtheorem{lemma}{Lemma}[section]
\newtheorem{corollary}[lemma]{Corollary}
\newtheorem{theorem}[lemma]{Theorem}
\newtheorem{example}[lemma]{Example}
\newtheorem{proposition}[lemma]{Proposition}

\theoremstyle{definition}
\newtheorem{definition}[lemma]{Definition}
\theoremstyle{notation}
\newtheorem{notation}[lemma]{Notation}

\DeclareMathOperator{\lcm}{lcm}
\DeclareMathOperator{\cox}{cox}
\DeclareMathOperator{\Cox}{Cox}
\DeclareMathOperator{\modm}{mod}
\newcommand{\ivec}{\underline{i}}

\def\Prufer#1#2#3#4{
	\coordinate (x) at (#1,#2);
	\fill (x) circle (.1);
	\draw[#4] (x) .. controls (#1,#2+2.5) and (#1+.5,#2+3) .. (#1+#3,#2+3);
}

\def\adic#1#2#3#4{
	\coordinate (x) at (#1,#2);
	\fill (x) circle (.1);
	\draw[#4] (x) .. controls (#1,#2+2.5) and (#1-.5,#2+3) .. (#1-#3,#2+3);
}

\def\hadic#1#2#3#4{
	\coordinate (x) at (#1,#2);
	\fill (x) circle (.1);
	\draw[#4] (x) .. controls (#1,#2-2.5) and (#1+.5,#2-3) .. (#1+#3,#2-3);
}

\def\hPrufer#1#2#3#4{
	\coordinate (x) at (#1,#2);
	\fill (x) circle (.1);
	\draw[#4] (x) .. controls (#1,#2-2.5) and (#1-.5,#2-3) .. (#1-#3,#2-3);
}

\title[Asymptotic triangulations and Coxeter transformations of $C_{p,q}$]{Asymptotic triangulations and Coxeter transformations of the annulus}
\author[H.~Vogel]{Hannah~Vogel, with an appendix by Anna~Felikson \and  Pavel~Tumarkin}
\address{University of Graz\\
        NAWI Graz\\
        Institute for Mathematics and Scientific Computing\\
        Heinrichstra\ss e 36\\
        8010 Graz, Austria}
\email{hannah.vogel@uni-graz.at}

\begin{document}
\maketitle
\begin{abstract}
Asymptotic triangulations can be viewed as limits of triangulations under the action of the mapping class group. In the case of the annulus, such triangulations have been introduced in \cite{BD}. We construct an alternative method of obtaining these asymptotic triangulations using Coxeter transformations. This provides us with an algebraic and combinatorial framework for studying these limits via the associated quivers.
\end{abstract}


\section{Introduction}

Coxeter transformation are important in the study of representations of algebras, quivers, partially ordered sets, and lattices. In this article, we describe how Coxeter transformations act on triangulations. We first briefly introduce Coxeter transformations.

Let $\Gamma$ be an oriented graph with vertex set $\Gamma_0$, $|\Gamma_0| = n$, and edge set $\Gamma_1$. An arrow $\alpha \in \Gamma_1$, $\alpha: i \rightarrow j$, \emph{starts} at $s(\alpha) = i$, and \emph{terminates} at $t(\alpha) = j$.

To $\Gamma$, we can associate its \emph{Euler form}, a bilinear form on $\mathbb{Z}^n$:

\[ \langle -, - \rangle : \mathbb{Z}^n \times \mathbb{Z}^n \longrightarrow \mathbb{Z} \mbox{\,\,\,\, with \,\,\,\,} \langle x,y\rangle = \sum_{i \in \Gamma_0}x_iy_i - \sum_{\alpha \in \Gamma_1}x_{s(\alpha)}y_{t(\alpha)}.
\]

We obtain the following symmetric bilinear form on $\mathbb{Z}^n$:
\[
(x,y) = \langle x,y\rangle + \langle y,x\rangle.
\]

If $\Gamma$ has no loops, we can define the \emph{reflection map} with respect to a vertex $i$:
\[
\sigma_i : \mathbb{Z}^n \longrightarrow \mathbb{Z}^n \mbox{\,\,\,\, with \,\,\,\,} \sigma_i(x) = x - \frac{2(x,e_i)}{(e_i, e_i)}e_i,
\]
where $e_i$ is the $i$th coordinate vector. The $\sigma_i$ are automorphisms of $\mathbb{Z}^n$ of order two that preserve the bilinear form $(-,-)$. 

 A vertex $i$ of $\Gamma$ is called a \emph{source} (resp. \emph{sink}) if there is no arrow in $\Gamma$ ending (resp. starting) at $i$. If $i$ is a source or a sink, the graph $\sigma_i\Gamma$ is obtained from $\Gamma$ by reversing all arrows which start or end at $i$.

\begin{definition}
An ordering $i_1, \ldots, i_n$ of the vertices of $\Gamma$ is called \emph{source-admissible} if for each $p$ the vertex $i_p$ is a source for $\sigma_{i_{p-1}} \ldots \sigma_{i_1}\Gamma$.
\label{def:admiss}
\end{definition}

In this case we have that $$ \sigma_{i_n}\sigma_{i_{n-1}} \ldots \sigma_{i_2} \sigma_{i_1}\Gamma = \Gamma.$$
Now if $\Gamma$ is an acyclic graph, and $i_1, \ldots, i_n$ is an admissible ordering of its vertices, then the automorphism
\[
c: \mathbb{Z}^n \longrightarrow \mathbb{Z}^n \mbox{\,\,\,\, with \,\,\,\,} c(x) = \sigma_{i_n} \ldots \sigma_{i_1}(x)
\]
is called a \emph{Coxeter transformation}.

For oriented trees, there always exists an admissible ordering $\ivec$. To every such sequence, we assign a Coxeter transformation depending on the order of the vertices in $\ivec$:
\[
c = \sigma_{i_n}\sigma_{i_{n-1}} \ldots \sigma_{i_2} \sigma_{i_1}
\] 
For every orientation of a given simply-laced Dynkin diagram, every admissible ordering gives rise to the same Coxeter transformation \cite{K}. If the underlying graph is not a tree, we need to consider the orientations of the arrows in the graph before we can assign a Coxeter transformation to the graph.

\medskip
Asymptotic triangulations were introduced by Baur and Dupont in \cite{BD}, with respect to unpunctured marked surfaces. These asymptotic triangulations can be mutated as usual triangulations, and they provide a natural way to compactify the usual exchange graph of the triangulations of an annulus. 

\medskip
In this article, we will focus on triangulations of annuli. It is known that such triangulations give rise to cluster algebras of extended Dynkin type $\tilde{A}_n$. We introduce triangulations and quivers in Section 2, and define asymptotic triangulations and their associated quivers in Section 3.  In Sections 4, 5 \& 6 we discuss sequences of flips in triangulations that correspond to Coxeter transformations on the associated quiver, and we describe what happens in the limit of these transformations. Appendix A gives an alternative way to perform quiver mutation for quivers associated to asymptotic triangulations by using potentials, and Appendix B introduces an alternative cluster structure on asymptotic triangulations, and gives a geometric interpretation of these structures. 

\medskip
\textbf{Acknowledgements:} The author would like to thank her supervisor Karin Baur for the many helpful discussions, as well as Gregg Musiker for his insightful discussions on using quivers with potentials, and Anna Felikson and Pavel Tumarkin for their work that is included in Appendix B. 
The author was supported by the Austrian Science Fund (FWF): projects No.\ P25141-N26 and W1230, and acknowledges support from NAWI Graz.

\section{Definitions and notation}

\subsection{Triangulations}

Let $S$ be a connected, oriented Riemann surface with boundary, and let $M$ be a finite set of marked points in the closure of $S$. We assume that $M$ is non-empty, and there is at least one marked point on each boundary component. We choose a counter-clockwise orientation of $S$ and label the marked points on each boundary component in a counter-clockwise order.

\begin{definition} An \emph{arc} $\gamma$ of a marked surface $S$ is a curve whose endpoints are marked points of $S$, and which does not intersect itself in the interior of $S$. The interior of the arc is disjoint from the boundary of $S$ and it does not cut out an unpunctured monogon or digon.
\end{definition}

\begin{figure}[ht]

\begin{tikzpicture}[scale = 1.5]
		\tikzstyle{every node} = [font = \small]
		\foreach \x in {0}
		{
			\foreach \y in {0}
			{
			\draw (\x,\y) circle (1);
			\fill(\x+.866,\y+0.5) circle (.05);
			\fill(\x-.866,\y+0.5) circle (.05);	
			\fill(\x-.5,\y-0.866) circle (.05);
			\fill(\x+.5,\y-0.866) circle (.05);
			\fill(\x+0,\y-.4) circle (.05);

			\draw [thick,red] (\x-.866,\y+.5) .. controls (\x-.5,\y+.3) and (\x+.5,\y+.3) ..  (\x+.866,\y+.5);
			\fill (\x-.3,\y+.4) node [above] {$\gamma$};
			\draw [thick,red] (\x-.5,\y-.866) .. controls (\x-.25,\y-.02) and (\x+.25,\y-.02) ..  (\x+.5,\y-.866);
			\fill (\x+.25,\y-.3) node [right] {$\eta$};
		
			}
		}
	\end{tikzpicture}
\caption{Two arcs $\gamma$, $\eta$ of a marked surface $S$.}
\label{two_curves}
\end{figure}
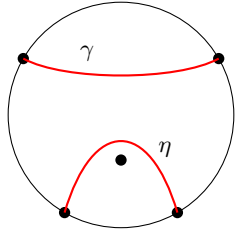

We consider arcs in $(S,M)$ up to isotopy. We write $\gamma = [i,j]$ to denote the arc with endpoints $i,j \in M$. Note that depending on the surface, there may be multiple (non-isotopic) arcs with the same endpoints.

\begin{definition} 
An \emph{ideal triangulation} $T$ of a surface $S$ is a maximal collection of pairwise non-intersecting arcs of $S$.
\end{definition}

\begin{figure}
\centering
\subfigure{
	\begin{tikzpicture}[scale = .75]
		\tikzstyle{every node} = [font = \small]
		\foreach \x in {0}
		{
			\foreach \y in {-8}
			{
			\fill (\x,\y+1.75) circle (.05);
			\fill (\x,\y+1.75) node [above] {\tiny{$1$}};
			\fill(\x+1.5158,\y+.875) circle (.05);
			\fill (\x+1.5158,\y+.875) node [right] {\tiny{$2 $}};
			\fill (\x+1.5158,\y-.875) circle (.05);
			\fill (\x+1.5158,\y-.875) node [right] {\tiny{$3$}};
			\fill(\x,\y-1.75) circle (.05);
			\fill (\x,\y-1.75) node [below] {\tiny{$4$}};
			\fill(\x-1.5158,\y-.875) circle (.05);
			\fill (\x-1.5158,\y-.875) node [left] {\tiny{$5$}};
			\fill (\x-1.5158,\y+.875) circle (.05);
			\fill (\x-1.5158,\y+.875) node [left] {\tiny{$6$}};
		
			\draw [] (\x,\y+1.75)--(\x+1.5158,\y+.875); 
			\draw [] (\x+1.5158,\y+.875) -- (\x+1.5158,\y-.875); 
			\draw [] (\x+1.5158,\y-.875)--(\x,\y-1.75); 
			\draw [] (\x,\y-1.75) -- (\x-1.5158,\y-.875); 
			\draw [] (\x-1.5158,\y-.875)--(\x-1.5158,\y+.875); 
			\draw [] (\x-1.5158,\y+.875) -- (\x,\y+1.75); 
		
			\draw [] (\x,\y+1.75) -- (\x-1.5158,\y-.875);
			\draw [] (\x,\y+1.75) -- (\x,\y-1.75);
			\draw [] (\x,\y+1.75) -- (\x+1.5158,\y-.875);
		
			}
		}
	\end{tikzpicture}
	}
	\quad
	\subfigure{
		\begin{tikzpicture}[scale = .75]
		\tikzstyle{every node} = [font = \small]
		\foreach \x in {0}
		{
			\foreach \y in {-8}
			{
			\fill (\x,\y+1.75) circle (.05);
			\fill (\x,\y+1.75) node [above] {\tiny{$1$}};
			\fill(\x+1.5158,\y+.875) circle (.05);
			\fill (\x+1.5158,\y+.875) node [right] {\tiny{$2 $}};
			\fill (\x+1.5158,\y-.875) circle (.05);
			\fill (\x+1.5158,\y-.875) node [right] {\tiny{$3$}};
			\fill(\x,\y-1.75) circle (.05);
			\fill (\x,\y-1.75) node [below] {\tiny{$4$}};
			\fill(\x-1.5158,\y-.875) circle (.05);
			\fill (\x-1.5158,\y-.875) node [left] {\tiny{$5$}};
			\fill (\x-1.5158,\y+.875) circle (.05);
			\fill (\x-1.5158,\y+.875) node [left] {\tiny{$6$}};
		
			\draw [] (\x,\y+1.75)--(\x+1.5158,\y+.875); 
			\draw [] (\x+1.5158,\y+.875) -- (\x+1.5158,\y-.875); 
			\draw [] (\x+1.5158,\y-.875)--(\x,\y-1.75); 
			\draw [] (\x,\y-1.75) -- (\x-1.5158,\y-.875); 
			\draw [] (\x-1.5158,\y-.875)--(\x-1.5158,\y+.875); 
			\draw [] (\x-1.5158,\y+.875) -- (\x,\y+1.75); 
		
			\draw [] (\x-1.5158,\y+.875) -- (\x+1.5158,\y+.875);
			\draw [] (\x+1.5158,\y+.875) -- (\x-1.5158,\y-.875);
			\draw [] (\x-1.5158,\y-.875) -- (\x+1.5158,\y-.875);
		
			}
		}
	\end{tikzpicture}
	}
	\caption{Two triangulations of the hexagon $P_6$.}
	\label{triang_hex2}
	\end{figure}
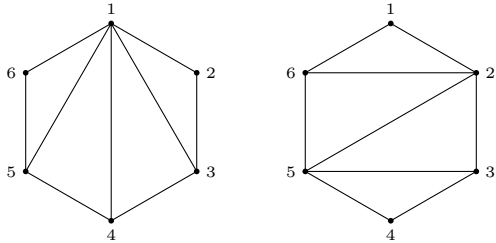 
	
Throughout this paper, all triangulations will be ideal. The arcs of a triangulation $T$ cut $S$ into (\emph{ideal}) \emph{triangles}. Triangles are three-sided regions, and self-folded triangles may occur.

\begin{definition}
A \emph{flip} $\mu$ of an arc in a triangulation is a move that replaces an arc of any given quadrilateral with the other arc in the quadrilateral. We sometimes use $\mu_k$ to indicate mutation at an arc $d_k$.
\end{definition}

\begin{figure}[ht]
\centering
\subfigure{
	\begin{tikzpicture}[scale = 1]
		\tikzstyle{every node} = [font = \small]
		\foreach \x in {0}
		{
			\foreach \y in {0}
			{
			\draw(\x+0,\y+0) rectangle (3,2);
			\fill(\x+3,\y+0) circle (.07);
			\fill (\x+3,\y+0) node [right] {$k$};

			\fill(\x+3,\y+2) circle (.07);
			\fill (\x+3,\y+2) node [right] {$j$};
			
			\fill(\x+0,\y+2) circle (.07);
			\fill (\x+0,\y+2) node [left] {$i$};
			
			\fill(\x+0,\y+0) circle (.07);
			\fill (\x+0,\y+0) node [left] {$l$};
			
			\draw[thick] (\x,\y) -- (\x+3,\y+2);
			\fill (\x+1.35,\y+.75) node [right] {\tiny{$\gamma = [j,l]$}};
			
						}
		}
	\draw[<->] (4,1) -- (4.5,1);
	\end{tikzpicture}
	}
	\quad
	\subfigure{
	\begin{tikzpicture}[scale = 1]
		\tikzstyle{every node} = [font = \small]
		\foreach \x in {0}
		{
			\foreach \y in {0}
			{

			\draw(\x+0,\y+0) rectangle (3,2);

						\fill(\x+3,\y+0) circle (.07);
			\fill (\x+3,\y+0) node [right] {$k$};

			\fill(\x+3,\y+2) circle (.07);
			\fill (\x+3,\y+2) node [right] {$j$};
			
			\fill(\x+0,\y+2) circle (.07);
			\fill (\x+0,\y+2) node [left] {$i$};
			
			\fill(\x+0,\y+0) circle (.07);
			\fill (\x+0,\y+0) node [left] {$l$};
			
			\draw[thick] (\x,\y+2) -- (\x+3,\y);
			\fill (\x+1.35,\y+1.25) node [right] {\tiny{$\gamma' = [i,k]$}};

		}
		}
	\end{tikzpicture}
	}

\caption{Flip of the arc $\gamma$}
\end{figure}
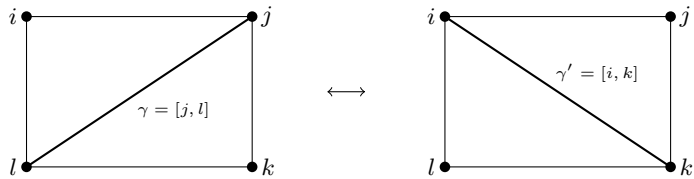

Given a marked surface $(S,M)$, there may exist many different triangulations of $S$ (cf. Figure \ref{triang_hex2}). By a theorem of \cite{H}, any two triangulations of a surface $S$ are related by a sequence of flips. For an example, see Figure \ref{seq_flips}.

\begin{figure}
\centering
\subfigure{
	\begin{tikzpicture}[scale = .5]
		\tikzstyle{every node} = [font = \small]
		\foreach \x in {0}
		{
			\foreach \y in {-8}
			{
			\fill (\x,\y+1.75) circle (.05);
			\fill (\x,\y+1.75) node [above] {\tiny{$1$}};
			\fill(\x+1.5158,\y+.875) circle (.05);
			\fill (\x+1.5158,\y+.875) node [right] {\tiny{$2 $}};
			\fill (\x+1.5158,\y-.875) circle (.05);
			\fill (\x+1.5158,\y-.875) node [right] {\tiny{$3$}};
			\fill(\x,\y-1.75) circle (.05);
			\fill (\x,\y-1.75) node [below] {\tiny{$4$}};
			\fill(\x-1.5158,\y-.875) circle (.05);
			\fill (\x-1.5158,\y-.875) node [left] {\tiny{$5$}};
			\fill (\x-1.5158,\y+.875) circle (.05);
			\fill (\x-1.5158,\y+.875) node [left] {\tiny{$6$}};
		
			\draw [] (\x,\y+1.75)--(\x+1.5158,\y+.875); 
			\draw [] (\x+1.5158,\y+.875) -- (\x+1.5158,\y-.875); 
			\draw [] (\x+1.5158,\y-.875)--(\x,\y-1.75); 
			\draw [] (\x,\y-1.75) -- (\x-1.5158,\y-.875); 
			\draw [] (\x-1.5158,\y-.875)--(\x-1.5158,\y+.875); 
			\draw [] (\x-1.5158,\y+.875) -- (\x,\y+1.75); 
		
			\draw [] (\x,\y+1.75) -- (\x-1.5158,\y-.875);
			\draw [] (\x,\y+1.75) -- (\x,\y-1.75);
			\draw [] (\x,\y+1.75) -- (\x+1.5158,\y-.875);
		
			}
		}
		\draw[<->] (3,-8) -- (4,-8);
	\end{tikzpicture}
	}
	\quad
	\subfigure{
	\begin{tikzpicture}[scale = .5]
		\tikzstyle{every node} = [font = \small]
		\foreach \x in {0}
		{
			\foreach \y in {-8}
			{
			\fill (\x,\y+1.75) circle (.05);
			\fill (\x,\y+1.75) node [above] {\tiny{$1$}};
			\fill(\x+1.5158,\y+.875) circle (.05);
			\fill (\x+1.5158,\y+.875) node [right] {\tiny{$2 $}};
			\fill (\x+1.5158,\y-.875) circle (.05);
			\fill (\x+1.5158,\y-.875) node [right] {\tiny{$3$}};
			\fill(\x,\y-1.75) circle (.05);
			\fill (\x,\y-1.75) node [below] {\tiny{$4$}};
			\fill(\x-1.5158,\y-.875) circle (.05);
			\fill (\x-1.5158,\y-.875) node [left] {\tiny{$5$}};
			\fill (\x-1.5158,\y+.875) circle (.05);
			\fill (\x-1.5158,\y+.875) node [left] {\tiny{$6$}};
		
			\draw [] (\x,\y+1.75)--(\x+1.5158,\y+.875); 
			\draw [] (\x+1.5158,\y+.875) -- (\x+1.5158,\y-.875); 
			\draw [] (\x+1.5158,\y-.875)--(\x,\y-1.75); 
			\draw [] (\x,\y-1.75) -- (\x-1.5158,\y-.875); 
			\draw [] (\x-1.5158,\y-.875)--(\x-1.5158,\y+.875); 
			\draw [] (\x-1.5158,\y+.875) -- (\x,\y+1.75); 
		
			\draw [] (\x,\y+1.75) -- (\x-1.5158,\y-.875);
			\draw [] (\x,\y+1.75) -- (\x,\y-1.75);
			\draw [thick] (\x,\y-1.75) -- (\x+1.5158,\y+.875);
		
			}
		}
		\draw[<->] (3,-8) -- (4,-8);
	\end{tikzpicture}
	}
		\quad
	\subfigure{
	\begin{tikzpicture}[scale = .5]
		\tikzstyle{every node} = [font = \small]
		\foreach \x in {0}
		{
			\foreach \y in {-8}
			{
			\fill (\x,\y+1.75) circle (.05);
			\fill (\x,\y+1.75) node [above] {\tiny{$1$}};
			\fill(\x+1.5158,\y+.875) circle (.05);
			\fill (\x+1.5158,\y+.875) node [right] {\tiny{$2 $}};
			\fill (\x+1.5158,\y-.875) circle (.05);
			\fill (\x+1.5158,\y-.875) node [right] {\tiny{$3$}};
			\fill(\x,\y-1.75) circle (.05);
			\fill (\x,\y-1.75) node [below] {\tiny{$4$}};
			\fill(\x-1.5158,\y-.875) circle (.05);
			\fill (\x-1.5158,\y-.875) node [left] {\tiny{$5$}};
			\fill (\x-1.5158,\y+.875) circle (.05);
			\fill (\x-1.5158,\y+.875) node [left] {\tiny{$6$}};
		
			\draw [] (\x,\y+1.75)--(\x+1.5158,\y+.875); 
			\draw [] (\x+1.5158,\y+.875) -- (\x+1.5158,\y-.875); 
			\draw [] (\x+1.5158,\y-.875)--(\x,\y-1.75); 
			\draw [] (\x,\y-1.75) -- (\x-1.5158,\y-.875); 
			\draw [] (\x-1.5158,\y-.875)--(\x-1.5158,\y+.875); 
			\draw [] (\x-1.5158,\y+.875) -- (\x,\y+1.75); 
		
			\draw [] (\x,\y+1.75) -- (\x-1.5158,\y-.875);
			\draw [thick] (\x-1.5158,\y-.875) -- (\x+1.5158,\y+.875);
			\draw [] (\x,\y-1.75) -- (\x+1.5158,\y+.875);
		
			}
		}
	\end{tikzpicture}
	}
	\quad \quad \quad
	\subfigure{
	\begin{tikzpicture}[scale = .5]
		\tikzstyle{every node} = [font = \small]
		
		\draw[<->] (-3,-8) -- (-4,-8);
		
		\foreach \x in {0}
		{
			\foreach \y in {-8}
			{
			\fill (\x,\y+1.75) circle (.05);
			\fill (\x,\y+1.75) node [above] {\tiny{$1$}};
			\fill(\x+1.5158,\y+.875) circle (.05);
			\fill (\x+1.5158,\y+.875) node [right] {\tiny{$2 $}};
			\fill (\x+1.5158,\y-.875) circle (.05);
			\fill (\x+1.5158,\y-.875) node [right] {\tiny{$3$}};
			\fill(\x,\y-1.75) circle (.05);
			\fill (\x,\y-1.75) node [below] {\tiny{$4$}};
			\fill(\x-1.5158,\y-.875) circle (.05);
			\fill (\x-1.5158,\y-.875) node [left] {\tiny{$5$}};
			\fill (\x-1.5158,\y+.875) circle (.05);
			\fill (\x-1.5158,\y+.875) node [left] {\tiny{$6$}};
		
			\draw [] (\x,\y+1.75)--(\x+1.5158,\y+.875); 
			\draw [] (\x+1.5158,\y+.875) -- (\x+1.5158,\y-.875); 
			\draw [] (\x+1.5158,\y-.875)--(\x,\y-1.75); 
			\draw [] (\x,\y-1.75) -- (\x-1.5158,\y-.875); 
			\draw [] (\x-1.5158,\y-.875)--(\x-1.5158,\y+.875); 
			\draw [] (\x-1.5158,\y+.875) -- (\x,\y+1.75); 
		
			\draw [thick] (\x-1.5158,\y+.875) -- (\x+1.5158,\y+.875);
			\draw [] (\x-1.5158,\y-.875) -- (\x+1.5158,\y+.875);
			\draw [] (\x,\y-1.75) -- (\x+1.5158,\y+.875);
		
			}
		}
		\draw[<->] (3,-8) -- (4,-8);
	\end{tikzpicture}
	}
		\quad
	\subfigure{
	\begin{tikzpicture}[scale = .5]
		\tikzstyle{every node} = [font = \small]
		\foreach \x in {0}
		{
			\foreach \y in {-8}
			{
			\fill (\x,\y+1.75) circle (.05);
			\fill (\x,\y+1.75) node [above] {\tiny{$1$}};
			\fill(\x+1.5158,\y+.875) circle (.05);
			\fill (\x+1.5158,\y+.875) node [right] {\tiny{$2 $}};
			\fill (\x+1.5158,\y-.875) circle (.05);
			\fill (\x+1.5158,\y-.875) node [right] {\tiny{$3$}};
			\fill(\x,\y-1.75) circle (.05);
			\fill (\x,\y-1.75) node [below] {\tiny{$4$}};
			\fill(\x-1.5158,\y-.875) circle (.05);
			\fill (\x-1.5158,\y-.875) node [left] {\tiny{$5$}};
			\fill (\x-1.5158,\y+.875) circle (.05);
			\fill (\x-1.5158,\y+.875) node [left] {\tiny{$6$}};
		
			\draw [] (\x,\y+1.75)--(\x+1.5158,\y+.875); 
			\draw [] (\x+1.5158,\y+.875) -- (\x+1.5158,\y-.875); 
			\draw [] (\x+1.5158,\y-.875)--(\x,\y-1.75); 
			\draw [] (\x,\y-1.75) -- (\x-1.5158,\y-.875); 
			\draw [] (\x-1.5158,\y-.875)--(\x-1.5158,\y+.875); 
			\draw [] (\x-1.5158,\y+.875) -- (\x,\y+1.75); 
		
			\draw [] (\x-1.5158,\y+.875) -- (\x+1.5158,\y+.875);
			\draw [] (\x-1.5158,\y-.875) -- (\x+1.5158,\y+.875);
			\draw [thick] (\x-1.5158,\y-.875) -- (\x+1.5158,\y-.875);
		
			}
		}
	\end{tikzpicture}
	}

\caption{Sequence of flips.}
\label{seq_flips}
\end{figure}
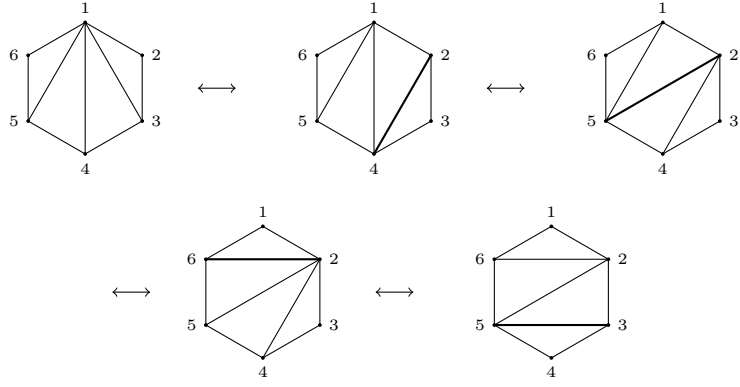

For the rest of this paper, we will be considering triangulations of the annulus (a region bounded by two concentric circles).

\begin{definition}
$C_{p,q}$ denotes the annulus with $p > 0$ points marked on the outer boundary component $\partial$, and $q > 0$ marked points on the inner boundary component $\partial'$. Without loss of generality, we assume that $p \ge q$.
\end{definition}

\begin{definition}
An arc in $C_{p,q}$ is called \emph{peripheral} if its two endpoints lie on the same boundary component. It is called \emph{bridging} otherwise.
\end{definition}

\begin{definition}
Let $T$ be a triangulation of $C_{p,q}$. A peripheral arc $\gamma \in T$ is called \emph{bounding} (with respect to $T$) if $\mu_\gamma$ is a bridging arc.
\label{def:bound}
\end{definition}

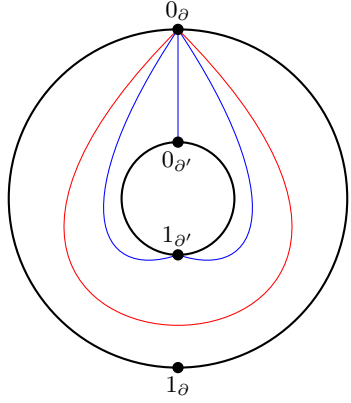
\begin{figure}
\centering
	\begin{tikzpicture}[scale = .75]
		\tikzstyle{every node} = [font = \small]
		\foreach \x in {0}
		{
			\draw[thick] (0,0) circle (3cm);
			\draw[thick] (0,0) circle (1cm);
				
			\draw [blue] (0,3) .. controls (2,0) and (1.5,-1.5) .. (0,-1);
			\draw [blue] (0,3) .. controls (-2,0) and (-1.5,-1.5) .. (0,-1);
			\draw [red] (0,3) .. controls (7,-4) and (-7,-4) .. (0,3);
			\draw[blue] (0,3)--(0,1);
			
			\fill (0,1) circle (.1);
			\fill (0,1) node [below] {$0_{\partial'}$};
			\fill(0,3) circle (.1);
			\fill (0,3) node [above] {$0_\partial$};
			\fill (0,-1) circle (.1);
			\fill (0,-1) node [above] {$1_{\partial'}$};
			\fill(0,-3) circle (.1);
			\fill (0,-3) node [below] {$1_{\partial}$};
			
		}
	\end{tikzpicture}
\caption{A triangulation of $C_{p,q}$. The bridging arcs are marked in blue and the peripheral arc is marked in red.}
\label{ann_triangulation}
\end{figure}

The following result appears in \cite[Lemma 1.7]{BD}. For convenience, we include a proof below.
\begin{lemma}
A triangulation $T$ of the annulus contains at least 2 bridging arcs.
\label{ann_bridge}
\end{lemma}

\begin{proof}
Let $T$ be a triangulation of $C_{p,q}$. Then there is at least one point on each boundary component which does not have a peripheral arc lying above it. The triangulation $T$ requires at least two bridging arcs connecting these two points.

\begin{center}
\begin{tikzpicture}[scale = .4]
	\foreach \x in {0}
	{	
		\draw[thick] (\x-3,0) -- (\x+4,0);
		\draw[thick](\x-3,3) -- (\x+4,3);
			
		\draw[red] (\x-2,0) .. controls (\x-1,1) and (\x+2,1) .. (\x+3,0);
		\fill (\x-2,0) circle (.1);
		\fill(\x+3,0) circle (.1);
		\fill(\x+.5,0) circle (.1);
		\fill (\x-2,0) node [below] {$i_\partial$};
		\fill (\x+3,0) node [below] {$i_\partial$};
		
		\fill (\x+.5,3) circle (.1);
		\fill (\x+.5,3) node [above] {$j_{\partial'}$};
		\draw[red] (\x+.5,3) .. controls (\x+1.5,2.5) and (\x+2,2.5) .. (\x+3.5,2.5);
		\draw[red] (\x+.5,3) .. controls (\x-.5,2.5) and (\x-1.5,2.5) .. (\x-2.5,2.5);
		
		\draw[dashed,blue] (\x+.5,3)--(\x+3,0);
		\draw[dashed,blue] (\x+.5,3)--(\x-2,0);
	}
	
	\end{tikzpicture}
	\end{center}
\end{proof}

\subsection{Universal cover of the annulus}

It is convenient to work with the universal cover of the annulus. We will use the notation and universal cover as described in \cite{BD}. Let $C_{p,q}$ be an annulus. We identify $C_{p,q}$ with a cylinder of height 1 with $p$ marked points on the lower boundary, which corresponds to $\partial$, and $q$ marked points on the upper boundary, which corresponds to $\partial'$. We denote this cylinder by $Cyl_{p,q}$, where the marked points are placed equidistant from one another on each boundary. We keep the orientation of the annulus, so the marked points are labeled left to right on the lower boundary by $(0,0), (q,0), \ldots, (pq-q, 0)$, and $(pq,0)$. On the upper boundary we label the marked points from right to left by $(0,1), (p,1), \ldots, (pq-p, 1)$, and $(pq,1)$ (see Figure \ref{tri_ann_cyl} for an example of a triangulation of $C_{2,2}$ drawn as a cylinder $Cyl_{2,2}$, and see Figure \ref{ann_cyl} for a general example of the marked cylinder $Cyl_{p,q}$).

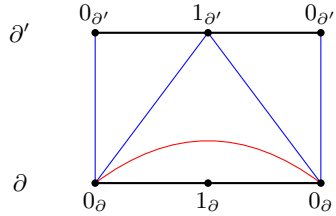
\begin{figure}[h!]
\centering
\begin{tikzpicture}[scale = .5]
		\tikzstyle{every node} = [font = \small]
		\foreach \x in {0}
		{
			\foreach \y in {-8}
			{
				\draw[thick] (\x-3,\y+4) -- (\x+3,\y+4);
				\fill (\x-5,\y+4) node {$\partial'$};
				\draw[thick] (\x-3,\y-0) -- (\x+3,\y-0);
				\fill (\x-5,\y-0) node {$\partial$};

				\draw[blue] (\x-3,\y+4) -- (\x-3,\y-0);
				\draw[blue] (\x-0,\y+4) -- (\x-3,\y-0);
				\draw[blue] (\x,\y+4) -- (\x+3,\y-0);
				\draw[blue] (\x+3,\y+4) -- (\x+3,\y-0);				
				\draw[red] (\x-3,\y-0) .. controls (\x-1,\y+1.5) and (\x+1,\y+1.5) .. (\x+3,\y-0);
		
				\foreach \t in {-3,0,3}
				{
					\fill (\x+\t,\y+4) circle (.1);
					\fill (\x+\t,\y-0) circle (.1);
				}
					
				\fill (\x-3,\y+4) node [above] {$0_{\partial'}$};
				\fill (\x-0,\y+4) node [above] {$1_{\partial'}$};
				\fill (\x+3,\y+4) node [above] {$0_{\partial'}$};

				\fill (\x-3,\y-0) node [below] {$0_{\partial}$};
				\fill (\x-0,\y-0) node [below] {$1_{\partial}$};
				\fill (\x+3,\y-0) node [below] {$0_{\partial}$};	
			}
		}
		\end{tikzpicture}
	\caption{A triangulation of $C_{2,2}$ represented as a cylinder.}
	\label{tri_ann_cyl}
	\end{figure}

\begin{figure}[h!]
	\begin{center}
	\begin{tikzpicture}[scale = .5]
		\tikzstyle{every node} = [font = \small]
		\foreach \x in {0}
		{
			\foreach \y in {-8}
			{
				\filldraw[fill=black!20!white, draw=black,dashed,thick] (\x-8,\y+2) rectangle (\x+8,\y-2);

				\foreach \t in {-8,-4,4,8}
				{
					\fill (\x+\t,\y+2) circle (.1);
				}
				
				\foreach \t in {-8,-5,-2,5,8}
				{
					\fill (\x+\t,\y-2) circle (.1);
				}

				\fill (\x-8,\y+2) node [above] {$(0,1)$};
				\fill (\x-4,\y+2) node [above] {$(p,1)$};
				\fill (\x-0,\y+2) node [above] {$\cdots$};
				\fill (\x+4,\y+2) node [above] {$(pq-p,1)$};
				\fill (\x+8,\y+2) node [above] {$(0,1)$};

				\fill (\x-8,\y-2) node [below] {$(0,0)$};
				\fill (\x-5,\y-2) node [below] {$(q,0)$};
				\fill (\x-2,\y-2) node [below] {$(2q,0)$};
				\fill (\x+1.5,\y-2) node [below] {$\cdots$};
				\fill (\x+5,\y-2) node [below] {$(pq-q,0)$};
				\fill (\x+8,\y-2) node [below] {$(0,0)$};
			}
		}
		\end{tikzpicture}
	\end{center}
\caption{Annulus as a cylinder.}
\label{ann_cyl}
\end{figure}
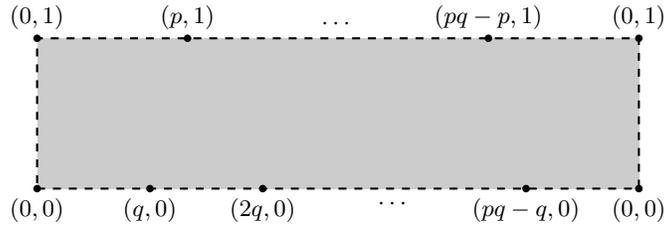

Now we define the universal cover $\mathbb{U} = (\mathbb{U},\pi_{pq})$ of the cylinder $Cyl_{p,q}$, with $\mathbb{U} = \{ (x,y) \in \mathbb{R} \mid 0 \leq y \leq 1 \}$ an infinite strip in the plane. The orientation is inherited from its embedding in $\mathbb{R}^2$. The covering map $\pi_{pq}: \mathbb{U} \rightarrow Cyl_{p,q}$ is induced from wrapping $\mathbb{U}$ around $Cyl_{p,q}$. For $(x,y) \in \mathbb{U}$, we have:
\[
\pi_{pq}(x,y) = (x \modm pq, y).
\]

The marked points on the lower boundary of $\mathbb{U}$ are $\{ (qx,0) \mid x \in \mathbb{Z}\}$, and the marked points on the upper  boundary are $\{(px,1) \mid x \in \mathbb{Z}\}$.

We denote the points on the lower boundary by $i_\partial$, so for $0 \leq i \leq p-1$:
\[
0_\partial := (0,0) = (pq,0), \,\,\,\, 1_\partial := (q,0), \,\,\,\, \ldots \,\,\,\,, (p-1)_\partial := (pq-q,0),
\]
and the points on the upper boundary are denoted by $j_{\partial'}$, so for $0 \leq j \leq q-1$:
\[
0_{\partial'} := (pq,1) = (0,1), \,\,\,\, 1_{\partial'} := (pq-p,1), \,\,\,\, \ldots \,\,\,\,, (p-1)_{\partial'} := (p,1).
\]
Figure \ref{ann_cyl_note} illustrates the cylinder $Cyl_{p,q}$.  

\begin{figure}[h!]
	\begin{center}
	\begin{tikzpicture}[scale = .5]
		\tikzstyle{every node} = [font = \small]
		\foreach \x in {0}
		{
			\foreach \y in {-8}
			{
				\filldraw[fill=black!20!white, draw=black,dashed,thick] (\x-8,\y+2) rectangle (\x+8,\y-2);

				\foreach \t in {-8,-4,4,8}
				{
					\fill (\x+\t,\y+2) circle (.1);
				}
				
				\foreach \t in {-8,-5,-2,5,8}
				{
					\fill (\x+\t,\y-2) circle (.1);
				}

				\fill (\x-8,\y+2) node [above] {$0_{\partial'}$};
				\fill (\x-4,\y+2) node [above] {$(q-1)_{\partial'}$};
				\fill (\x-0,\y+2) node [above] {$\cdots$};
				\fill (\x+4,\y+2) node [above] {$1_{\partial'}$};
				\fill (\x+8,\y+2) node [above] {$0_{\partial'}$};

				\fill (\x-8,\y-2) node [below] {$0_{\partial}$};
				\fill (\x-5,\y-2) node [below] {$1_{\partial}$};
				\fill (\x-2,\y-2) node [below] {$2_{\partial}$};
				\fill (\x+1.5,\y-2) node [below] {$\cdots$};
				\fill (\x+5,\y-2) node [below] {$(p-1)_{\partial}$};
				\fill (\x+8,\y-2) node [below] {$0_{\partial}$};
			}
		}
		\end{tikzpicture}
	\end{center}
\caption{Annulus as a cylinder with additional notation.}
\label{ann_cyl_note}
\end{figure}
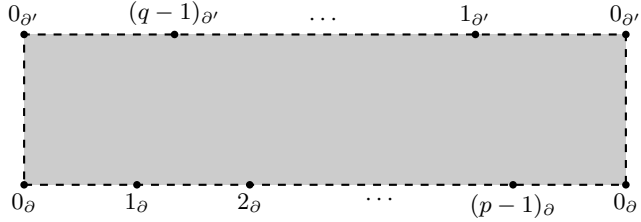

For clarity when working with the Coxeter transformation, we use integers to denote the marked points on the upper and lower boundary as follows:
\begin{align*}
\{i_\partial &= (iq, 0) \mid i \in \mathbb{Z}\},\\ 
\{j_{\partial'} &= (-jp, 1) \mid j \in \mathbb{Z}\},
\end{align*}
where the subscripts $\partial$ and $\partial'$ indicate on which boundary the point lies. This labeling is in the universal cover $\mathbb{U}$.

We refer to a lift of $Cyl_{p,q}$ in the universal cover as a \emph{frame}. Figure \ref{univ_cover} shows an example of the universal cover for $p = 3, q = 2$. The translates of such a frame cover $\mathbb{U}$. Figure \ref{univ_cover} shows $3$ frames.  

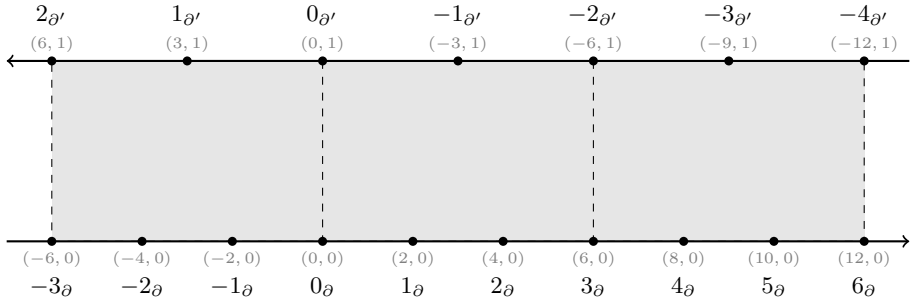
\begin{figure}[h!]
	\begin{center}
	\begin{tikzpicture}[scale = .6]
		\tikzstyle{every node} = [font = \small]
		\foreach \x in {0}
		{
			\foreach \y in {-8}
			{
				\filldraw[fill=black!10!white, draw=black,dashed] (\x-9,\y+2) rectangle (\x+9,\y-2);
				\draw[black,thick,<-] (\x-10,\y+2) -- (\x+10,\y+2);
				\draw[black,thick,->] (\x-10,\y-2) -- (\x+10,\y-2);
				\draw[dashed] (\x-3,\y+2) -- (\x-3,\y-2);
				\draw[dashed] (\x+3,\y+2) -- (\x+3,\y-2);

				\foreach \t in {-9,-7,-5,-3,-1,1,3,5,7,9}
				{
					\fill (\x+\t,\y-2) circle (.1);
				}
				
				\foreach \t in {-9,-6,-3,0,3,6,9}
				{
					\fill (\x+\t,\y+2) circle (.1);
				}

				\fill (\x-9,\y+2) node [above,gray] {\tiny{$(6,1)$}};
				\fill (\x-6,\y+2) node [above,gray] {\tiny{$(3,1)$}};
				\fill (\x-3,\y+2) node [above,gray] {\tiny{$(0,1)$}};
				\fill (\x+0,\y+2) node [above,gray] {\tiny{$(-3,1)$}};
				\fill (\x+3,\y+2) node [above,gray] {\tiny{$(-6,1)$}};
				\fill (\x+6,\y+2) node [above,gray] {\tiny{$(-9,1)$}};
				\fill (\x+9,\y+2) node [above,gray] {\tiny{$(-12,1)$}};
				\fill (\x-9,\y+2.6) node [above] {$2_{\partial'}$};
				\fill (\x-6,\y+2.6) node [above] {$1_{\partial'}$};
				\fill (\x-3,\y+2.6) node [above] {$0_{\partial'}$};
				\fill (\x+0,\y+2.6) node [above] {$-1_{\partial'}$};
				\fill (\x+3,\y+2.6) node [above] {$-2_{\partial'}$};
				\fill (\x+6,\y+2.6) node [above] {$-3_{\partial'}$};
				\fill (\x+9,\y+2.6) node [above] {$-4_{\partial'}$};

				\fill (\x-9,\y-2) node [below,gray] {\tiny{$(-6,0)$}};
				\fill (\x-7,\y-2) node [below,gray] {\tiny{$(-4,0)$}};
				\fill (\x-5,\y-2) node [below,gray] {\tiny{$(-2,0)$}};
				\fill (\x-3,\y-2) node [below,gray] {\tiny{$(0,0)$}};
				\fill (\x-1,\y-2) node [below,gray] {\tiny{$(2,0)$}};
				\fill (\x+1,\y-2) node [below,gray] {\tiny{$(4,0)$}};
				\fill (\x+3,\y-2) node [below,gray] {\tiny{$(6,0)$}};
				\fill (\x+5,\y-2) node [below,gray] {\tiny{$(8,0)$}};
				\fill (\x+7,\y-2) node [below,gray] {\tiny{$(10,0)$}};
				\fill (\x+9,\y-2) node [below,gray] {\tiny{$(12,0)$}};
				\fill (\x-9,\y-2.6) node [below] {$-3_{\partial}$};
				\fill (\x-7,\y-2.6) node [below] {$-2_{\partial}$};
				\fill (\x-5,\y-2.6) node [below] {$-1_{\partial}$};
				\fill (\x-3,\y-2.6) node [below] {$0_{\partial}$};
				\fill (\x-1,\y-2.6) node [below] {$1_{\partial}$};
				\fill (\x+1,\y-2.6) node [below] {$2_{\partial}$};
				\fill (\x+3,\y-2.6) node [below] {$3_{\partial}$};
				\fill (\x+5,\y-2.6) node [below] {$4_{\partial}$};
				\fill (\x+7,\y-2.6) node [below] {$5_{\partial}$};
				\fill (\x+9,\y-2.6) node [below] {$6_{\partial}$};
			}
		}
		\end{tikzpicture}
	\end{center}
\caption{Universal cover of the annulus $C_{3,2}$.}
\label{univ_cover}
\end{figure}

Throughout this paper, we will use the notation that is the most convenient in the context.

\subsection{Quivers}

A \emph{quiver $Q = (Q_0, Q_1)$} is an oriented graph. $Q_0$ denotes the set of vertices of $Q$, and $Q_1$ denotes the set of arrows between vertices. The right-hand side of Figure \ref{quiver_of_T} gives an example of a quiver. Given a triangulation $T$ of a surface $S$, we can associate a quiver to $T$.

\begin{definition}
The quiver $Q_T$ associated to a triangulation $T$ is obtained as follows:
\begin{enumerate}
\item The vertices of $Q_T$ correspond to the arcs in $T$, with vertex $i$ corresponding to the arc $d_i$,
\item There is an arrow from $i$ to $j$ in $Q_T$ if $d_i$ and $d_j$ in $T$ bound a common triangle, and $d_j$ is a clockwise rotation of $d_i$.
\end{enumerate}
\label{def:QT}
\end{definition}

For two arcs to bound a common triangle, they must have a common endpoint. Let  $d_i = [a,b]$ and $d_j = [b,c]$ be two arcs in a triangulation $T$ of $S$. Then $d_j$ is a clockwise rotation of $d_i$ if the endpoint of $d_i$ at the marked point $a$ can be rotated clockwise to the marked point $c$ of $S$ such that the new arc $d_i' = [b,c]$ is isotopic to $d_j$. An example of a quiver associated to a triangulation is illustrated in Figure \ref{quiver_of_T}.

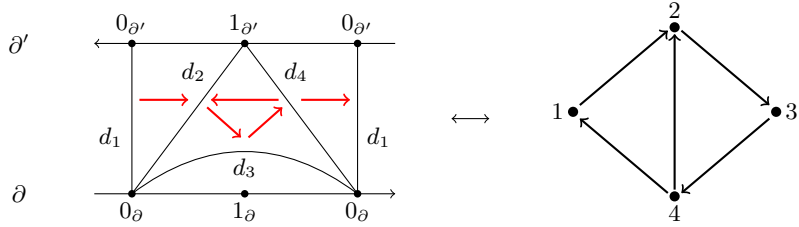
\begin{figure}[h!]
\centering
\subfigure{
\begin{tikzpicture}[scale = .5]
		\tikzstyle{every node} = [font = \small]
		\foreach \x in {0}
		{
			\foreach \y in {-8}
			{
				\draw[<-] (\x-4,\y+4) -- (\x+4,\y+4);
				\fill (\x-6,\y+4) node {$\partial'$};
				\draw[->] (\x-4,\y-0) -- (\x+4,\y-0);
				\fill (\x-6,\y-0) node {$\partial$};

				\foreach \t in {-3,0,3}
				{
					\fill (\x+\t,\y+4) circle (.1);
					\fill (\x+\t,\y-0) circle (.1);
				}
				
				\fill (\x-3,\y+4) node [above] {$0_{\partial'}$};
				\fill (\x-0,\y+4) node [above] {$1_{\partial'}$};
				\fill (\x+3,\y+4) node [above] {$0_{\partial'}$};

				\fill (\x-3,\y-0) node [below] {$0_{\partial}$};
				\fill (\x-0,\y-0) node [below] {$1_{\partial}$};
				\fill (\x+3,\y-0) node [below] {$0_{\partial}$};

				\fill[] (\x-3,\y+1.5) node [left] {$d_1$};
				\fill[] (\x-.8,\y+3.25) node [left] {$d_2$};
				\fill[] (\x-0,\y+1.2) node [below] {$d_3$};
				\fill[] (\x+.8,\y+3.25) node [right] {$d_4$};
				\fill[] (\x+3,\y+1.5) node [right] {$d_1$};

				\draw[] (\x-3,\y+4) -- (\x-3,\y-0);
				\draw[] (\x-0,\y+4) -- (\x-3,\y-0);
				\draw[] (\x,\y+4) -- (\x+3,\y-0);
				\draw[] (\x+3,\y+4) -- (\x+3,\y-0);				
				\draw (\x-3,\y-0) .. controls (\x-1,\y+1.5) and (\x+1,\y+1.5) .. (\x+3,\y-0);
				
				\draw [->,thick,red] (\x-2.8,\y+2.5)--(\x-1.5,\y+2.5);
				\draw [<-,thick,red] (\x+2.8,\y+2.5)--(\x+1.5,\y+2.5);
				
				\draw [<-,thick,red] (\x-.9,\y+2.5)--(\x+.9,\y+2.5);
				\draw [->,thick,red] (\x+.1,\y+1.5)--(\x+1,\y+2.3);
				\draw [<-,thick,red] (\x-.1,\y+1.5)--(\x-1,\y+2.3);
			}
			
			\draw[<->] (5.5,-6) -- (6.5,-6);
		}
		\end{tikzpicture}
}
	\quad
	\subfigure{
		\begin{tikzpicture}[scale = .45]
		\tikzstyle{every node} = [font = \small]
		\foreach \x in {0}
		{
			\foreach \y in {-8}
			{
			\fill (\x-3,\y+17) circle (.15);
			\fill(\x,\y+14.5) circle (.15);
			\fill(\x,\y+19.5) circle (.15);
			\fill(\x+3,\y+17) circle (.15);
			
			\fill[] (\x-3,\y+17) node [left] {$1$};
			\fill[] (\x,\y+19.5) node [above] {$2$};
			\fill[] (\x+3,\y+17) node [right] {$3$};
			\fill[] (\x,\y+14.5) node [below] {$4$};
		
			\draw [->,thick] (\x-2.8,\y+17.2)--(\x-.2,\y+19.4); 
			\draw [<-,thick] (\x-2.8,\y+16.8) -- (\x-.2,\y+14.6); 
			\draw [<-,thick] (\x+0.2,\y+14.6)--(\x+2.8,\y+16.8); 
			\draw [->,thick] (\x+0.2,\y+19.4) -- (\x+2.8,\y+17.2); 
			\draw [->,thick] (\x-0,\y+14.7) -- (\x-0,\y+19.3);

		}
		}
	\end{tikzpicture}
	}
	\caption{Triangulation $T$ and associated quiver $Q_T$.}
	\label{quiver_of_T}
	\end{figure}

Recall that we can perform flips of arcs in a triangulation. On the level of quivers, there exists a procedure called mutation of vertices. 

\begin{definition}
Let $Q$ be a quiver without loops or two-cycles. The mutation of a vertex $k \in Q_0$ is defined as follows:
\begin{enumerate}
\item For all paths of the form $i \overset{a}{\longrightarrow} k \overset{b}{\longrightarrow} j$, where $a,b$ denote the multiplicity of the arrows, add arrow $i \overset{ab}{\longrightarrow} j$ to $Q$.
\item Reverse all arrows incident with $k$.
\item Cancel a maximal number of two cycles created in $(1)$.
\end{enumerate}
\label{def:mutation}
\end{definition}

Mutation of a vertex $k$ will be denoted by $\sigma_k$.

\begin{example}
Let $Q$ be the following quiver, and consider mutation at vertex $3$. The blue arrow $4 \rightarrow 2$ is the arrow added in step (1) of Def. \ref{def:mutation}, since there is a path $4 \rightarrow 3 \rightarrow 2$. We then reverse the arrows incident to the vertex $3$ (Def. \ref{def:mutation}, step (2)). We cancel out the newly arising two-cycle (Def \ref{def:mutation}, step (3)):
\clearpage
\begin{figure}[h!]
\centering
\subfigure{
		\begin{tikzpicture}[scale = .45]
		\tikzstyle{every node} = [font = \small]
		\foreach \x in {0}
		{
			\foreach \y in {-8}
			{
			\fill (\x-3,\y+17) circle (.15);
			\fill(\x,\y+14.5) circle (.15);
			\fill(\x,\y+19.5) circle (.15);
			\fill(\x+3,\y+17) circle (.15);
			
			\fill[] (\x-3,\y+17) node [left] {$1$};
			\fill[] (\x,\y+19.5) node [above] {$2$};
			\fill[] (\x+3,\y+17) node [right] {$3$};
			\fill[] (\x,\y+14.5) node [below] {$4$};
		
			\draw [->,thick] (\x-2.8,\y+17.2)--(\x-.2,\y+19.4); 
			\draw [->,thick] (\x+0.2,\y+14.6)--(\x+2.8,\y+16.8); 
			\draw [<-,thick] (\x+0.2,\y+19.4) -- (\x+2.8,\y+17.2);
			\draw [<-,thick] (\x-0,\y+14.7) -- (\x-0,\y+19.3);

		\node[] at (\x+6,\y+17.2) {$\overset{(1)}{\longrightarrow}$};
		}
		}
	\end{tikzpicture}
	}
	\quad
	\subfigure{
		\begin{tikzpicture}[scale = .45]
		\tikzstyle{every node} = [font = \small]
		\foreach \x in {0}
		{
			\foreach \y in {-8}
			{
			
			\fill (\x-3,\y+17) circle (.15);
			\fill(\x,\y+14.5) circle (.15);
			\fill(\x,\y+19.5) circle (.15);
			\fill(\x+3,\y+17) circle (.15);
			
			\fill[] (\x-3,\y+17) node [left] {$1$};
			\fill[] (\x,\y+19.5) node [above] {$2$};
			\fill[] (\x+3,\y+17) node [right] {$3$};
			\fill[] (\x,\y+14.5) node [below] {$4$};
		
			\draw [->,thick] (\x-2.8,\y+17.2)--(\x-.2,\y+19.4); 
			\draw [->,thick] (\x+0.2,\y+14.6)--(\x+2.8,\y+16.8); 
			\draw [<-,thick] (\x+0.2,\y+19.4) -- (\x+2.8,\y+17.2);
			\draw [<-,thick] (\x-0.1,\y+14.7) -- (\x-0.1,\y+19.3);
			\draw [->,thick,blue] (\x+0.1,\y+14.7) -- (\x+0.1,\y+19.3);
			
			\node[] at (\x+6,\y+17.2)  {$\overset{(2)}{\longrightarrow}$};

		}
		
		}
	\end{tikzpicture}
}
	\quad
	\subfigure{
		\begin{tikzpicture}[scale = .45]
		\tikzstyle{every node} = [font = \small]
		\foreach \x in {0}
		{
			\foreach \y in {-8}
			{
			
						\fill (\x-3,\y+17) circle (.15);
			\fill(\x,\y+14.5) circle (.15);
			\fill(\x,\y+19.5) circle (.15);
			\fill(\x+3,\y+17) circle (.15);
			
			\fill[] (\x-3,\y+17) node [left] {$1$};
			\fill[] (\x,\y+19.5) node [above] {$2$};
			\fill[] (\x+3,\y+17) node [right] {$3$};
			\fill[] (\x,\y+14.5) node [below] {$4$};
		
			\draw [->,thick] (\x-2.8,\y+17.2)--(\x-.2,\y+19.4); 
			\draw [<-,thick, green] (\x+0.2,\y+14.6)--(\x+2.8,\y+16.8); 
			\draw [->,green, thick] (\x+0.2,\y+19.4) -- (\x+2.8,\y+17.2);
			\draw [<-,thick] (\x-0.1,\y+14.7) -- (\x-0.1,\y+19.3);
			\draw [->,thick, blue] (\x+0.1,\y+14.7) -- (\x+0.1,\y+19.3);
			
		\node[] at (\x+6,\y+17.2)  {$\overset{(3)}{\longrightarrow}$};
		}
		}
	\end{tikzpicture}
	}
	\quad
		\subfigure{
		\begin{tikzpicture}[scale = .45]
		\tikzstyle{every node} = [font = \small]
		\foreach \x in {0}
		{
			\foreach \y in {-8}
			{
			
			\fill (\x-3,\y+17) circle (.15);
			\fill(\x,\y+14.5) circle (.15);
			\fill(\x,\y+19.5) circle (.15);
			\fill(\x+3,\y+17) circle (.15);
			
			\fill[] (\x-3,\y+17) node [left] {$1$};
			\fill[] (\x,\y+19.5) node [above] {$2$};
			\fill[] (\x+3,\y+17) node [right] {$3'$};
			\fill[] (\x,\y+14.5) node [below] {$4$};
		
			\draw [->,thick] (\x-2.8,\y+17.2)--(\x-.2,\y+19.4); 
			\draw [<-,thick,green] (\x+0.2,\y+14.6)--(\x+2.8,\y+16.8); 
			\draw [->, green,thick] (\x+0.2,\y+19.4) -- (\x+2.8,\y+17.2);
			\draw [<-,thick] (\x-0.1,\y+14.7) -- (\x-0.1,\y+19.3);
			\draw [->, blue,thick] (\x+0.1,\y+14.7) -- (\x+0.1,\y+19.3);
			\draw[red,thick] (\x-.4,\y+16.5) -- (\x+.4,\y+17.3);

			\node[] at (\x+6,\y+17.2)  {$\overset{\sigma_3Q}{\Longrightarrow}$};
			}
		}
	\end{tikzpicture}
	}
	\quad
	\subfigure{
		\begin{tikzpicture}[scale = .45]
		\tikzstyle{every node} = [font = \small]
		\foreach \x in {0}
		{
			\foreach \y in {-8}
			{
			\fill (\x-3,\y+17) circle (.15);
			\fill(\x,\y+14.5) circle (.15);
			\fill(\x,\y+19.5) circle (.15);
			\fill(\x+3,\y+17) circle (.15);
			
			\fill[] (\x-3,\y+17) node [left] {$1$};
			\fill[] (\x,\y+19.5) node [above] {$2$};
			\fill[] (\x+3,\y+17) node [right] {$3'$};
			\fill[] (\x,\y+14.5) node [below] {$4$};
		
			\draw [->,thick] (\x-2.8,\y+17.2)--(\x-.2,\y+19.4); 
			\draw [<-,thick] (\x+0.2,\y+14.6)--(\x+2.8,\y+16.8); 
			\draw [->,thick] (\x+0.2,\y+19.4) -- (\x+2.8,\y+17.2);

			}
		}
		\end{tikzpicture}
		}
	\end{figure}
	
\end{example}

The flips of arcs in a triangulation $T$ and mutations of the associated quiver $Q_T$ correspond to each other. 

\section{Asymptotic triangulations}

In this section, we recall \emph{asymptotic triangulations}, which are defined by the presence of \emph{strictly asymptotic arcs}, and were first defined by Baur and Dupont in \cite{BD}. We will first define strictly asymptotic arcs, asymptotic triangulations, and flips of asymptotic arcs. To any asymptotic triangulation we can associate a quiver as in Definition \ref{def:QT}. Such a quiver may have loops and 2-cycles, and hence classical quiver mutation cannot be applied. In Subsection 3.1, we introduce a modified version of quivers of asymptotic triangulations in order to deal with this issue.

\medskip We denote by $z$ a non-contractible closed curve in the annulus.

\begin{definition}
Let $m$ be a marked point in $C_{p,q}$. Let $\pi_m$ be the isotopy class of the arc starting at $m$ and spiraling positively around the annulus. We call $\pi_m$ the \emph{Pr\"ufer arc} at $m$. Similarly, let $\alpha_m$ be the isotopy class of the arc starting at $m$ and spiraling negatively around the annulus. We call $\alpha_m$ the \emph{adic arc} at $m$ (cf. Figures \ref{fig:asymp_cpq} and \ref{fig:arcsCpq}).

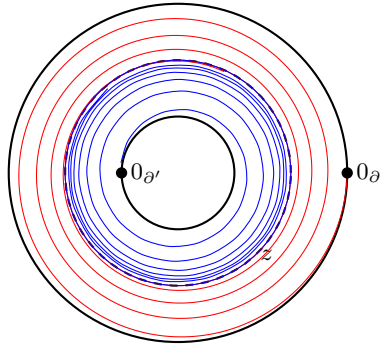
\begin{figure}[h!]
\centering
	\begin{tikzpicture}[scale = .75]
		\tikzstyle{every node} = [font = \small]
		\foreach \x in {0}
		{
			\draw[thick] (0,0) circle (3cm);
			\draw[thick] (0,0) circle (1cm);
			
			\fill (1.3,-1.45) node [right] {$z$};
			 \draw [thick,dashed] (0,0) circle [radius=2];
 			 \draw [red,domain=0:2.4,variable=\t,smooth,samples=400] plot ({-10*\t r}: {1+2*exp(-0.3*\t))});	
			\draw [red] (0,0) circle (2cm); 
			
			 \draw [blue,domain=2.4:12,variable=\t,smooth,samples=100] plot ({-4*\t r}: {(2-2*exp(-0.3*\t)))});	
			\draw [blue] (0,0) circle (2cm); 

			\fill (3,0) circle (.1);
			\fill (3,0) node [right] {$0_{\partial}$};
			\fill(-1,0) circle (.1);
			\fill (-1,0) node [right] {$0_{\partial'}$};

		}
	\end{tikzpicture}
	\caption{Adic arc $\pi_{0_\partial}$ in red, Pr\"ufer arc $\pi_{0_{\partial'}}$ in blue.}
	\label{fig:asymp_cpq}
\end{figure}
\end{definition}
	
\begin{figure}[h!]
	\begin{center}
	\begin{tikzpicture}[scale = .6]
		\tikzstyle{every node} = [font = \small]
		\foreach \x in {0}
		{
			\foreach \y in {-8}
			{
				\draw[<-] (\x-5,\y+4) -- (\x+5,\y+4);
				\fill (\x-6,\y+4) node {$\partial'$};
				\draw[->] (\x-5,\y-4) -- (\x+5,\y-4);
				\fill (\x-6,\y-4) node {$\partial$};

				\foreach \t in {-4,-2.5,-0.5,0.5,2.5,4}
				{
					\fill (\x+\t,\y+4) circle (.1);
					\fill (\x+\t,\y-4) circle (.1);
				}

				\fill (\x-4,\y+4) node [above] {$0_{\partial'}$};
				\fill (\x-2.25,\y+4) node [above] {$(q-1)_{\partial'}$};
				\fill (\x-.5,\y+4) node [above] {$\cdots$};
				\fill (\x+.5,\y+4) node [above] {$\cdots$};
				\fill (\x+2.5,\y+4) node [above] {$1_{\partial'}$};
				\fill (\x+4,\y+4) node [above] {$0_{\partial'}$};

				\fill (\x-4,\y-4) node [below] {$0_{\partial}$};
				\fill (\x-2.5,\y-4) node [below] {$1_{\partial}$};
				\fill (\x-.5,\y-4) node [below] {$\cdots$};
				\fill (\x+.5,\y-4) node [below] {$\cdots$};
				\fill (\x+2.5,\y-4) node [below] {$\tiny{(p-1)}_{\partial}$};
				\fill (\x+4,\y-4) node [below] {$0_{\partial}$};

				\hadic{\x+4}{\y+4}{1}{}
				\fill (\x,\y+1) node [above] {$\alpha_{0_{\partial'}}$};
				\hadic{\x-4}{\y+4}{9}{}

				\Prufer{\x-4}{\y-4}{9}{}
				\fill (\x,\y-1) node [below] {$\pi_{0_{\partial}}$};
				\Prufer{\x+4}{\y-4}{1}{}

				\draw[dashed,gray] (\x-4,\y+4) -- (\x-4,\y-4);
				\draw[dashed,gray] (\x+4,\y+4) -- (\x+4,\y-4);

				\draw (\x-2.5,\y+4) .. controls (\x-2,\y+3) and (\x+2,\y+3) .. (\x+2.5,\y+4);
				\draw (\x-2.5,\y-4) .. controls (\x-2,\y-3) and (\x+2,\y-3) .. (\x+2.5,\y-4);
			
			}
		}
		\end{tikzpicture}
	\end{center}
\caption{Asymptotic arcs in $C_{p,q}$.}
\label{fig:arcsCpq}
\end{figure}
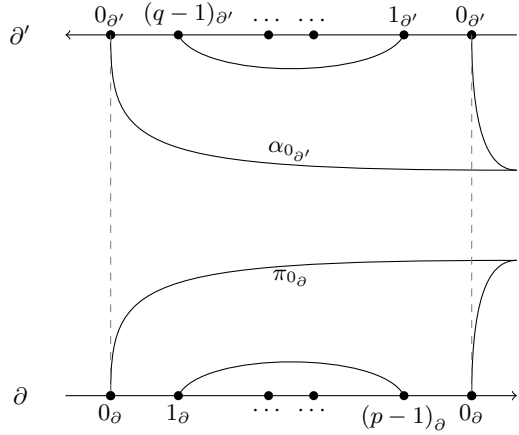

We call $\pi_m$ and $\alpha_m$ \emph{strictly asymptotic} arcs. We define the set of asymptotic arcs to be the union of the finite arcs and the strictly asymptotic arcs in a triangulation. Two arcs of $C_{p,q}$ are compatible if they do not intersect.

\begin{definition}
An \emph{asymptotic triangulation of the annulus} is a maximal collection of pairwise distinct and compatible asymptotic arcs, and contains strictly asymptotic arcs.
\end{definition}

Figure \ref{fig:asymp_cpq} shows two asymptotic arcs in the annulus, spiraling around $z$, and Figure \ref{fig:arcsCpq} shows examples of asymptotic arcs drawn in the cylinder $Cyl_{p,q}$.

\begin{definition}
Let $\beta \in \{\partial, \partial'\}$ be a boundary component. We say that an asymptotic arc is \emph{based at $\beta$} if it is either a peripheral arc with both endpoints on $\beta$, or it is a strictly asymptotic arc with its unique endpoint on $\beta$. A \emph{partial asymptotic triangulation} $T_\beta$ is the collection of arcs of an asymptotic triangulation based at the boundary component $\beta$. 
\end{definition}

The following result is from \cite{BD} (see article for proof).
 
\begin{lemma}
Let $T$ be a strictly asymptotic triangulation of $C_{p,q}$. Then $T$ contains at least two strictly asymptotic arcs, and there are two partial asymptotic triangulations $T_\partial, T_{\partial'}$ based at $\partial, \partial'$, respectively, such that $T = T_\partial \sqcup T_{\partial'}$.
\end{lemma}

Flips of asymptotic arcs are defined in the same way as flips of finite arcs, except we may consider quadrilaterals formed with strictly asymptotic arcs (cf. Fig. \ref{asymp_flip}). When there is only one strictly asymptotic arc in the partial asymptotic triangulation of a boundary component, based at a marked point $m$, then $\mu_{\pi_m} = \alpha_m$, and $\mu_{\alpha_m} = \pi_m$ (Fig. \ref{strict_asymp_flip}). This is because the strictly asymptotic arc $\alpha_m$ $(\pi_m)$ is the only arc compatible with $T \setminus \{\pi_m\}$ ($T \setminus \{\alpha_m\}$).

\begin{figure}[h!]
\begin{center}
\begin{tikzpicture}[scale = .4]
	\foreach \x in {0}
	{
		
		\draw (\x-3,0) -- (\x+4,0);
		\Prufer{\x-2}{0}{7}{gray}
		
		\draw[thick] (\x-2,0) .. controls (\x-1,1) and (\x+2,1) .. (\x+3,0);
		\fill (\x+1,0) circle (.1);
		
		\Prufer{\x+3}{0}{2}{gray}
		
		\fill (\x-2,0) node [below] {$i_\beta$};
		\fill (\x+1,0) node [below] {$m_\beta$};
		\fill (\x+3,0) node [below] {$j_\beta$};
		\draw[gray] (\x-2,0) .. controls (\x-1,.5) and (\x+0,.5) .. (\x+1,0);
		\draw[gray] (\x+1,0) .. controls (\x+1.5,.5) and (\x+2.5,.5) .. (\x+3,0);
	}
	
	\draw[<->] (8,1.5) -- (10,1.5);
	
	\foreach \x in {15}
	{
		\draw (\x-3,0) -- (\x+4,0);
		\Prufer{\x-2}{0}{6}{gray}
		\Prufer{\x+1}{0}{4}{thick}
		\Prufer{\x+3}{0}{2}{gray}
		\fill (\x-2,0) node [below] {$i_\beta$};
		\fill (\x+1,0) node [below] {$m_\beta$};
		\fill (\x+3,0) node [below] {$j_\beta$};
		\draw[gray] (\x-2,0) .. controls (\x-1,.5) and (\x+0,.5) .. (\x+1,0);
		\draw[gray] (\x+1,0) .. controls (\x+1.5,.5) and (\x+2.5,.5) .. (\x+3,0);
	}

\end{tikzpicture}
\end{center}
\caption{Flips of asymptotic arcs}
\label{asymp_flip}
\end{figure}

\begin{figure}[h!]
\begin{center}
\begin{tikzpicture}[scale = .4]
	\foreach \x in {0}
	{
		
		\draw (\x-3,0) -- (\x+4,0);
		\Prufer{\x-2}{0}{7}{thick}
		
		\draw[gray] (\x-2,0) .. controls (\x-1,1) and (\x+2,1) .. (\x+3,0);
		\fill (\x+1,0) circle (.1);
		
		\Prufer {\x+3}{0}{2}{thick}
		
		\fill (\x-2,0) node [below] {$m_\beta$};
		\fill (\x+1,0) node [below] {$i_\beta$};
		\fill (\x+3,0) node [below] {$m_\beta$};
		\draw[gray] (\x-2,0) .. controls (\x-1,.5) and (\x+0,.5) .. (\x+1,0);
		\draw[gray] (\x+1,0) .. controls (\x+1.5,.5) and (\x+2.5,.5) .. (\x+3,0);
	}
	
	\draw[<->] (7,1.5) -- (9,1.5);
	
	\foreach \x in {15}
	{
		\draw (\x-3,0) -- (\x+4,0);
		\adic{\x-2}{0}{2}{thick}
		\draw [gray] (\x-2,0) .. controls (\x-1,1) and (\x+2,1) .. (\x+3,0);
		\adic{\x+3}{0}{7}{thick}
		\fill (\x-2,0) node [below] {$m_\beta$};
		\fill (\x+1,0) node [below] {$i_\beta$};
		\fill (\x+3,0) node [below] {$m_\beta$};
		\draw[gray] (\x-2,0) .. controls (\x-1,.5) and (\x+0,.5) .. (\x+1,0);
		\draw[gray] (\x+1,0) .. controls (\x+1.5,.5) and (\x+2.5,.5) .. (\x+3,0);
	}

\end{tikzpicture}
\end{center}
\caption{Flips of asymptotic arcs}
\label{strict_asymp_flip}
\end{figure}

We can extend Definition \ref{def:bound} to the context of asymptotic arcs.
 
\begin{definition}
A \emph{bounding arc} $\gamma$ is a finite arc in an asymptotic triangulation such that the flipped arc $\mu_\gamma$ is a strictly asymptotic arc. 

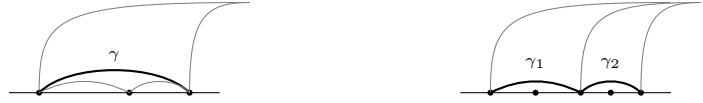
\begin{figure}[h!]
\begin{center}
\begin{tikzpicture}[scale = .4]
	\foreach \x in {0}
	{
		
		\draw (\x-3,0) -- (\x+4,0);
		\Prufer{\x-2}{0}{7}{gray}
		
		\draw[thick] (\x-2,0) .. controls (\x-1,1) and (\x+2,1) .. (\x+3,0);
		\fill (\x+1,0) circle (.1);
		\Prufer{\x+3}{0}{2}{gray}
		
		\fill (\x+.5,.75) node [above] {\tiny{$\gamma$}};
		\draw[gray] (\x-2,0) .. controls (\x-1,.5) and (\x+0,.5) .. (\x+1,0);
		\draw[gray] (\x+1,0) .. controls (\x+1.5,.5) and (\x+2.5,.5) .. (\x+3,0);
	}
	
		\foreach \x in {15}
	{

		\fill (\x-.5,0) circle (.1);
		\fill (\x+2,0) circle (.1);
		\draw (\x-3,0) -- (\x+4,0);
		\Prufer{\x-2}{0}{6}{gray}
		\Prufer{\x+1}{0}{4}{gray}
		\Prufer{\x+3}{0}{2}{gray}
		\fill (\x-.5,0.5) node [above] {\tiny{$\gamma_1$}};
		\fill (\x+2,0.5) node [above] {\tiny{$\gamma_2$}};
		\draw[thick] (\x-2,0) .. controls (\x-1,.5) and (\x+0,.5) .. (\x+1,0);
		\draw[thick] (\x+1,0) .. controls (\x+1.5,.5) and (\x+2.5,.5) .. (\x+3,0);
	}

\end{tikzpicture}
\end{center}
\caption{Bounding arcs $\gamma, \gamma_1, \gamma_2.$}
\label{bounding}
\end{figure}
\end{definition}

Figure \ref{bounding} shows three examples of bounding arcs. Bounding arcs are the finite arcs ``closest" to the asymptotic arcs in a triangulation. We call them bounding arcs because they separate all other non-bounding finite arcs from the strictly asymptotic arcs in the asymptotic triangulation.

Now just as in the finite case, given an asymptotic triangulation $T$, we can associate a quiver to $T$ (cf. Def. \ref{def:QT}).


\subsection{Quivers of asymptotic triangulations}
It is easy to see that the quiver associated to an asymptotic triangulation always has two connected components. It contains a quiver $Q_\partial$ corresponding to the triangulation based on the outer boundary $T_\partial$, and the quiver $Q_{\partial'}$ corresponding to the triangulation based on the inner boundary $T_{\partial'}$, and $Q_T =  Q_{\partial} \sqcup  Q_{\partial'}$.

Quiver mutation as in Definition \ref{def:mutation} only works for loop-free quivers without 2-cycles. If we associate quivers to asymptotic triangulations as in Definition \ref{def:QT} then loops and 2-cycles may appear. In order to define quiver mutation in this set-up, we need to modify the definition of a quiver associated to an asymptotic triangulation.

Recall that a frame of $T_\beta$ is one lift of $T_\beta$ in the universal cover. We start by choosing a frame of $T_\beta$ for $\beta \in \{\partial, \partial'\}$ with two copies of a strictly asymptotic arc as the end arcs of the frame. If we need to specify, we refer to this as a $d_i$ frame, with $d_i$ the \emph{framing arc}, and denote it by $T_\beta(d_i)$.

\begin{figure}[h!]
\begin{center}
\begin{tikzpicture}[scale = .4]
	\foreach \x in {0}
	{
		
		\draw (\x-3,0) -- (\x+4,0);
		\Prufer{\x-2}{0}{7}{thick}
		
		\Prufer {\x+3}{0}{2}{thick}
		\draw (\x-1.8,0.3)-- (\x+0,2.5);
		\draw (\x+1.2,0.3)-- (\x+3,2.5);
		\draw (\x-0.3,0.3)-- (\x+1.5,2.5);
		
		\fill (\x-2,1.25) node [left] {\tiny{$d_i$}};
		\fill (\x+3,1.25) node [right] {\tiny{$d_i$}};
	}

\end{tikzpicture}
\end{center}
\end{figure}

In the frame of $T_\beta$, each arc gives rise to a vertex in $Q_\beta$, and in particular, each copy of $d_i$ gives rise to separate vertex in $Q_\beta$. We denote the quiver corresponding to $T_\beta(d_i)$ by $Q_\beta(i)$.

\begin{figure}[h!]
\begin{center}
\begin{tikzpicture}[scale = .5]
	\foreach \x in {0}
	{
		\fill (\x+4,-0) circle (.1);
		\fill (\x+4,-0.2) node [below] {$i$};
		\fill(\x-3,-0) circle (.1); 
		\fill (\x-3,-0.2) node [below] {$i$};
		
		\draw[dashed,thick] (\x-2.5,0) -- (\x+3.5,0);
	}

\end{tikzpicture}
\end{center}
\end{figure}

We call these two $i$ vertices \emph{framing vertices}. These framing vertices do not get mutated during the quiver mutation. However, we don't consider them to be frozen because we allow arrows between framing vertices. If we want to mutate these vertices, we need to switch from our quiver $Q_\beta(i)$ to a new quiver $Q_\beta(j)$, for $d_j$ another strictly asymptotic arc in $T_\beta$. The corresponding operation in our triangulation is switching frames in the universal cover. We can go between a frame $T_\beta(d_i)$ and another frame $T_\beta(d_j)$ by shifting in one direction in our universal cover until we reach another strictly asymptotic arc $d_j$, which we now choose to be our framing arc. If there is no other strictly asymptotic arc in $T_\beta$, then we cannot switch frames, and therefore we cannot mutate the vertex $i \in Q_\beta(i)$. Recall from Figure \ref{strict_asymp_flip} that when we only have one strictly asymptotic arc $\gamma$, then a flip will only reverse the orientation of $\gamma$ without affecting the quiver. Since $\sigma_\gamma Q_\beta = Q_\beta$ for $\gamma$ the only strictly asymptotic arc in $T_\beta$, we can use this  definition of quiver mutation for a framing quiver. 

If there is another strictly asymptotic arc in the associated triangulation $T_\beta(d_i)$, then we can mutate our frozen vertices by modifying our quiver $Q_\beta(i)$. Let $d_j$ be another strictly asymptotic arc in $T_\beta$. Then we can move between $Q_\beta(i)$ and $Q_\beta(j)$ by identifying the $i$ vertices in $Q_\beta(i)$, and then break the quiver at vertex $j$ so that our quiver now has two $j$ vertices. All arrows remain the same.

\begin{example} 
Let $T$ be the following asymptotic triangulation of $C_{2,2}$:
	\begin{center}
	\begin{tikzpicture}[scale = .45]
		\tikzstyle{every node} = [font = \small]
		\foreach \x in {0}
		{
			\foreach \y in {-8}
			{
				\draw[<-] (\x-5,\y+4) -- (\x+5,\y+4);
				\fill (\x-6,\y+4) node {$\partial'$};
				\draw[->] (\x-5,\y-4) -- (\x+5,\y-4);
				\fill (\x-6,\y-4) node {$\partial$};

				\fill (\x-4,\y+4) circle (.1);
				\fill(\x,\y+4) circle (.1);
				\fill (\x-4,\y-4) circle (.1);
				\fill(\x,\y-4) circle (.1);

				\fill (\x-4,\y+4) node [above] {$0_{\partial'}$};
				\fill (\x,\y+4) node [above] {$1_{\partial'}$};
				\fill (\x+4,\y+4) node [above] {$0_{\partial'}$};

				\fill (\x-4,\y-4) node [below] {$0_{\partial}$};
				\fill (\x,\y-4) node [below] {$1_{\partial}$};
				\fill (\x+4,\y-4) node [below] {$0_{\partial}$};

				\hadic{\x+4}{\y+4}{2}{}
				\fill (\x-3,\y+2) node [above] {$d_{1}$};
				\hadic{\x-4}{\y+4}{10}{}
				\fill (\x+5,\y+2) node [above] {$d_{1}$};
				\hadic{\x}{\y+4}{4}{}
				\fill (\x+1,\y+2) node [above] {$d_{2}$};
				
				\Prufer{\x-4}{\y-4}{10}{}
				\fill (\x-3,\y-2) node [below] {$d_{3}$};
				\Prufer{\x+4}{\y-4}{1}{}
				\fill (\x+5,\y-2) node [below] {$d_{3}$};

				\draw[dashed,gray] (\x-4,\y+4) -- (\x-4,\y-4);
				\draw[dashed,gray] (\x+4,\y+4) -- (\x+4,\y-4);

				\draw (\x-4,\y-4) .. controls (\x-2,\y-3) and (\x+2,\y-3) .. (\x+4,\y-4);
				\fill (\x,\y-2.7) node [] {$d_4$};
			}
		}
	\end{tikzpicture}
	\end{center}
\medskip
Then the quivers $Q_{\partial'}$ and $Q_\partial$ are:

\begin{figure}[h!]
\begin{center}
	\begin{tikzpicture}[scale = .5]
		\tikzstyle{every node} = [font = \small]
		\foreach \x in {0}
		{
			\foreach \y in {-8}
			{
			\fill (\x-4,\y+2) circle (.15);
			\fill (\x-4,\y+1.8) node [below] {$1 \,\,$};
			\fill(\x,\y+2) circle (.15);
			\fill (\x,\y+1.8) node [below] {$\,\, 2 $};
			\fill(\x+4,\y+2) circle (.15);
			\fill (\x+4,\y+1.8) node [below] {$\,\, 1 $};

			\fill (\x-4,\y-3) circle (.15);
			\fill (\x-4,\y-3.2) node [below] {$3$};
			\fill(\x,\y-3) circle (.15);
			\fill (\x,\y-3.2) node [below] {$ 4 $};
			\fill(\x+4,\y-3) circle (.15);
			\fill (\x+4,\y-3.2) node [below] {$ 3 $};
			\fill(\x-8, \y+2) node [left] {\large{$Q_{\partial'}:$}};
			\fill(\x-8, \y-3) node [left] {\large{$Q_{\partial}:$}};
			
			\draw[->] (\x-3.5,\y+2) -- (\x-.5,\y+2);
			\draw[->] (\x+.5,\y+2) -- (\x+3.5,\y+2);
			
			\draw [->] (\x-3.5, \y-3) -- (\x-.5, \y-3);
			\draw [->] (\x+.5, \y-3) -- (\x+3.5, \y-3);
			\draw [<-] (\x-3.7,\y-2.6) .. controls (\x-1.5,\y-1.5) and (\x+1.5,\y-1.5) .. (\x+3.7,\y-2.6);

			}
		}
	\end{tikzpicture}
	\end{center}
\end{figure} 
\end{example}

Now we can perform the classical quiver mutation as per Definition \ref{def:mutation}.

\begin{example}

Let $T_\partial$ be the following asymptotic triangulation based on boundary component $\partial$, and let $Q_\partial$ be the quiver associated to $T_\partial$. Consider what happens when we flip the arc $d_2$ in $T$.

\begin{figure}[h!]
\centering
		\subfigure{
		\begin{tikzpicture}[scale = .4]
		\tikzstyle{every node} = [font = \small]
		\foreach \x in {0}
		{
			\foreach \y in {-8}
			{
				\draw[->] (\x-8,\y-3) -- (\x+2,\y-3);
				\fill (\x-9,\y-3) node {$\partial$};

				\foreach \t in {-6,-4,-2,0}
				{
					\fill (\x+\t,\y-3) circle (.1);
				}

				\fill (\x-6,\y-3) node [below] {$0_{\partial}$};
				\fill (\x-4,\y-3) node [below] {$1_{\partial}$};
				\fill (\x-2,\y-3) node [below] {$2_{\partial}$};
				\fill (\x+0,\y-3) node [below] {$0_{\partial}$};
				
				\Prufer{\x-6}{\y-3}{8}{}
				\fill (\x-6,\y-1.5) node [right] {\tiny{$d_1$}};
				\Prufer{\x-4}{\y-3}{6}{}
				\fill (\x-4,\y-1.5) node [right] {\tiny{$d_2$}};
				\Prufer{\x-2}{\y-3}{4}{}
				\fill (\x-2,\y-1.5) node [right] {\tiny{$d_3$}};
				\Prufer{\x-0}{\y-3}{2}{}
				\fill (\x-0,\y-1.5) node [right] {\tiny{$d_1$}};			
			
			\node[left] at (\x-8,\y+0) {\textbf{$T_\partial$}:};
			\node[] at (\x+6,\y-1.5) {$\overset{\mu_{2}}{\longrightarrow}$};
			}
			
		}	
		\end{tikzpicture}	
		}
		\quad
		\subfigure{
		\begin{tikzpicture}[scale = .4]
		\tikzstyle{every node} = [font = \small]
		\foreach \x in {0}
		{
			\foreach \y in {-8}
			{
				\draw[->] (\x-8,\y-3) -- (\x+2,\y-3);
				\fill (\x-9,\y-3) node {$\partial$};

				\foreach \t in {-6,-4,-2,0}
				{
					\fill (\x+\t,\y-3) circle (.1);
				}

				\fill (\x-6,\y-3) node [below] {$0_{\partial}$};
				\fill (\x-4,\y-3) node [below] {$1_{\partial}$};
				\fill (\x-2,\y-3) node [below] {$2_{\partial}$};
				\fill (\x+0,\y-3) node [below] {$0_{\partial}$};
				
				\Prufer{\x-6}{\y-3}{8}{}
				\fill (\x-6,\y-1.5) node [right] {\tiny{$d_1$}};
				\draw[] (\x-6,\y-3) .. controls (\x-5,\y-2) and (\x-3,\y-2) .. (\x-2,\y-3);
				\fill (\x-4,\y-2.4) node [above] {\tiny{$d_2'$}};
				\Prufer{\x-2}{\y-3}{4}{}
				\fill (\x-2,\y-1.5) node [right] {\tiny{$d_3$}};
				\Prufer{\x-0}{\y-3}{2}{}
				\fill (\x-0,\y-1.5) node [right] {\tiny{$d_1$}};			
			}
		}	
		\end{tikzpicture}		
		}
\end{figure}

Then the corresponding quiver mutation is:

\begin{figure}[h!]
\centering
\subfigure{
		\begin{tikzpicture}[scale = .35]
		\tikzstyle{every node} = [font = \small]
		\foreach \x in {0}
		{
			\foreach \y in {0}
			{
			\fill (\x-6, \y) circle (.15);
			\fill (\x-2,\y) circle (.15);
			\fill (\x+2,\y) circle (.15);
			\fill (\x+6,\y) circle (.15);
			
			\fill[] (\x-6,\y-.2) node [below] {$1$};
			\fill[] (\x-2,\y-.2) node [below] {$2$};
			\fill[] (\x+2,\y-.2) node [below] {$3$};
			\fill[] (\x+6,\y-.2) node [below] {$1$};
			
			\draw[<-, thick] (\x-5.7,\y) -- (\x-2.3,\y); 
			\draw[<-, thick] (\x-1.7,\y) -- (\x+1.7,\y); 
			\draw[<-, thick] (\x+2.3,\y) -- (\x+5.7,\y); 
				
			\node[] at (\x+10,\y) {$\overset{\sigma_{2}}{\longrightarrow}$};
			}
					}
	\end{tikzpicture}
	}
	\quad
	\subfigure{
		\begin{tikzpicture}[scale = .35]
		\tikzstyle{every node} = [font = \small]
		\foreach \x in {0}
		{
			\foreach \y in {0}
			{
			
			\fill (\x-6, \y) circle (.15);
			\fill (\x-2,\y) circle (.15);
			\fill (\x+2,\y) circle (.15);
			\fill (\x+6,\y) circle (.15);
			
			\fill[] (\x-6,\y-.2) node [below] {$1$};
			\fill[] (\x-2,\y-.2) node [below] {$2'$};
			\fill[] (\x+2,\y-.2) node [below] {$3$};
			\fill[] (\x+6,\y-.2) node [below] {$1$};
			
			\draw[->, thick] (\x-5.7,\y) -- (\x-2.3,\y); 
			\draw[->, thick] (\x-1.7,\y) -- (\x+1.7,\y); 
			\draw[<-, thick] (\x+2.3,\y) -- (\x+5.7,\y); 
			\draw[<-,thick] (\x-5.8, \y+.3) .. controls (\x-3,\y+2) and (\x-1,\y+2) .. (\x+1.8,\y+.3);
			
		}
		}
		
	\end{tikzpicture}
	}
	\end{figure}
	
	\medskip
	And we have the quiver $\sigma_2Q$. Note that if we were to identify the framing vertices, we would have a two-cycle between $1 \leftrightarrows 3$ that would be canceled using the classical definition of quiver mutation. However, by drawing the quiver with two framing vertices, we keep these arrows and our quiver is the quiver associated to $\mu_2T$.
	
	\medskip
	Now consider the flip $\mu_3$ and the corresponding quiver mutation $\sigma_3$. We denote by $\tilde{\sigma}_i$ the \emph{premutation at vertex i}, that is, the process of applying the first two steps of quiver mutation at vertex $i$ (before canceling 2-cycles):
	
\begin{figure}[h!]		
\subfigure{
		\begin{tikzpicture}[scale = .4]
		\tikzstyle{every node} = [font = \small]
		\foreach \x in {0}
		{
			\foreach \y in {-8}
			{
				\draw[->] (\x-8,\y-3) -- (\x+2,\y-3);
				\fill (\x-9,\y-3) node {$\partial$};

				\foreach \t in {-6,-4,-2,0}
				{
					\fill (\x+\t,\y-3) circle (.1);
				}

				\fill (\x-6,\y-3) node [below] {$0_{\partial}$};
				\fill (\x-4,\y-3) node [below] {$1_{\partial}$};
				\fill (\x-2,\y-3) node [below] {$2_{\partial}$};
				\fill (\x+0,\y-3) node [below] {$0_{\partial}$};
				
				\Prufer{\x-6}{\y-3}{8}{}
				\fill (\x-6,\y-1.5) node [right] {\tiny{$d_1$}};
				\draw[] (\x-6,\y-3) .. controls (\x-5,\y-2) and (\x-3,\y-2) .. (\x-2,\y-3);
				\fill (\x-4,\y-2.4) node [above] {\tiny{$d_2'$}};
				\Prufer{\x-2}{\y-3}{4}{}
				\fill (\x-2,\y-1.5) node [right] {\tiny{$d_3$}};
				\Prufer{\x-0}{\y-3}{2}{}
				\fill (\x-0,\y-1.5) node [right] {\tiny{$d_1$}};			
			
			\node[] at (\x+6,\y-1.5) {$\overset{\mu_{3}}{\longrightarrow}$};
			}
		}	
		\end{tikzpicture}		
		}
		\quad
		\subfigure{
		\begin{tikzpicture}[scale = .4]
		\tikzstyle{every node} = [font = \small]
		\foreach \x in {0}
		{
			\foreach \y in {-8}
			{
				\draw[->] (\x-8,\y-3) -- (\x+2,\y-3);
				\fill (\x-9,\y-3) node {$\partial$};

				\foreach \t in {-6,-4,-2,0}
				{
					\fill (\x+\t,\y-3) circle (.1);
				}

				\fill (\x-6,\y-3) node [below] {$0_{\partial}$};
				\fill (\x-4,\y-3) node [below] {$1_{\partial}$};
				\fill (\x-2,\y-3) node [below] {$2_{\partial}$};
				\fill (\x+0,\y-3) node [below] {$0_{\partial}$};
				
				\Prufer{\x-6}{\y-3}{8}{}
				\fill (\x-6,\y-1.5) node [right] {\tiny{$d_1$}};
				\draw[] (\x-6,\y-3) .. controls (\x-5,\y-2) and (\x-3,\y-2) .. (\x-2,\y-3);
				\fill (\x-3.85,\y-1.95) node [below] {\tiny{$d_2'$}};
				\draw[] (\x-6,\y-3) .. controls (\x-4.75,\y-1.5) and (\x-2,\y-1.5) .. (\x-0,\y-3);
				\fill (\x-2,\y-1.5) node [right] {\tiny{$d_3'$}};
				\Prufer{\x-0}{\y-3}{2}{}
				\fill (\x-0,\y-1.5) node [right] {\tiny{$d_1$}};			
			}
		}	
		\end{tikzpicture}		
		}
\end{figure}

\begin{figure}[h!]
\centering
	\subfigure{
		\begin{tikzpicture}[scale = .35]
		\tikzstyle{every node} = [font = \small]
		\foreach \x in {0}
		{
		
			\foreach \y in {0}
			{
			
			\fill (\x-6, \y) circle (.15);
			\fill (\x-2,\y) circle (.15);
			\fill (\x+2,\y) circle (.15);
			\fill (\x+6,\y) circle (.15);
			
			\fill[] (\x-6,\y-.2) node [below] {$1$};
			\fill[] (\x-2,\y-.2) node [below] {$2'$};
			\fill[] (\x+2,\y-.2) node [below] {$3$};
			\fill[] (\x+6,\y-.2) node [below] {$1$};
			
			\draw[->, thick] (\x-5.7,\y) -- (\x-2.3,\y); 
			\draw[->, thick] (\x-1.7,\y) -- (\x+1.7,\y); 
			\draw[<-, thick] (\x+2.3,\y) -- (\x+5.7,\y); 
			\draw[<-,thick] (\x-5.8, \y+.3) .. controls (\x-3,\y+2) and (\x-1,\y+2) .. (\x+1.8,\y+.3);

		\node[] at (\x+10,\y-.05) {$\overset{\tilde{\sigma}_{3}}{\longrightarrow}$};
		}
		
		}
	\end{tikzpicture}
	}
		\quad
	\subfigure{
		\begin{tikzpicture}[scale = .35]
		\tikzstyle{every node} = [font = \small]
		\foreach \x in {0}
		{
			\foreach \y in {0}
			{
			
			\fill (\x-6, \y) circle (.15);
			\fill (\x-2,\y) circle (.15);
			\fill (\x+2,\y) circle (.15);
			\fill (\x+6,\y) circle (.15);
			
			\fill[] (\x-6,\y-.2) node [below] {$1$};
			\fill[] (\x-2,\y-.2) node [below] {$2'$};
			\fill[] (\x+2,\y-.2) node [below] {$3'$};
			\fill[] (\x+6,\y-.2) node [below] {$1$};
			
			\draw[->, thick] (\x-5.6,\y+.2) -- (\x-2.4,\y+.2); 
			\draw[<-, thick] (\x-1.7,\y) -- (\x+1.7,\y); 
			\draw[<-, thick] (\x-5.6,\y-.2) -- (\x-2.4,\y-.2); 
			\draw[->, thick] (\x+2.3,\y) -- (\x+5.7,\y); 
			\draw[->,thick] (\x-5.7, \y+.3) .. controls (\x-3,\y+2) and (\x-1,\y+2) .. (\x+1.8,\y+.3);
			\draw[<-,thick] (\x-6, \y+.4) .. controls (\x-3,\y+3) and (\x+2,\y+3) .. (\x+5.9,\y+.4);

		\node[] at (\x+10,\y) {$\overset{\sigma_{3}}{\longrightarrow}$};
		}		
		}
	\end{tikzpicture}
	}	
	\quad
	\subfigure{
		\begin{tikzpicture}[scale = .35]
		\tikzstyle{every node} = [font = \small]
		\foreach \x in {0}
		{
			\foreach \y in {0}
			{
			
			\fill (\x-6, \y) circle (.15);
			\fill (\x-2,\y) circle (.15);
			\fill (\x+2,\y) circle (.15);
			\fill (\x+6,\y) circle (.15);
			
			\fill[] (\x-6,\y-.2) node [below] {$1$};
			\fill[] (\x-2,\y-.2) node [below] {$2'$};
			\fill[] (\x+2,\y-.2) node [below] {$3'$};
			\fill[] (\x+6,\y-.2) node [below] {$1$};
			
			\draw[<-, thick] (\x-1.7,\y) -- (\x+1.7,\y); 
			\draw[->, thick] (\x+2.3,\y) -- (\x+5.7,\y); 
			\draw[->,thick] (\x-5.7, \y+.3) .. controls (\x-3,\y+2) and (\x-1,\y+2) .. (\x+1.8,\y+.3);
			\draw[<-,thick] (\x-6, \y+.4) .. controls (\x-3,\y+3) and (\x+2,\y+3) .. (\x+5.9,\y+.4);
		}	
		}
	\end{tikzpicture}
	}	
\end{figure}
	Our resulting quiver is $\sigma_3 \sigma_2 Q$.
	
	\medskip
	If we flip $d_3'$ in $T$, we get the previous triangulation back. The quiver mutation rules should also give us the previous quiver $\sigma_2Q$ back:

	\begin{figure}[h!]
	\centering
	\subfigure{
		\begin{tikzpicture}[scale = .35]
		\tikzstyle{every node} = [font = \small]
		\foreach \x in {0}
		{
		
			\foreach \y in {0}
			{
			
			\fill (\x-6, \y) circle (.15);
			\fill (\x-2,\y) circle (.15);
			\fill (\x+2,\y) circle (.15);
			\fill (\x+6,\y) circle (.15);
			
			\fill[] (\x-6,\y-.2) node [below] {$1$};
			\fill[] (\x-2,\y-.2) node [below] {$2'$};
			\fill[] (\x+2,\y-.2) node [below] {$3''$};
			\fill[] (\x+6,\y-.2) node [below] {$1$};
			
			\draw[<-, thick] (\x-1.7,\y) -- (\x+1.7,\y); 
			\draw[->, thick] (\x+2.3,\y) -- (\x+5.7,\y); 
			\draw[->,thick] (\x-5.7, \y+.3) .. controls (\x-3,\y+2) and (\x-1,\y+2) .. (\x+1.8,\y+.3);
			\draw[<-,thick] (\x-6, \y+.4) .. controls (\x-3,\y+3) and (\x+2,\y+3) .. (\x+5.9,\y+.4);

		\node[] at (\x+10,\y) {$\overset{\tilde{\sigma}_{3'}}{\longrightarrow}$};
		}
		
		}
	\end{tikzpicture}
	}	
	\quad
		\subfigure{
		\begin{tikzpicture}[scale = .35]
		\tikzstyle{every node} = [font = \small]
		\foreach \x in {0}
		{
			\foreach \y in {0}
			{
			
			\fill (\x-6, \y) circle (.15);
			\fill (\x-2,\y) circle (.15);
			\fill (\x+2,\y) circle (.15);
			\fill (\x+6,\y) circle (.15);
			
			\fill[] (\x-6,\y-.2) node [below] {$1$};
			\fill[] (\x-2,\y-.2) node [below] {$2'$};
			\fill[] (\x+2,\y-.2) node [below] {$3''$};
			\fill[] (\x+6,\y-.2) node [below] {$1$};
			
			\draw[->, thick] (\x-5.7,\y) -- (\x-2.3,\y); 
			\draw[->, thick] (\x-1.7,\y) -- (\x+1.7,\y); 
			\draw[<-, thick] (\x+2.3,\y) -- (\x+5.7,\y); 
			\draw[<-,thick] (\x-5.7, \y+.3) .. controls (\x-3,\y+2) and (\x-1,\y+2) .. (\x+1.8,\y+.3);
			\draw[<-,thick] (\x-6, \y+.4) .. controls (\x-3,\y+3) and (\x+2,\y+3) .. (\x+5.9,\y+.4);
			\draw[->,thick] (\x-6, \y+.7) .. controls (\x-3,\y+3.5) and (\x+2,\y+3.5) .. (\x+6,\y+.7);
			
			\node[] at (\x+10,\y) {$\overset{\sigma_{3'}}{\longrightarrow}$};
		}
			
		}
	\end{tikzpicture}
	}		
	\quad
	\subfigure{
		\begin{tikzpicture}[scale = .35]
		\tikzstyle{every node} = [font = \small]
		\foreach \x in {0}
		{
		
			\foreach \y in {0}
			{
			
			\fill (\x-6, \y) circle (.15);
			\fill (\x-2,\y) circle (.15);
			\fill (\x+2,\y) circle (.15);
			\fill (\x+6,\y) circle (.15);
			
			\fill[] (\x-6,\y-.2) node [below] {$1$};
			\fill[] (\x-2,\y-.2) node [below] {$2'$};
			\fill[] (\x+3,\y-.2) node [below] {$3''=3$};
			\fill[] (\x+6,\y-.2) node [below] {$1$};
			
			\draw[->, thick] (\x-5.7,\y) -- (\x-2.3,\y); 
			\draw[->, thick] (\x-1.7,\y) -- (\x+1.7,\y); 
			\draw[->, thick] (\x+2.3,\y) -- (\x+5.7,\y); 
			\draw[<-,thick] (\x-5.7, \y+.3) .. controls (\x-3,\y+2) and (\x-1,\y+2) .. (\x+1.8,\y+.3);		
		}
		}
	\end{tikzpicture}
	}
\end{figure}
\end{example}

\begin{proposition}
Flips of arcs in a frame $T_\beta(d_i)$ correspond to mutations of vertices in the associated framing quiver $Q_\beta(i)$.  
\end{proposition}

\begin{proof}
Let $T_\beta(d_i)$ be a $d_i$-frame of $T_\beta$, and let $Q_\beta(i)$ be the quiver associated to $T_\beta(d_i)$. Without loss of generality, we relabel the marked points of the frame from $0, \ldots, p-1, p$.

\begin{center}
\begin{tikzpicture}[scale = .45]
		\tikzstyle{every node} = [font = \small]
	\foreach \x in {0}
	{
		
		\draw (\x-3,0) -- (\x+4,0);
		\Prufer{\x-2}{0}{7}{}
		
		\fill (\x+1.5,0) circle (.1);
		\fill (\x-1,0) circle (.1);
		\fill (\x,0) circle (.1);
		
		\Prufer{\x+3}{0}{2}{}
		
		\fill (\x-2,0) node [below] {$0$};
		\fill (\x+1.5,0) node [below] {$p-1$};
		\fill (\x+3,0) node [below] {$p$};

	}

	\end{tikzpicture}
	\end{center}

and give a name to the ``point" where the two strictly asymptotic arcs meet:

\begin{center}
\begin{tikzpicture}[scale = .45]
		\tikzstyle{every node} = [font = \small]
	\foreach \x in {0}
	{
		\fill (\x+5,3) circle (.1) node [right] {$p+1$};
		
		\draw (\x-3,0) -- (\x+4,0);
		\Prufer{\x-2}{0}{7}{}
		
		\fill (\x+1.5,0) circle (.1);
		\fill (\x-1,0) circle (.1);
		\fill (\x,0) circle (.1);
		
		\Prufer{\x+3}{0}{2}{}
		
		\fill (\x-2,0) node [below] {$0$};
		\fill (\x+1.5,0) node [below] {$p-1$};
		\fill (\x+3,0) node [below] {$p$};

	}

	\end{tikzpicture}
	\end{center}
	
Then this is equivalent to a triangulated polygon on $p+2$ vertices with $p-1$ arcs.

\begin{center}
\begin{tikzpicture}[scale = 1.5]
		\tikzstyle{every node} = [font = \small]
		\foreach \x in {0}
		{
			\foreach \y in {0}
			{
			\draw (\x,\y) circle (1);
			\fill(\x,\y+1) circle (.05);
			\fill (\x+0,\y+1) node [above] {$p+1$};
			\fill(\x+.866,\y+0.5) circle (.05);
			\fill (\x+0.866,\y+.5) node [right] {$p$};
			\fill(\x-.866,\y+0.5) circle (.05);
			\fill (\x-0.866,\y+.5) node [left] {$0$};	
			\fill(\x-1,\y-0) circle (.05);
			\fill (\x-1,\y-0) node [left] {$1$};	
			\fill(\x+1,\y-0) circle (.05);
			\fill (\x+1,\y) node [right] {$p-1$};
			\fill (\x-0,\y-1) node [below] {$\ldots$};
			\fill (\x-0.7,\y+.816) node [above] {$d_i$};
			\fill (\x+0.7,\y+.816) node [above] {$d_i$};		
		
			}
		}
	\end{tikzpicture}
	\end{center}

Now it is known that flips of arcs in a triangulation of an unpunctured polygon correspond to mutation of vertices of the associated quiver. Therefore, any sequence of flips of arcs in $T_\beta(d_i)$ does the same thing as the corresponding sequence of mutations of vertices in the associated quiver $Q_\beta(i)$.
\end{proof}

An alternative way to mutate quivers associated to asymptotic triangulations is by using quivers with potentials (cf. Appendix A).

\section{Dehn Twist}

In the previous section we defined asymptotic triangulations. In this section we describe the process of going from a finite triangulation to an asymptotic triangulation. This process constitutes applying the Dehn twist infinitely many times to a triangulation. Applying the Dehn twist infinitely many times causes some arcs of the triangulation to become identified, while breaking other arcs into two parts so that we are left with two triangulations -- one based at each boundary component of our annulus.

\begin{definition}
Let $z$ be a non-contractible closed curve in $C_{p,q}$. Consider the homeomorphism of $C_{p,q}$ obtained by cutting $C_{p,q}$ along $z$ and gluing it back after rotating the inner boundary by $2\pi$. This homeomorphism is called a \emph{Dehn twist}.
We have chosen a counter-clockwise orientation of our surface, so when applying a positive Dehn twist, we rotate the inner boundary $\partial'$ of the annulus clockwise by $2\pi$ (Fig. \ref{dehn_twist}). A negative Dehn twist would be a rotation of $\partial'$ by $2\pi$ in the counter-clockwise direction.
\end{definition}

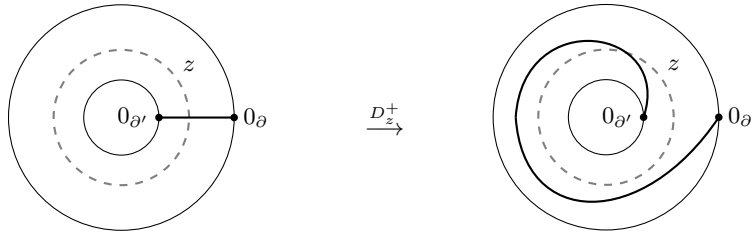
\begin{figure}[ht]
	\subfigure{
			\begin{tikzpicture}[scale = 1]
			\tikzstyle{every node} = [font = \small]
			\foreach \x in {0}
			{
				\foreach \y in {0}
				{
					\draw (\x,\y) circle (.5);
					\draw(\x, \y) circle (1.5);
					\draw[dashed, thick, gray] (\x, \y) circle (.9);
					\fill (\x+.9, \y+.5) node [above] {$z$};
					
					\draw[thick] (\x+.5,\y) -- (\x+1.5, \y);
					\fill (\x+.5, \y) circle (.05);
					\fill (\x+.5,\y) node [left] {$0_{\partial'}$};
					
					\fill (\x+1.5,\y) circle (.05);
					\fill (\x+1.5,\y) node [right] {$0_\partial$};
					
					\node[] at (\x+3.5,\y) {$\overset{D_z^+}{\longrightarrow}$};
			}
		}
		
		\end{tikzpicture}
		}
		\quad \quad
			\subfigure{
			\begin{tikzpicture}[scale = 1]
			\tikzstyle{every node} = [font = \small]
			\foreach \x in {0}
			{
				\foreach \y in {0}
				{
					\draw (\x,\y) circle (.5);
					\draw(\x, \y) circle (1.5);
					\draw[dashed, thick, gray] (\x, \y) circle (.9);
					\fill (\x+.9, \y+.5) node [above] {$z$};
					
					\draw[thick] (\x+1.5,\y) .. controls (\x+.5, \y-1.5) and (\x-1.1, \y-1.5) .. (\x-1.2, \y);
					\draw[thick] (\x-1.2,\y) .. controls (\x-1.1, \y+1.5) and (\x+.9, \y+1.2) .. (\x+.5, \y);

					\fill (\x+.5, \y) circle (.05);
					\fill (\x+.5,\y) node [left] {$0_{\partial'}$};
					
					\fill (\x+1.5,\y) circle (.05);
					\fill (\x+1.5,\y) node [right] {$0_\partial$};
			}
		}
		
		\end{tikzpicture}
		}
\caption{Dehn twist around closed curve $z$.}
\label{dehn_twist}
\end{figure}

\begin{notation}
$D_z^+$ denotes the \emph{positive Dehn twist} with respect to $z$, and $D_z^-$ denotes the \emph{negative Dehn twist} with respect to $z$.
\end{notation}

$D^n_z$ is the $n$-th Dehn twist (the Dehn twist applied $n$ times). We define: 
\[
D_z^{+\infty} = \lim_{n \to \infty} D^n_z,
\]
and 
\[
D_z^{-\infty} = \lim_{n \to -\infty} D^n_z.
\]

Let $\gamma = [i,j]$ be an arc in a triangulation $T$ of  $C_{p,q}$. If $i$ and $j$ lie on different boundary components, then:
\begin{align}
D_z^{+\infty} \cdot \gamma &= \{ \pi_{i},\pi_{j}\},
\label{eq:dehn_plus}
\end{align}
and
\begin{align}
D_z^{-\infty} \cdot \gamma &= \{ \alpha_{i},\alpha_{j}\}.
\label{eq:dehn_minus}
\end{align}
If $i$ and $j$ lie on the same boundary component, then $D^{\pm \infty}_z \cdot \gamma = \gamma$. 

Given a triangulation $T$ of the annulus $C_{p,q}$, we define
\[
D^{+\infty}_z(T) = \bigcup_{\gamma \in T} D_z^{+\infty} \cdot \gamma \mbox{\,\,\,\,\,\,and\,\,\,\,\,\,} D^{-\infty}_z(T) = \bigcup_{\gamma \in T} D_z^{-\infty} \cdot \gamma .
\]

Note that since $T$ contains at least two bridging arcs (Prop. \ref{ann_bridge}), $D_z^{\pm \infty}(T)$ is always asymptotic.

\begin{example}
Consider the following triangulation $T$ of $C_{p,q}$. Each application of the Dehn twist lengthens the bridging arcs of $T$. After infinitely many Dehn twists, the bridging arcs have infinite length and ``break" into Pr\"ufer arcs stemming from both boundary components.
\begin{figure}[h!]
\centering
\subfigure{
\begin{tikzpicture}[xscale = .175, yscale = .175]
					\foreach \x in {0}
					{
						\foreach \y in {0}
						{
							\coordinate (H1) at (\x-3,\y+4);
							\coordinate (H2) at (\x-1,\y+4);
							\coordinate (H3) at (\x+1,\y+4);
							\coordinate (H4) at (\x+3,\y+4);

							\draw[thick] (\x-5,\y+4) -- (\x+5,\y+4);
							\fill (H1) circle (.1);
							\fill (H2) circle (.1);
							\fill (H3) circle (.1);
							\fill (H4) circle (.1);

							\coordinate (B-1) at (\x-5,\y-4);
							\coordinate (B0) at (\x-4,\y-4);
							\coordinate (B1) at (\x-3,\y-4);
							\coordinate (B2) at (\x-2,\y-4);
							\coordinate (B3) at (\x-1,\y-4);
							\coordinate (B4) at (\x,\y-4);
							\coordinate (B5) at (\x+1,\y-4);
							\coordinate (B6) at (\x+2,\y-4);
							\coordinate (B7) at (\x+3,\y-4);

							\draw[thick] (\x-5,\y-4) -- (\x+5,\y-4);
							\fill (B1) circle (.1);
							\fill (B2) circle (.1);
							\fill (B3) circle (.1);
							\fill (B4) circle (.1);
							\fill (B5) circle (.1);
							\fill (B6) circle (.1);
							\fill (B7) circle (.1);
							
							\draw (H1) -- (B1) -- (H2) -- (B3) -- (H3) -- (B5) -- (H4) -- (B7);
							\draw (B1) .. controls (\x-2.5,\y-3) and (\x-1.5,\y-3) .. (B3);
							\draw (B3) .. controls (\x-.5,\y-3) and (\x+.5,\y-3) .. (B5);
							\draw (B5) .. controls (\x+1.5,\y-3) and (\x+2.5,\y-3) .. (B7);

							\fill (\x,\y-5) node {\tiny{$T$}};
						}
					}

					\draw[->] (8,-0) -- (10,0);
\end{tikzpicture}}
\quad
\subfigure{
\begin{tikzpicture}[xscale = .175, yscale = .175]

					\foreach \x in {0}
					{
						\foreach \y in {0}
						{
							\coordinate (H1) at (\x-3,\y+4);
							\coordinate (H2) at (\x-1,\y+4);
							\coordinate (H3) at (\x+1,\y+4);
							\coordinate (H4) at (\x+3,\y+4);
							\coordinate (H5) at (\x+5,\y+4);

							\draw[thick] (\x-5,\y+4) -- (\x+5,\y+4);
							\fill (H1) circle (.1);
							\fill (H2) circle (.1);
							\fill (H3) circle (.1);
							\fill (H4) circle (.1);

							\coordinate (B-1) at (\x-5,\y-4);
							\coordinate (B0) at (\x-4,\y-4);
							\coordinate (B1) at (\x-3,\y-4);
							\coordinate (B2) at (\x-2,\y-4);
							\coordinate (B3) at (\x-1,\y-4);
							\coordinate (B4) at (\x,\y-4);
							\coordinate (B5) at (\x+1,\y-4);
							\coordinate (B6) at (\x+2,\y-4);
							\coordinate (B7) at (\x+3,\y-4);

							\draw[thick] (\x-5,\y-4) -- (\x+5,\y-4);
							\fill (B1) circle (.1);
							\fill (B2) circle (.1);
							\fill (B3) circle (.1);
							\fill (B4) circle (.1);
							\fill (B5) circle (.1);
							\fill (B6) circle (.1);
							\fill (B7) circle (.1);
							
							\draw (H2) -- (B1) -- (H3) -- (B3) -- (H4) -- (B5) -- (H5) -- (B7);
							\draw (B1) .. controls (\x-2.5,\y-3) and (\x-1.5,\y-3) .. (B3);
							\draw (B3) .. controls (\x-.5,\y-3) and (\x+.5,\y-3) .. (B5);
							\draw (B5) .. controls (\x+1.5,\y-3) and (\x+2.5,\y-3) .. (B7);

							\fill (\x,\y-5) node {\tiny{$D_z (T)$}};
						}
					}
\draw[->] (8,-0) -- (10,0);

\end{tikzpicture}
}
\quad
\subfigure{
\begin{tikzpicture}[xscale = .175, yscale = .175]
\foreach \x in {0}
					{
						\foreach \y in {0}
						{
							\coordinate (H0) at (\x-5,\y+4);
							\coordinate (H1) at (\x-3,\y+4);
							\coordinate (H2) at (\x-1,\y+4);
							\coordinate (H3) at (\x+1,\y+4);
							\coordinate (H4) at (\x+3,\y+4);
							\coordinate (H5) at (\x+5,\y+4);

							\draw[thick] (\x-5,\y+4) -- (\x+5,\y+4);
							\fill (H1) circle (.1);
							\fill (H2) circle (.1);
							\fill (H3) circle (.1);
							\fill (H4) circle (.1);

							\coordinate (B-1) at (\x-5,\y-4);
							\coordinate (B0) at (\x-4,\y-4);
							\coordinate (B1) at (\x-3,\y-4);
							\coordinate (B2) at (\x-2,\y-4);
							\coordinate (B3) at (\x-1,\y-4);
							\coordinate (B4) at (\x,\y-4);
							\coordinate (B5) at (\x+1,\y-4);
							\coordinate (B6) at (\x+2,\y-4);
							\coordinate (B7) at (\x+3,\y-4);

							\draw[thick] (\x-5,\y-4) -- (\x+5,\y-4);
							\fill (B1) circle (.1);
							\fill (B2) circle (.1);
							\fill (B3) circle (.1);
							\fill (B4) circle (.1);
							\fill (B5) circle (.1);
							\fill (B6) circle (.1);
							\fill (B7) circle (.1);
							
							\draw (H2) -- (B-1) -- (H3) -- (B1) -- (H4) -- (B3) -- (H5) -- (B5);
							\draw (B-1) .. controls (\x-4.5,\y-3) and (\x-3.5,\y-3) .. (B1);
							\draw (B1) .. controls (\x-2.5,\y-3) and (\x-1.5,\y-3) .. (B3);
							\draw (B3) .. controls (\x-.5,\y-3) and (\x+.5,\y-3) .. (B5);

							\fill (\x,\y-5) node {\tiny{$D_z^2 (T)$}};
						}
					}
					
\draw[->] (7.5,-0) -- (11.5,0);

\end{tikzpicture}
}
\quad
\subfigure{
\begin{tikzpicture}[xscale = .175, yscale = .175]
	\foreach \x in {0}
					{
						\foreach \y in {0}
						{
							\coordinate (H0) at (\x-5,\y+4);
							\coordinate (H1) at (\x-3,\y+4);
							\coordinate (H2) at (\x-1,\y+4);
							\coordinate (H3) at (\x+1,\y+4);
							\coordinate (H4) at (\x+3,\y+4);

							\draw[thick] (\x-5,\y+4) -- (\x+5,\y+4);
							\fill (H1) circle (.1);
							\fill (H2) circle (.1);
							\fill (H3) circle (.1);
							\fill (H4) circle (.1);

							\coordinate (B-1) at (\x-5,\y-4);
							\coordinate (B0) at (\x-4,\y-4);
							\coordinate (B1) at (\x-3,\y-4);
							\coordinate (B2) at (\x-2,\y-4);
							\coordinate (B3) at (\x-1,\y-4);
							\coordinate (B4) at (\x,\y-4);
							\coordinate (B5) at (\x+1,\y-4);
							\coordinate (B6) at (\x+2,\y-4);
							\coordinate (B7) at (\x+3,\y-4);

							\draw[thick] (\x-5,\y-4) -- (\x+5,\y-4);
							\fill (B1) circle (.1);
							\fill (B2) circle (.1);
							\fill (B3) circle (.1);
							\fill (B4) circle (.1);
							\fill (B5) circle (.1);
							\fill (B6) circle (.1);
							\fill (B7) circle (.1);
							
							\hPrufer{\x-3}{\y+4}{3}{}
							\hPrufer{\x-1}{\y+4}{5}{}
							\hPrufer{\x+1}{\y+4}{7}{}
							\hPrufer{\x+3}{\y+4}{9}{}

							\Prufer{\x-3}{\y-4}{9}{}
							\Prufer{\x-1}{\y-4}{7}{}
							\Prufer{\x+1}{\y-4}{5}{}
							\Prufer{\x+3}{\y-4}{3}{}
							\draw (B1) .. controls (\x-2.5,\y-3) and (\x-1.5,\y-3) .. (B3);
							\draw (B3) .. controls (\x-.5,\y-3) and (\x+.5,\y-3) .. (B5);
							\draw (B5) .. controls (\x+1.5,\y-3) and (\x+2.5,\y-3) .. (B7);

							\fill (\x,\y-5) node {\tiny{$D_z^{+\infty} (T)$}};
						}
					}

\end{tikzpicture}
}
\end{figure}
\end{example}

From now on, we will only consider the positive Dehn twist $D^+_z$, but everything can be defined analogously for $D^-_z$. 

Notice that in Equations (\ref{eq:dehn_plus}) \& (\ref{eq:dehn_minus}), our arc $\gamma$ gives rise to two new arcs in the limit. However, we do not end up with twice as many arcs in the asymptotic triangulation. This is because all bridging arcs originating at the same boundary vertex become identified in the limit: Consider two bridging arcs $d_{[i,j]}, d_{[i,k]}$ in $T$ where $i$ lies on $\partial$ and $j,k$ lie on $\partial'$ (possibly, $j =k$). Then $D^+_z \cdot  d_{[i,j]} = \{ \pi_i, \pi_j\}$ and $D^+_z \cdot  d_{[i,j]} = \{ \pi_i, \pi_k\}$. So two of the vertices arising from $D^+_z \cdot d_{[i,j]}$ and $D^+_z \cdot d_{[i,k]}$ in $Q_T$ are one vertex $\pi_i$ in $Q_\partial$. This is illustrated in Figures \ref{one_to_two} \& \ref{two_to_one}. A finite triangulation of an annulus $C_{p,q}$ has $p+q$ arcs. The number of asymptotic arcs of an asymptotic triangulation is also $p+q$, which we can see by the decomposition $T = T_\partial \sqcup T_{\partial'}$ where $|T_\partial| = p$, $|T_{\partial'}| = q$.

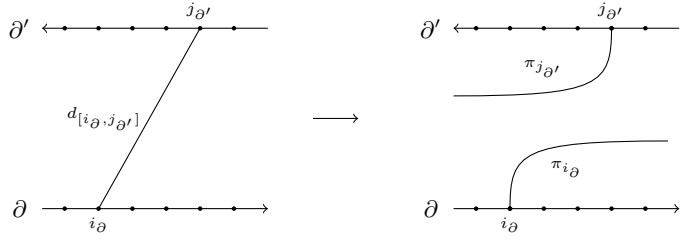
\begin{figure}[h!]
\centering
\subfigure{
	\begin{tikzpicture}[scale = .3]
		\tikzstyle{every node} = [font = \small]
		\foreach \x in {0}
		{
			\foreach \y in {-8}
			{
				\draw[<-] (\x-5,\y+4) -- (\x+5,\y+4);
				\fill (\x-6,\y+4) node {$\partial'$};
				\draw[->] (\x-5,\y-4) -- (\x+5,\y-4);
				\fill (\x-6,\y-4) node {$\partial$};

				\foreach \t in {-4,-2.5,-1,0.5,2,3.5}
				{
					\fill (\x+\t,\y+4) circle (.1);
					\fill (\x+\t,\y-4) circle (.1);
				}

				\draw (\x+2,\y+4) -- (\x-2.5,\y-4);
				\fill (\x+2,\y+4) node [above] {\tiny{$j_{\partial'}$}};
				\fill (\x-2.5,\y-4) node [below] {\tiny{$i_{\partial}$}};
				\fill (\x-2.25,\y) node [] {\tiny{$d_{[i_\partial,j_{\partial'}]}$}};
				
			}
		}
		
		\draw[->] (7,-8) -- (9,-8);
		\end{tikzpicture}
		}
		\quad
		\subfigure{
			\begin{tikzpicture}[scale = .3]
			\tikzstyle{every node} = [font = \small]
			\foreach \x in {0}
			{
				\foreach \y in {-8}
				{
					\draw[<-] (\x-5,\y+4) -- (\x+5,\y+4);
					\fill (\x-6,\y+4) node {$\partial'$};
					\draw[->] (\x-5,\y-4) -- (\x+5,\y-4);
					\fill (\x-6,\y-4) node {$\partial$};

					\foreach \t in {-4,-2.5,-1,0.5,2,3.5}
					{
						\fill (\x+\t,\y+4) circle (.1);
						\fill (\x+\t,\y-4) circle (.1);
					}

				\Prufer{\x-2.5}{\y-4}{7}{}
				\fill (\x-1,\y+3) node [below] {\tiny{$\pi_{j_{\partial'}}$}};
				\hPrufer{\x+2}{\y+4}{7}{}
				\fill (\x-0,\y-1.5) node [below] {\tiny{$\pi_{i_{\partial}}$}};

				\fill (\x+2,\y+4) node [above] {\tiny{$j_{\partial'}$}};
				\fill (\x-2.5,\y-4) node [below] {\tiny{$i_{\partial}$}};
			}
		}
		
		\end{tikzpicture}
		}
\caption{Bridging arc becoming two asymptotic arcs.}
\label{one_to_two}
\end{figure}

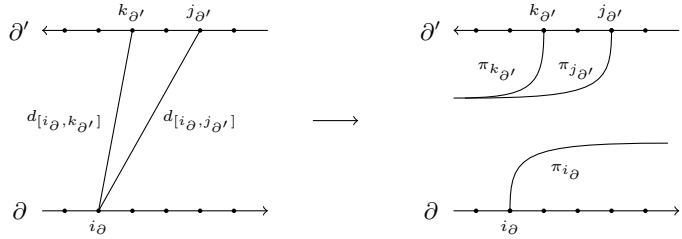
\begin{figure}[h!]
\centering
\subfigure{
	\begin{tikzpicture}[scale = .3]
		\tikzstyle{every node} = [font = \small]
		\foreach \x in {0}
		{
			\foreach \y in {-8}
			{
				\draw[<-] (\x-5,\y+4) -- (\x+5,\y+4);
				\fill (\x-6,\y+4) node {$\partial'$};
				\draw[->] (\x-5,\y-4) -- (\x+5,\y-4);
				\fill (\x-6,\y-4) node {$\partial$};

				\foreach \t in {-4,-2.5,-1,0.5,2,3.5}
				{
					\fill (\x+\t,\y+4) circle (.1);
					\fill (\x+\t,\y-4) circle (.1);
				}

				\draw (\x+2,\y+4) -- (\x-2.5,\y-4);
				\draw(\x-1,\y+4) -- (\x-2.5,\y-4);
				\fill (\x+2,\y+4) node [above] {\tiny{$j_{\partial'}$}};
				\fill (\x-2.5,\y-4) node [below] {\tiny{$i_{\partial}$}};
				\fill (\x-1,\y+4) node [above] {\tiny{$k_{\partial'}$}};
				\fill (\x-4,\y) node [] {\tiny{$d_{[i_\partial,k_{\partial'}]}$}};
				\fill (\x+2,\y) node [] {\tiny{$d_{[i_\partial,j_{\partial'}]}$}};
				
			}
		}
		
		\draw[->] (7,-8) -- (9,-8);
		\end{tikzpicture}
		}
		\quad
		\subfigure{
			\begin{tikzpicture}[scale = .3]
			\tikzstyle{every node} = [font = \small]
			\foreach \x in {0}
			{
				\foreach \y in {-8}
				{
					\draw[<-] (\x-5,\y+4) -- (\x+5,\y+4);
					\fill (\x-6,\y+4) node {$\partial'$};
					\draw[->] (\x-5,\y-4) -- (\x+5,\y-4);
					\fill (\x-6,\y-4) node {$\partial$};

					\foreach \t in {-4,-2.5,-1,0.5,2,3.5}
					{
						\fill (\x+\t,\y+4) circle (.1);
						\fill (\x+\t,\y-4) circle (.1);
					}

				\Prufer{\x-2.5}{\y-4}{7}{}
				\fill (\x+.5,\y+3) node [below] {\tiny{$\pi_{j_{\partial'}}$}};
				\hPrufer{\x+2}{\y+4}{6.5}{}
				\fill (\x-0,\y-1.5) node [below] {\tiny{$\pi_{i_{\partial}}$}};
				\hPrufer{\x-1}{\y+4}{4}{}
				\fill (\x-3,\y+3) node [below] {\tiny{$\pi_{k_{\partial'}}$}};

				\fill (\x+2,\y+4) node [above] {\tiny{$j_{\partial'}$}};
				\fill (\x-2.5,\y-4) node [below] {\tiny{$i_{\partial}$}};
				\fill (\x-1,\y+4) node [above] {\tiny{$k_{\partial'}$}};
			}
		}
		
		\end{tikzpicture}
		}
\caption{Two bridging arcs becoming one asymptotic arc.}
\label{two_to_one}
\end{figure}

\section{Coxeter transformations}

The Dehn twist provides us with a topological way of obtaining asymptotic triangulations. In this section, we describe a combinatorial method of obtaining asymptotic triangulations. This \emph{Coxeter transformation} is done by applying a sequence of flips to the arcs of the triangulation. On the level of quivers, we perform the sequence of corresponding mutations. The benefit of having a combinatorial method to describe this process is that we can now study other variables and systems associated to the surface. For example, we can look at root systems and (cluster) variables associated to arcs of the triangulation, and we expect this to provide a way to define cluster structures on asymptotic triangulations.

\subsection{Quivers}
Given a source $i$ of a quiver $Q$, the quiver $\sigma_iQ$ is obtained by reversing all arrows in $Q$ which start or end at $i$.

Recall from Definition \ref{def:admiss}, that an ordering $i_1, \ldots, i_n$ of the vertices of $Q$ is (source-) admissible if for each $p$ the vertex $i_p$ is a source in the quiver $\sigma_{i_{p-1}} \ldots \sigma_{i_1}Q$.
  
The following lemma is a well-known result from graph theory.
\begin{lemma}
There exists an admissible ordering of the vertices of $Q$ if and only if there are no oriented cycles in $Q$.
\end{lemma}

\begin{definition}
If ${\ivec}= i_1, \ldots, i_n$ is an admissible ordering on a quiver $Q$, then the Coxeter transformation of $Q$ is 
$$\cox_{\ivec}(Q) = \sigma_{i_n} \ldots \sigma_{i_1}Q.$$
\label{def:admissCoxQ}
\end{definition}

Recall from Section 1 that $\cox_{\ivec}(Q) \cong Q$, and note that $\cox_{\ivec}(Q)$ is independent of the chosen admissible ordering $\ivec$. Thus we will drop the index and just write $\cox(Q)$. 

As described in Definition \ref{def:QT}, we have quivers associated to triangulations.  We consider what happens to an associated triangulation when we mutate the arcs of the triangulation and the vertices of the associated quiver concurrently.

Let $T$ be a triangulation of $C_{p,q}$. Cycles in $Q_T$ occur when there are peripheral arcs in $T$. If there are any peripheral arcs in $T$, we can flip them until we obtain a triangulation $\tilde{T}$ consisting only of bridging arcs. For an example, see Figure \ref{fig:periph_mut_T}.

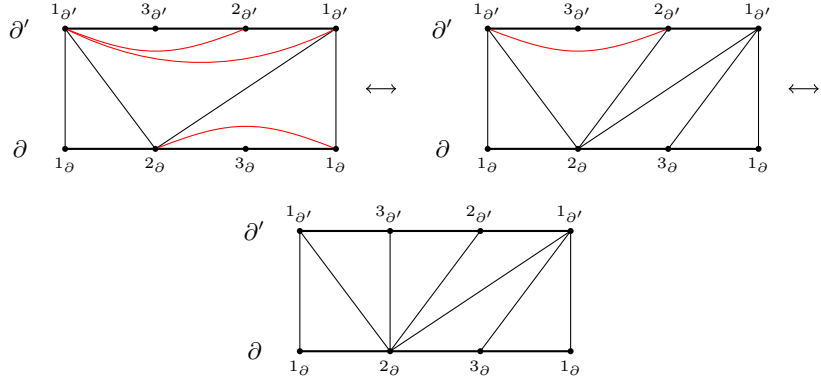
\begin{figure}[h!]
\centering
\subfigure{
\begin{tikzpicture}[scale = .4]
		\tikzstyle{every node} = [font = \tiny]
		\foreach \x in {0}
		{
			\foreach \y in {-8}
			{
				\draw[thick] (\x-4,\y+4) -- (\x+5,\y+4);
				\fill (\x-5.5,\y+4) node {\small{$\partial'$}};
				\draw[thick] (\x-4,\y-0) -- (\x+5,\y-0);
				\fill (\x-5.5,\y-0) node {\small{$\partial$}};

				\draw[] (\x-4,\y+4) -- (\x-4,\y-0);
				\draw[] (\x-4,\y+4) -- (\x-1,\y-0);
				\draw[] (\x+5,\y+4) -- (\x-1,\y-0);
				\draw[] (\x+5,\y+4) -- (\x+5,\y-0);			
					
				\draw[red] (\x-1,\y-0) .. controls (\x+1.5,\y+1) and (\x+2.5,\y+1) .. (\x+5,\y-0);
				\draw[red] (\x-4,\y+4) .. controls (\x-1,\y+2.5) and (\x+2,\y+2.5) .. (\x+5,\y+4);
				\draw[red] (\x-4,\y+4) .. controls (\x-1.5,\y+3) and (\x-.5,\y+3) .. (\x+2,\y+4);
				
				\foreach \t in {-4,-1,2,5}
				{
					\fill (\x+\t,\y+4) circle (.1);
					\fill (\x+\t,\y-0) circle (.1);
				}
					
				\fill (\x-4,\y+4) node [above] {$1_{\partial'}$};
				\fill (\x-1,\y+4) node [above] {$3_{\partial'}$};
				\fill (\x+2,\y+4) node [above] {$2_{\partial'}$};
				\fill (\x+5,\y+4) node [above] {$1_{\partial'}$};

				\fill (\x-4,\y-0) node [below] {$1_{\partial}$};
				\fill (\x-1,\y-0) node [below] {$2_{\partial}$};
				\fill (\x+2,\y-0) node [below] {$3_{\partial}$};
				\fill (\x+5,\y-0) node [below] {$1_{\partial}$};	
				
				\draw [<->] (\x+6,\y+2) -- (\x+7,\y+2);
			}
			
		}
		
		\end{tikzpicture}
		}
		\subfigure{
		\begin{tikzpicture}[scale = .4]
		\tikzstyle{every node} = [font = \tiny]
		\foreach \x in {0}
		{
			\foreach \y in {-8}
			{
				\draw[thick] (\x-4,\y+4) -- (\x+5,\y+4);
				\fill (\x-5.5,\y+4) node {\small{$\partial'$}};
				\draw[thick] (\x-4,\y-0) -- (\x+5,\y-0);
				\fill (\x-5.5,\y-0) node {\small{$\partial$}};

				\draw[] (\x-4,\y+4) -- (\x-4,\y-0);
				\draw[] (\x-4,\y+4) -- (\x-1,\y-0);
				\draw[] (\x+5,\y+4) -- (\x-1,\y-0);
				\draw[] (\x+5,\y+4) -- (\x+5,\y-0);			
				\draw[] (\x+5,\y+4) -- (\x+2,\y-0);
				\draw[] (\x+2,\y+4) -- (\x-1,\y-0);	
					
				\draw[red] (\x-4,\y+4) .. controls (\x-1.5,\y+3) and (\x-.5,\y+3) .. (\x+2,\y+4);
				
				\foreach \t in {-4,-1,2,5}
				{
					\fill (\x+\t,\y+4) circle (.1);
					\fill (\x+\t,\y-0) circle (.1);
				}
					
				\fill (\x-4,\y+4) node [above] {$1_{\partial'}$};
				\fill (\x-1,\y+4) node [above] {$3_{\partial'}$};
				\fill (\x+2,\y+4) node [above] {$2_{\partial'}$};
				\fill (\x+5,\y+4) node [above] {$1_{\partial'}$};

				\fill (\x-4,\y-0) node [below] {$1_{\partial}$};
				\fill (\x-1,\y-0) node [below] {$2_{\partial}$};
				\fill (\x+2,\y-0) node [below] {$3_{\partial}$};
				\fill (\x+5,\y-0) node [below] {$1_{\partial}$};	
				
				\draw [<->] (\x+6,\y+2) -- (\x+7,\y+2);
			}
			
		}
		
		\end{tikzpicture}
		}
		\quad
		\subfigure{
		\begin{tikzpicture}[scale = .4]
		\tikzstyle{every node} = [font = \tiny]
		\foreach \x in {0}
		{
			\foreach \y in {-8}
			{
				\draw[thick] (\x-4,\y+4) -- (\x+5,\y+4);
				\fill (\x-5.5,\y+4) node {\small{$\partial'$}};
				\draw[thick] (\x-4,\y-0) -- (\x+5,\y-0);
				\fill (\x-5.5,\y-0) node {\small{$\partial$}};

				\draw[] (\x-4,\y+4) -- (\x-4,\y-0);
				\draw[] (\x-4,\y+4) -- (\x-1,\y-0);
				\draw[] (\x+5,\y+4) -- (\x-1,\y-0);
				\draw[] (\x+5,\y+4) -- (\x+5,\y-0);			
				\draw[] (\x+5,\y+4) -- (\x+2,\y-0);
				\draw[] (\x+2,\y+4) -- (\x-1,\y-0);	
				\draw[] (\x-1,\y+4) -- (\x-1,\y-0);

				\foreach \t in {-4,-1,2,5}
				{
					\fill (\x+\t,\y+4) circle (.1);
					\fill (\x+\t,\y-0) circle (.1);
				}
					
				\fill (\x-4,\y+4) node [above] {$1_{\partial'}$};
				\fill (\x-1,\y+4) node [above] {$3_{\partial'}$};
				\fill (\x+2,\y+4) node [above] {$2_{\partial'}$};
				\fill (\x+5,\y+4) node [above] {$1_{\partial'}$};

				\fill (\x-4,\y-0) node [below] {$1_{\partial}$};
				\fill (\x-1,\y-0) node [below] {$2_{\partial}$};
				\fill (\x+2,\y-0) node [below] {$3_{\partial}$};
				\fill (\x+5,\y-0) node [below] {$1_{\partial}$};	
			}
			
		}
		\end{tikzpicture}}
\caption{Flipping the peripheral arcs of a triangulation $T$.}
\label{fig:periph_mut_T}
\end{figure}

We call such a triangulation $\tilde{T}$ a \emph{bridging triangulation}.

\begin{definition}
Let $\tilde{T}$ be a bridging triangulation. Let $Q_{\tilde{T}}$ be the associated quiver, and ${\ivec}= i_{1} \ldots i_{n}$ an admissible ordering of the vertices of $Q_{\tilde{T}}$. Then the \emph{Coxeter transformation} of $\tilde{T}$ is: 
\[
\cox_{{\ivec}}(\tilde{T}) = \mu_{d_{i_{n}}} \ldots \mu_{d_{i_{1}}}\tilde{T}.
\]
\end{definition}

\begin{lemma}
Let $\tilde{T}$ be a bridging triangulation and ${\ivec}= i_1 \ldots i_n$ an admissible ordering of $Q_{\tilde{T}}$. Then $\cox_{\ivec}(\tilde{T})$ moves endpoints of all arcs by -1 on both $\partial$ and $\partial'$.
\label{lemma:endpoints}
\end{lemma}

\begin{proof}
Let $\tilde{T}$ be a bridging triangulation, and $i_1 \ldots i_n$ an admissible ordering of the vertices of the associated quiver $Q_{\tilde{T}}$. Then for every $p = 2, \ldots, n$, vertex $i_p$ is a source in the quiver $\sigma_{i_{p-1}} \ldots \sigma_{i_1} Q$, and the corresponding arc $d_{i_p} = [j_\partial, k_{\partial'}]$ lies in such a quadrilateral in $\tilde{T}$:

\begin{center}
		\begin{tikzpicture}[scale = .4]
		\tikzstyle{every node} = [font = \tiny]
		\foreach \x in {0}
		{
			\foreach \y in {-8}
			{
				\draw[<-,thick] (\x-3.5,\y+4) -- (\x+3.5,\y+4);
				\draw[->,thick] (\x-3.5,\y-0) -- (\x+3.5,\y-0);

				\draw[thick] (\x-2,\y-0) -- (\x-2,\y+4);
				\draw[thick] (\x+2,\y-0) -- (\x-2,\y+4);
				\draw[thick] (\x+2,\y+0) -- (\x+2,\y+4);
				
				\foreach \t in {-2,2}
				{
					\fill (\x+\t,\y-0) circle (.1);
					\fill (\x+\t,\y+4) circle (.1);
				}
					
				\fill (\x-2,\y+4) node [above] {$k_{\partial'}$};
				\fill (\x+2,\y+4) node [above] {$(k-1)_{\partial'}$};
				\fill (\x-2,\y-0) node [below] {$(j-1)_{\partial}$};
				\fill (\x+2,\y-0) node [below] {$j_{\partial}$};	
				
				\fill (\x-0.25,\y+2.5) node [right] {$d_{i_p}$};	
		
			}
			
		}
\end{tikzpicture}
\end{center}

The flip corresponding to the mutation $\sigma_{i_p}$ is $\mu_{d_{i_p}}$. It replaces $d_{i_p}$ in $\mu_{i_{p-1}} \cdots \mu_{i_{1}}(T)$ with the other diagonal $d_{i_p}' = [(j-1)_\partial, (k-1)_{\partial'}]$ in the quadrilateral. This holds for all $1 \leq p \leq n$. Thus we have that the map $\cox_{\ivec}(\tilde{T})$ sends  $[j_\partial, k_{\partial'}]$ to  $ [(j-1)_\partial, (k-1)_{\partial'}]$. So the Coxeter transformation moves the endpoints of every arc in $\tilde{T}$ by $-1$ on each boundary component.
\end{proof}

The following corollary is a direct consequence of Lemma \ref{lemma:endpoints}:

\begin{corollary}
Let $\tilde{T}$ be a bridging triangulation with admissible ordering $\i$. Then $\cox_{\ivec}(\tilde{T})$ is independent of the choice of $\ivec$.
\end{corollary}

We will thus simply write $\cox(\tilde{T})$ for the Coxeter transformation of a bridging triangulation $\tilde{T}$.

\begin{example}
Let $T$ be the following triangulation of $C_{3,3}$ drawn in black, with labeled arcs. For convenience and clarity, we draw in a second copy of $T$ in gray.
\begin{figure}[h!]
	\begin{center}
	\begin{tikzpicture}[scale = .5]
		\tikzstyle{every node} = [font = \small]
		\foreach \x in {0}
		{
			\foreach \y in {-8}
			{
				\draw[<-] (\x-9,\y+4) -- (\x+9,\y+4);
				\fill (\x-10,\y+4) node {$\partial'$};
				\draw[->] (\x-9,\y-3) -- (\x+9,\y-3);
				\fill (\x-10,\y-3) node {$\partial$};

				\foreach \t in {-8,-6,-4,-2,0,2,4,6, 8}
				{
					\fill (\x+\t,\y+4) circle (.1);
					\fill (\x+\t,\y-3) circle (.1);
				}
				
				\fill (\x-8,\y+4) node [above] {$1_{\partial'}$};
				\fill (\x-6,\y+4) node [above] {$0_{\partial'}$};
				\fill (\x-4,\y+4) node [above] {$-1_{\partial'}$};
				\fill (\x-2,\y+4) node [above] {$-2_{\partial'}$};
				\fill (\x+0,\y+4) node [above] {$-3_{\partial'}$};
				\fill (\x+2,\y+4) node [above] {$-4_{\partial'}$};
				\fill (\x+4,\y+4) node [above] {$-5_{\partial'}$};
				\fill (\x+6,\y+4) node [above] {$-6_{\partial'}$};
				\fill (\x+8,\y+4) node [above] {$-7_{\partial'}$};

				\fill (\x-8,\y-3) node [below] {$-1_{\partial}$};
				\fill (\x-6,\y-3) node [below] {$0_{\partial}$};
				\fill (\x-4,\y-3) node [below] {$1_{\partial}$};
				\fill (\x-2,\y-3) node [below] {$2_{\partial}$};
				\fill (\x+0,\y-3) node [below] {$3_{\partial}$};
				\fill (\x+2,\y-3) node [below] {$4_{\partial}$};
				\fill (\x+4,\y-3) node [below] {$5_{\partial}$};
				\fill (\x+6,\y-3) node [below] {$6_{\partial}$};
				\fill (\x+8,\y-3) node [below] {$7_{\partial}$};

				\fill[] (\x-6,\y-1.5) node [left] {$d_1$};
				\fill[] (\x-4.5,\y-1.5) node [left] {$d_2$};
				\fill[] (\x-4,\y+2) node [left] {$d_3$};
				\fill[] (\x-2.5,\y+2) node [left] {$d_4$};
				\fill[] (\x-1,\y+2) node [left] {$d_5$};
				\fill[] (\x-1.5,\y-1.5) node [left] {$d_6$};
				\fill[] (\x-0,\y-1.5) node [left] {$d_1$};
				
				\fill[gray] (\x+1.5,\y-1.5) node [left] {$d_2$};
				\fill[gray] (\x+2,\y+2) node [left] {$d_3$};
				\fill[gray] (\x+3.5,\y+2) node [left] {$d_4$};
				\fill[gray] (\x+5,\y+2) node [left] {$d_5$};
				\fill[gray] (\x+4.5,\y-1.5) node [left] {$d_6$};
				\fill[gray] (\x+6,\y-1.5) node [left] {$d_1$};
				
				\draw[] (\x-6,\y+4) -- (\x-6,\y-3);
				\draw[] (\x-6,\y+4) -- (\x-4,\y-3);
				\draw[] (\x-4,\y+4) -- (\x-4,\y-3);
				\draw[] (\x-2,\y+4) -- (\x-4,\y-3);
				\draw[] (\x-0,\y+4) -- (\x-4,\y-3);
				\draw[] (\x-0,\y+4) -- (\x-2,\y-3);
				\draw[] (\x-0,\y+4) -- (\x-0,\y-3);
				
				\draw[gray] (\x-0,\y+4) -- (\x-0,\y-3);
				\draw[gray] (\x-0,\y+4) -- (\x+2,\y-3);
				\draw[gray] (\x+2,\y+4) -- (\x+2,\y-3);
				\draw[gray] (\x+4,\y+4) -- (\x+2,\y-3);
				\draw[gray] (\x+6,\y+4) -- (\x+2,\y-3);
				\draw[gray] (\x+6,\y+4) -- (\x+4,\y-3);
				\draw[gray] (\x+6,\y+4) -- (\x+6,\y-3);
				
			}
		}
		\end{tikzpicture}
	\end{center}
\end{figure}

The associated quiver $Q_T$:

\begin{figure}[h!]
\begin{center}
	\begin{tikzpicture}[scale = .35]
		\tikzstyle{every node} = [font = \small]
		\foreach \x in {0}
		{
			\foreach \y in {-8}
			{
			\fill (\x-6,\y) circle (.15);
			\fill (\x-6,\y) node [left] {$1 \,$};
			\fill (\x-4,\y-3) circle (.15);
			\fill (\x-4,\y-3) node [below] {$6 \,\,$};
			\fill(\x,\y-3) circle (.15);
			\fill (\x,\y-3) node [below] {$\,\, 5 $};
			\fill(\x+2,\y) circle (.15);
			\fill (\x+2,\y) node [right] {$\,4 $};
			\fill(\x,\y+3) circle (.15);
			\fill (\x,\y+3) node [above] {$\,\, 3 $};
			\fill (\x-4,\y+3) circle (.15);
			\fill (\x-4,\y+3) node [above] {$2\,\,$};
			
			\draw [->,thick] (\x-6,\y)--(\x-4,\y-2.9); 
			\draw [<-,thick] (\x-.1,\y-3) -- (\x-4,\y-3); 
			\draw [->,thick] (\x+2,\y)--(\x,\y-2.9); 
			\draw [->,thick] (\x,\y+2.9) -- (\x+2,\y); 
			\draw [->,thick] (\x-4,\y+3)--(\x-.1,\y+3); 
			\draw [<-,thick] (\x-6,\y+.1) -- (\x-4,\y+3); 

			}
		}
	\end{tikzpicture}
	\end{center}
	\end{figure} 
	
We use this quiver to obtain an admissible ordering of the vertices (and therefore also of arcs). Going from sources to sinks, we have an ordering $\ivec = 2,3,1,6,4,5$. So we will perform flips in the order $\mu_5\mu_4\mu_6\mu_1\mu_3\mu_2T$:

\begin{figure}[h!]
	\begin{center}
	\begin{tikzpicture}[scale = .5]
		\tikzstyle{every node} = [font = \small]
		\foreach \x in {0}
		{
			\foreach \y in {-8}
			{
				\draw[<-] (\x-9,\y+4) -- (\x+9,\y+4);
				\fill (\x-10,\y+4) node {$\partial'$};
				\draw[->] (\x-9,\y-3) -- (\x+9,\y-3);
				\fill (\x-10,\y-3) node {$\partial$};

				\foreach \t in {-8,-6,-4,-2,0,2,4,6, 8}
				{
					\fill (\x+\t,\y+4) circle (.1);
					\fill (\x+\t,\y-3) circle (.1);
				}
				
				\fill (\x-8,\y+4) node [above] {$1_{\partial'}$};
				\fill (\x-6,\y+4) node [above] {$0_{\partial'}$};
				\fill (\x-4,\y+4) node [above] {$-1_{\partial'}$};
				\fill (\x-2,\y+4) node [above] {$-2_{\partial'}$};
				\fill (\x+0,\y+4) node [above] {$-3_{\partial'}$};
				\fill (\x+2,\y+4) node [above] {$-4_{\partial'}$};
				\fill (\x+4,\y+4) node [above] {$-5_{\partial'}$};
				\fill (\x+6,\y+4) node [above] {$-6_{\partial'}$};
				\fill (\x+8,\y+4) node [above] {$-7_{\partial'}$};

				\fill (\x-8,\y-3) node [below] {$-1_{\partial}$};
				\fill (\x-6,\y-3) node [below] {$0_{\partial}$};
				\fill (\x-4,\y-3) node [below] {$1_{\partial}$};
				\fill (\x-2,\y-3) node [below] {$2_{\partial}$};
				\fill (\x+0,\y-3) node [below] {$3_{\partial}$};
				\fill (\x+2,\y-3) node [below] {$4_{\partial}$};
				\fill (\x+4,\y-3) node [below] {$5_{\partial}$};
				\fill (\x+6,\y-3) node [below] {$6_{\partial}$};
				\fill (\x+8,\y-3) node [below] {$7_{\partial}$};

				\fill[] (\x-6,\y-1.5) node [left] {$d_1$};
				\fill[green] (\x-4.5,\y+2) node [left] {$d'_2$};
				\fill[] (\x-4,\y-1.5) node [left] {$d_3$};
				\fill[] (\x-2.5,\y+2) node [left] {$d_4$};
				\fill[] (\x-1,\y+2) node [left] {$d_5$};
				\fill[] (\x-1.5,\y-1.5) node [left] {$d_6$};
				\fill[] (\x-0,\y-1.5) node [left] {$d_1$};
				
				\fill[green] (\x+1.5,\y+2) node [left] {$d'_2$};
				\fill[gray] (\x+2,\y-1.5) node [left] {$d_3$};
				\fill[gray] (\x+3.5,\y+2) node [left] {$d_4$};
				\fill[gray] (\x+5,\y+2) node [left] {$d_5$};
				\fill[gray] (\x+4.5,\y-1.5) node [left] {$d_6$};
				\fill[gray] (\x+6,\y-1.5) node [left] {$d_1$};
				
				\draw[] (\x-6,\y+4) -- (\x-6,\y-3);
				\draw[green,thick] (\x-4,\y+4) -- (\x-6,\y-3);
				\draw[] (\x-4,\y+4) -- (\x-4,\y-3);
				\draw[] (\x-2,\y+4) -- (\x-4,\y-3);
				\draw[] (\x-0,\y+4) -- (\x-4,\y-3);
				\draw[] (\x-0,\y+4) -- (\x-2,\y-3);
				\draw[] (\x-0,\y+4) -- (\x-0,\y-3);
				
				\draw[gray] (\x-0,\y+4) -- (\x-0,\y-3);
				\draw[green] (\x+2,\y+4) -- (\x,\y-3);
				\draw[gray] (\x+2,\y+4) -- (\x+2,\y-3);
				\draw[gray] (\x+4,\y+4) -- (\x+2,\y-3);
				\draw[gray] (\x+6,\y+4) -- (\x+2,\y-3);
				\draw[gray] (\x+6,\y+4) -- (\x+4,\y-3);
				\draw[gray] (\x+6,\y+4) -- (\x+6,\y-3);
				
			}
		}
		\end{tikzpicture}
	\end{center}
\end{figure}

\begin{figure}[h!]
	\begin{center}
	\begin{tikzpicture}[scale = .5]
		\tikzstyle{every node} = [font = \small]
		\foreach \x in {0}
		{
			\foreach \y in {-8}
			{
				\draw[<-] (\x-9,\y+4) -- (\x+9,\y+4);
				\fill (\x-10,\y+4) node {$\partial'$};
				\draw[->] (\x-9,\y-3) -- (\x+9,\y-3);
				\fill (\x-10,\y-3) node {$\partial$};

				\foreach \t in {-8,-6,-4,-2,0,2,4,6, 8}
				{
					\fill (\x+\t,\y+4) circle (.1);
					\fill (\x+\t,\y-3) circle (.1);
				}
				
				\fill (\x-8,\y+4) node [above] {$1_{\partial'}$};
				\fill (\x-6,\y+4) node [above] {$0_{\partial'}$};
				\fill (\x-4,\y+4) node [above] {$-1_{\partial'}$};
				\fill (\x-2,\y+4) node [above] {$-2_{\partial'}$};
				\fill (\x+0,\y+4) node [above] {$-3_{\partial'}$};
				\fill (\x+2,\y+4) node [above] {$-4_{\partial'}$};
				\fill (\x+4,\y+4) node [above] {$-5_{\partial'}$};
				\fill (\x+6,\y+4) node [above] {$-6_{\partial'}$};
				\fill (\x+8,\y+4) node [above] {$-7_{\partial'}$};

				\fill (\x-8,\y-3) node [below] {$-1_{\partial}$};
				\fill (\x-6,\y-3) node [below] {$0_{\partial}$};
				\fill (\x-4,\y-3) node [below] {$1_{\partial}$};
				\fill (\x-2,\y-3) node [below] {$2_{\partial}$};
				\fill (\x+0,\y-3) node [below] {$3_{\partial}$};
				\fill (\x+2,\y-3) node [below] {$4_{\partial}$};
				\fill (\x+4,\y-3) node [below] {$5_{\partial}$};
				\fill (\x+6,\y-3) node [below] {$6_{\partial}$};
				\fill (\x+8,\y-3) node [below] {$7_{\partial}$};

				\fill[] (\x-6,\y-1.5) node [left] {$d_1$};
				\fill[] (\x-4.5,\y+2) node [left] {$d'_2$};
				\fill[green] (\x-3,\y+2) node [left] {$d'_3$};
				\fill[] (\x-3.5,\y-1.5) node [left] {$d_4$};
				\fill[] (\x-1,\y+2) node [left] {$d_5$};
				\fill[] (\x-1.5,\y-1.5) node [left] {$d_6$};
				\fill[] (\x-0,\y-1.5) node [left] {$d_1$};
				
				\fill[gray] (\x+1.5,\y+2) node [left] {$d'_2$};
				\fill[green] (\x+3,\y+2) node [left] {$d'_3$};
				\fill[gray] (\x+2.5,\y-1.5) node [left] {$d_4$};
				\fill[gray] (\x+5,\y+2) node [left] {$d_5$};
				\fill[gray] (\x+4.5,\y-1.5) node [left] {$d_6$};
				\fill[gray] (\x+6,\y-1.5) node [left] {$d_1$};
				
				\draw[] (\x-6,\y+4) -- (\x-6,\y-3);
				\draw[] (\x-4,\y+4) -- (\x-6,\y-3);
				\draw[green,thick] (\x-2,\y+4) -- (\x-6,\y-3);
				\draw[] (\x-2,\y+4) -- (\x-4,\y-3);
				\draw[] (\x-0,\y+4) -- (\x-4,\y-3);
				\draw[] (\x-0,\y+4) -- (\x-2,\y-3);
				\draw[] (\x-0,\y+4) -- (\x-0,\y-3);
				
				\draw[gray] (\x-0,\y+4) -- (\x-0,\y-3);
				\draw[gray] (\x+2,\y+4) -- (\x,\y-3);
				\draw[green] (\x+4,\y+4) -- (\x+0,\y-3);
				\draw[gray] (\x+4,\y+4) -- (\x+2,\y-3);
				\draw[gray] (\x+6,\y+4) -- (\x+2,\y-3);
				\draw[gray] (\x+6,\y+4) -- (\x+4,\y-3);
				\draw[gray] (\x+6,\y+4) -- (\x+6,\y-3);
				
			}
		}
		\end{tikzpicture}
	\end{center}
\end{figure}

\begin{figure}[h!]
	\begin{center}
	\begin{tikzpicture}[scale = .5]
		\tikzstyle{every node} = [font = \small]
		\foreach \x in {0}
		{
			\foreach \y in {-8}
			{
				\draw[<-] (\x-9,\y+4) -- (\x+9,\y+4);
				\fill (\x-10,\y+4) node {$\partial'$};
				\draw[->] (\x-9,\y-3) -- (\x+9,\y-3);
				\fill (\x-10,\y-3) node {$\partial$};

				\foreach \t in {-8,-6,-4,-2,0,2,4,6, 8}
				{
					\fill (\x+\t,\y+4) circle (.1);
					\fill (\x+\t,\y-3) circle (.1);
				}
				
				\fill (\x-8,\y+4) node [above] {$1_{\partial'}$};
				\fill (\x-6,\y+4) node [above] {$0_{\partial'}$};
				\fill (\x-4,\y+4) node [above] {$-1_{\partial'}$};
				\fill (\x-2,\y+4) node [above] {$-2_{\partial'}$};
				\fill (\x+0,\y+4) node [above] {$-3_{\partial'}$};
				\fill (\x+2,\y+4) node [above] {$-4_{\partial'}$};
				\fill (\x+4,\y+4) node [above] {$-5_{\partial'}$};
				\fill (\x+6,\y+4) node [above] {$-6_{\partial'}$};
				\fill (\x+8,\y+4) node [above] {$-7_{\partial'}$};

				\fill (\x-8,\y-3) node [below] {$-1_{\partial}$};
				\fill (\x-6,\y-3) node [below] {$0_{\partial}$};
				\fill (\x-4,\y-3) node [below] {$1_{\partial}$};
				\fill (\x-2,\y-3) node [below] {$2_{\partial}$};
				\fill (\x+0,\y-3) node [below] {$3_{\partial}$};
				\fill (\x+2,\y-3) node [below] {$4_{\partial}$};
				\fill (\x+4,\y-3) node [below] {$5_{\partial}$};
				\fill (\x+6,\y-3) node [below] {$6_{\partial}$};
				\fill (\x+8,\y-3) node [below] {$7_{\partial}$};

				\fill[green] (\x-5,\y+2) node [left] {$d'_1$};
				\fill[] (\x-5.5,\y-1.5) node [left] {$d'_2$};
				\fill[] (\x-3,\y+2) node [left] {$d'_3$};
				\fill[] (\x-3.5,\y-1.5) node [left] {$d_4$};
				\fill[] (\x-1,\y+2) node [left] {$d_5$};
				\fill[] (\x-1.5,\y-1.5) node [left] {$d_6$};
				\fill[green] (\x+1,\y+2) node [left] {$d'_1$};
				
				\fill[gray] (\x+0.5,\y-1.5) node [left] {$d'_2$};
				\fill[gray] (\x+3,\y+2) node [left] {$d'_3$};
				\fill[gray] (\x+2.5,\y-1.5) node [left] {$d_4$};
				\fill[gray] (\x+5,\y+2) node [left] {$d_5$};
				\fill[gray] (\x+4.5,\y-1.5) node [left] {$d_6$};
				\fill[green] (\x+7,\y+2) node [left] {$d'_1$};
				
				\draw[green,thick] (\x-4,\y+4) -- (\x-8,\y-3);
				\draw[] (\x-4,\y+4) -- (\x-6,\y-3);
				\draw[] (\x-2,\y+4) -- (\x-6,\y-3);
				\draw[] (\x-2,\y+4) -- (\x-4,\y-3);
				\draw[] (\x-0,\y+4) -- (\x-4,\y-3);
				\draw[] (\x-0,\y+4) -- (\x-2,\y-3);
				\draw[green,thick] (\x+2,\y+4) -- (\x-2,\y-3);
				
				\draw[gray] (\x+2,\y+4) -- (\x,\y-3);
				\draw[gray] (\x+4,\y+4) -- (\x+0,\y-3);
				\draw[gray] (\x+4,\y+4) -- (\x+2,\y-3);
				\draw[gray] (\x+6,\y+4) -- (\x+2,\y-3);
				\draw[gray] (\x+6,\y+4) -- (\x+4,\y-3);
				\draw[green] (\x+8,\y+4) -- (\x+4,\y-3);
				
			}
		}
		\end{tikzpicture}
	\end{center}
\end{figure}

\begin{figure}[h!]
	\begin{center}
	\begin{tikzpicture}[scale = .5]
		\tikzstyle{every node} = [font = \small]
		\foreach \x in {0}
		{
			\foreach \y in {-8}
			{
				\draw[<-] (\x-9,\y+4) -- (\x+9,\y+4);
				\fill (\x-10,\y+4) node {$\partial'$};
				\draw[->] (\x-9,\y-3) -- (\x+9,\y-3);
				\fill (\x-10,\y-3) node {$\partial$};

				\foreach \t in {-8,-6,-4,-2,0,2,4,6, 8}
				{
					\fill (\x+\t,\y+4) circle (.1);
					\fill (\x+\t,\y-3) circle (.1);
				}
				
				\fill (\x-8,\y+4) node [above] {$1_{\partial'}$};
				\fill (\x-6,\y+4) node [above] {$0_{\partial'}$};
				\fill (\x-4,\y+4) node [above] {$-1_{\partial'}$};
				\fill (\x-2,\y+4) node [above] {$-2_{\partial'}$};
				\fill (\x+0,\y+4) node [above] {$-3_{\partial'}$};
				\fill (\x+2,\y+4) node [above] {$-4_{\partial'}$};
				\fill (\x+4,\y+4) node [above] {$-5_{\partial'}$};
				\fill (\x+6,\y+4) node [above] {$-6_{\partial'}$};
				\fill (\x+8,\y+4) node [above] {$-7_{\partial'}$};

				\fill (\x-8,\y-3) node [below] {$-1_{\partial}$};
				\fill (\x-6,\y-3) node [below] {$0_{\partial}$};
				\fill (\x-4,\y-3) node [below] {$1_{\partial}$};
				\fill (\x-2,\y-3) node [below] {$2_{\partial}$};
				\fill (\x+0,\y-3) node [below] {$3_{\partial}$};
				\fill (\x+2,\y-3) node [below] {$4_{\partial}$};
				\fill (\x+4,\y-3) node [below] {$5_{\partial}$};
				\fill (\x+6,\y-3) node [below] {$6_{\partial}$};
				\fill (\x+8,\y-3) node [below] {$7_{\partial}$};

				\fill[] (\x-7,\y-1.5) node [left] {$d'_1$};
				\fill[] (\x-5.5,\y-1.5) node [left] {$d'_2$};
				\fill[] (\x-3,\y+2) node [left] {$d'_3$};
				\fill[] (\x-3.5,\y-1.5) node [left] {$d_4$};
				\fill[] (\x-1,\y+2) node [left] {$d_5$};
				\fill[green] (\x+0.25,\y+2) node [left] {$d'_6$};
				\fill[] (\x-1,\y-1.5) node [left] {$d'_1$};
				
				\fill[gray] (\x+0.5,\y-1.5) node [left] {$d'_2$};
				\fill[gray] (\x+3,\y+2) node [left] {$d'_3$};
				\fill[gray] (\x+2.5,\y-1.5) node [left] {$d_4$};
				\fill[gray] (\x+5,\y+2) node [left] {$d_5$};
				\fill[green] (\x+6.25,\y+2) node [left] {$d'_6$};
				\fill[gray] (\x+5,\y-1.5) node [left] {$d'_1$};
				
				\draw[] (\x-4,\y+4) -- (\x-8,\y-3);
				\draw[] (\x-4,\y+4) -- (\x-6,\y-3);
				\draw[] (\x-2,\y+4) -- (\x-6,\y-3);
				\draw[] (\x-2,\y+4) -- (\x-4,\y-3);
				\draw[] (\x-0,\y+4) -- (\x-4,\y-3);
				\draw[green,thick] (\x+2,\y+4) -- (\x-4,\y-3);
				\draw[] (\x+2,\y+4) -- (\x-2,\y-3);
				
				\draw[gray] (\x+2,\y+4) -- (\x,\y-3);
				\draw[gray] (\x+4,\y+4) -- (\x+0,\y-3);
				\draw[gray] (\x+4,\y+4) -- (\x+2,\y-3);
				\draw[gray] (\x+6,\y+4) -- (\x+2,\y-3);
				\draw[green] (\x+8,\y+4) -- (\x+2,\y-3);
				\draw[gray] (\x+8,\y+4) -- (\x+4,\y-3);
				
			}
		}
		\end{tikzpicture}
	\end{center}
\end{figure}

\begin{figure}[h!]
	\begin{center}
	\begin{tikzpicture}[scale = .5]
		\tikzstyle{every node} = [font = \small]
		\foreach \x in {0}
		{
			\foreach \y in {-8}
			{
				\draw[<-] (\x-9,\y+4) -- (\x+9,\y+4);
				\fill (\x-10,\y+4) node {$\partial'$};
				\draw[->] (\x-9,\y-3) -- (\x+9,\y-3);
				\fill (\x-10,\y-3) node {$\partial$};

				\foreach \t in {-8,-6,-4,-2,0,2,4,6, 8}
				{
					\fill (\x+\t,\y+4) circle (.1);
					\fill (\x+\t,\y-3) circle (.1);
				}
				
				\fill (\x-8,\y+4) node [above] {$1_{\partial'}$};
				\fill (\x-6,\y+4) node [above] {$0_{\partial'}$};
				\fill (\x-4,\y+4) node [above] {$-1_{\partial'}$};
				\fill (\x-2,\y+4) node [above] {$-2_{\partial'}$};
				\fill (\x+0,\y+4) node [above] {$-3_{\partial'}$};
				\fill (\x+2,\y+4) node [above] {$-4_{\partial'}$};
				\fill (\x+4,\y+4) node [above] {$-5_{\partial'}$};
				\fill (\x+6,\y+4) node [above] {$-6_{\partial'}$};
				\fill (\x+8,\y+4) node [above] {$-7_{\partial'}$};

				\fill (\x-8,\y-3) node [below] {$-1_{\partial}$};
				\fill (\x-6,\y-3) node [below] {$0_{\partial}$};
				\fill (\x-4,\y-3) node [below] {$1_{\partial}$};
				\fill (\x-2,\y-3) node [below] {$2_{\partial}$};
				\fill (\x+0,\y-3) node [below] {$3_{\partial}$};
				\fill (\x+2,\y-3) node [below] {$4_{\partial}$};
				\fill (\x+4,\y-3) node [below] {$5_{\partial}$};
				\fill (\x+6,\y-3) node [below] {$6_{\partial}$};
				\fill (\x+8,\y-3) node [below] {$7_{\partial}$};

				\fill[] (\x-7,\y-1.5) node [left] {$d'_1$};
				\fill[] (\x-5.5,\y-1.5) node [left] {$d'_2$};
				\fill[] (\x-3,\y+2) node [left] {$d'_3$};
				\fill[green] (\x-1.75,\y+2) node [left] {$d'_4$};
				\fill[] (\x-3,\y-1.5) node [left] {$d_5$};
				\fill[] (\x+0.25,\y+2) node [left] {$d'_6$};
				\fill[] (\x-1,\y-1.5) node [left] {$d'_1$};
				
				\fill[gray] (\x+0.5,\y-1.5) node [left] {$d'_2$};
				\fill[gray] (\x+3,\y+2) node [left] {$d'_3$};
				\fill[green] (\x+4.25,\y+2) node [left] {$d'_4$};
				\fill[gray] (\x+3,\y-1.5) node [left] {$d_5$};
				\fill[gray] (\x+6.25,\y+2) node [left] {$d'_6$};
				\fill[gray] (\x+5,\y-1.5) node [left] {$d'_1$};
				
				\draw[] (\x-4,\y+4) -- (\x-8,\y-3);
				\draw[] (\x-4,\y+4) -- (\x-6,\y-3);
				\draw[] (\x-2,\y+4) -- (\x-6,\y-3);
				\draw[green,thick] (\x,\y+4) -- (\x-6,\y-3);
				\draw[] (\x-0,\y+4) -- (\x-4,\y-3);
				\draw[] (\x+2,\y+4) -- (\x-4,\y-3);
				\draw[] (\x+2,\y+4) -- (\x-2,\y-3);
				
				\draw[gray] (\x+2,\y+4) -- (\x,\y-3);
				\draw[gray] (\x+4,\y+4) -- (\x+0,\y-3);
				\draw[green] (\x+6,\y+4) -- (\x,\y-3);
				\draw[gray] (\x+6,\y+4) -- (\x+2,\y-3);
				\draw[gray] (\x+8,\y+4) -- (\x+2,\y-3);
				\draw[gray] (\x+8,\y+4) -- (\x+4,\y-3);
				
			}
		}
		\end{tikzpicture}
	\end{center}
\end{figure}

\begin{figure}[h!]
	\begin{center}
	\begin{tikzpicture}[scale = .5]
		\tikzstyle{every node} = [font = \small]
		\foreach \x in {0}
		{
			\foreach \y in {-8}
			{
				\draw[<-] (\x-9,\y+4) -- (\x+9,\y+4);
				\fill (\x-10,\y+4) node {$\partial'$};
				\draw[->] (\x-9,\y-3) -- (\x+9,\y-3);
				\fill (\x-10,\y-3) node {$\partial$};

				\foreach \t in {-8,-6,-4,-2,0,2,4,6, 8}
				{
					\fill (\x+\t,\y+4) circle (.1);
					\fill (\x+\t,\y-3) circle (.1);
				}
				
				\fill (\x-8,\y+4) node [above] {$1_{\partial'}$};
				\fill (\x-6,\y+4) node [above] {$0_{\partial'}$};
				\fill (\x-4,\y+4) node [above] {$-1_{\partial'}$};
				\fill (\x-2,\y+4) node [above] {$-2_{\partial'}$};
				\fill (\x+0,\y+4) node [above] {$-3_{\partial'}$};
				\fill (\x+2,\y+4) node [above] {$-4_{\partial'}$};
				\fill (\x+4,\y+4) node [above] {$-5_{\partial'}$};
				\fill (\x+6,\y+4) node [above] {$-6_{\partial'}$};
				\fill (\x+8,\y+4) node [above] {$-7_{\partial'}$};

				\fill (\x-8,\y-3) node [below] {$-1_{\partial}$};
				\fill (\x-6,\y-3) node [below] {$0_{\partial}$};
				\fill (\x-4,\y-3) node [below] {$1_{\partial}$};
				\fill (\x-2,\y-3) node [below] {$2_{\partial}$};
				\fill (\x+0,\y-3) node [below] {$3_{\partial}$};
				\fill (\x+2,\y-3) node [below] {$4_{\partial}$};
				\fill (\x+4,\y-3) node [below] {$5_{\partial}$};
				\fill (\x+6,\y-3) node [below] {$6_{\partial}$};
				\fill (\x+8,\y-3) node [below] {$7_{\partial}$};

				\fill[] (\x-7,\y-1.5) node [left] {$d'_1$};
				\fill[] (\x-5.5,\y-1.5) node [left] {$d'_2$};
				\fill[] (\x-3,\y+2) node [left] {$d'_3$};
				\fill[] (\x-1.75,\y+2) node [left] {$d'_4$};
				\fill[green] (\x-0.5,\y+2) node [left] {$d'_5$};
				\fill[] (\x-2.75,\y-1.5) node [left] {$d'_6$};
				\fill[] (\x-1,\y-1.5) node [left] {$d'_1$};
				
				\fill[gray] (\x+0.5,\y-1.5) node [left] {$d'_2$};
				\fill[gray] (\x+3,\y+2) node [left] {$d'_3$};
				\fill[gray] (\x+4.25,\y+2) node [left] {$d'_4$};
				\fill[green] (\x+5.5,\y+2) node [left] {$d'_5$};
				\fill[gray] (\x+3.25,\y-1.5) node [left] {$d'_6$};
				\fill[gray] (\x+5,\y-1.5) node [left] {$d'_1$};
				
				\draw[] (\x-4,\y+4) -- (\x-8,\y-3);
				\draw[] (\x-4,\y+4) -- (\x-6,\y-3);
				\draw[] (\x-2,\y+4) -- (\x-6,\y-3);
				\draw[] (\x,\y+4) -- (\x-6,\y-3);
				\draw[green,thick] (\x+2,\y+4) -- (\x-6,\y-3);
				\draw[] (\x+2,\y+4) -- (\x-4,\y-3);
				\draw[] (\x+2,\y+4) -- (\x-2,\y-3);
				
				\draw[gray] (\x+2,\y+4) -- (\x,\y-3);
				\draw[gray] (\x+4,\y+4) -- (\x+0,\y-3);
				\draw[gray] (\x+6,\y+4) -- (\x,\y-3);
				\draw[green] (\x+8,\y+4) -- (\x,\y-3);
				\draw[gray] (\x+8,\y+4) -- (\x+2,\y-3);
				\draw[gray] (\x+8,\y+4) -- (\x+4,\y-3);
				
			}
		}
		\end{tikzpicture}
	\end{center}
\caption{$\cox(T) = \mu_5\mu_4\mu_6\mu_1\mu_3\mu_2T$.}
\end{figure}
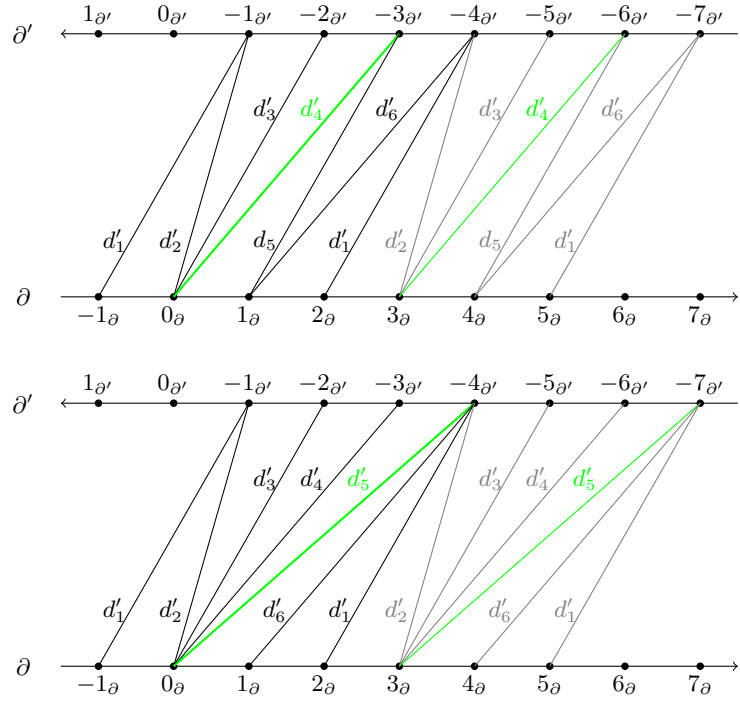
\label{cox_tran}
\end{example}
\clearpage

As shown in Lemma \ref{lemma:endpoints}, this Coxeter transformation shifts the endpoints of each arc in the triangulation by -1 on each boundary component. The same effect can be achieved by rotating the outer boundary component clockwise by $\frac{2\pi}{p}$, and the inner boundary component counter-clockwise by $\frac{2\pi}{q}$. In the example, $d_1 = [0_\partial, 0_{\partial'}] \mapsto [-1_\partial, -1_{\partial'}]$. Applying the Coxeter transformation $p$ times would result in moving the endpoints on $\partial$ a full turn in the clockwise direction. Similarly, applying the Coxeter transformation $q$ times would result in moving the endpoints on $\partial'$ a full turn around in the counter-clockwise direction. 

In the example above, $p = q = 3$, so applying the Coxeter transformation two more times ($\cox^2(\cox(T)) = \cox^3(T)$) would result in Figure \ref{two_cox_tran}.
\begin{figure}[h!]
	\begin{center}
	\begin{tikzpicture}[scale = .5]
		\tikzstyle{every node} = [font = \small]
		\foreach \x in {0}
		{
			\foreach \y in {-8}
			{
				\draw[<-] (\x-13,\y+4) -- (\x+7,\y+4);
				\fill (\x-15,\y+4) node {$\partial'$};
				\draw[->] (\x-13,\y-3) -- (\x+7,\y-3);
				\fill (\x-15,\y-3) node {$\partial$};

				\foreach \t in {-12,-10,-8,-6,-4,-2,0,2,4,6}
				{
					\fill (\x+\t,\y+4) circle (.1);
					\fill (\x+\t,\y-3) circle (.1);
				}
				
				\draw[dashed,gray] (\x-12,\y+4) -- (\x-12,\y-3);
				\draw[dashed,gray] (\x-6,\y+4) -- (\x-6,\y-3);
				\draw[dashed,gray] (\x-0,\y+4) -- (\x-0,\y-3);
				\draw[dashed,gray] (\x+6,\y+4) -- (\x+6,\y-3);

				\fill (\x-12,\y+4) node [above] {$3_{\partial'}$};
				\fill (\x-10,\y+4) node [above] {$2_{\partial'}$};
				\fill (\x-8,\y+4) node [above] {$1_{\partial'}$};
				\fill (\x-6,\y+4) node [above] {$0_{\partial'}$};
				\fill (\x-4,\y+4) node [above] {$-1_{\partial'}$};
				\fill (\x-2,\y+4) node [above] {$-2_{\partial'}$};
				\fill (\x+0,\y+4) node [above] {$-3_{\partial'}$};
				\fill (\x+2,\y+4) node [above] {$-4_{\partial'}$};
				\fill (\x+4,\y+4) node [above] {$-5_{\partial'}$};
				\fill (\x+6,\y+4) node [above] {$-6_{\partial'}$};

				\fill (\x-12,\y-3) node [below] {$-3_{\partial}$};
				\fill (\x-10,\y-3) node [below] {$-2_{\partial}$};
				\fill (\x-8,\y-3) node [below] {$-1_{\partial}$};
				\fill (\x-6,\y-3) node [below] {$0_{\partial}$};
				\fill (\x-4,\y-3) node [below] {$1_{\partial}$};
				\fill (\x-2,\y-3) node [below] {$2_{\partial}$};
				\fill (\x+0,\y-3) node [below] {$3_{\partial}$};
				\fill (\x+2,\y-3) node [below] {$4_{\partial}$};
				\fill (\x+4,\y-3) node [below] {$5_{\partial}$};
				\fill (\x+6,\y-3) node [below] {$6_{\partial}$};

				\fill[] (\x-10.4,\y-2) node [left] {$d_1$};
				\fill[] (\x-8.75,\y-2) node [left] {$d_2$};
				\fill[] (\x-0,\y+3) node [left] {$d_3$};
				\fill[] (\x+1.75,\y+3) node [left] {$d_4$};
				\fill[] (\x+3.25,\y+3) node [left] {$d_5$};
				\fill[] (\x-6.25,\y-2) node [left] {$d_6$};
				\fill[] (\x-4.5,\y-2) node [left] {$d_1$};

				\draw[] (\x-0,\y+4) -- (\x-12,\y-3);
				\draw[] (\x-0,\y+4) -- (\x-10,\y-3);
				\draw[] (\x+2,\y+4) -- (\x-10,\y-3);
				\draw[] (\x+4,\y+4) -- (\x-10,\y-3);
				\draw[] (\x+6,\y+4) -- (\x-10,\y-3);
				\draw[] (\x+6,\y+4) -- (\x-8,\y-3);
				\draw[] (\x+6,\y+4) -- (\x-6,\y-3);
				
			}
		}
		\end{tikzpicture}
	\end{center}
\caption{$\cox^3(T)$}
\label{two_cox_tran}
\end{figure}

\medskip
We can extend the definition of the Coxeter transformation to arbitrary triangulations of $C_{p,q}$ by first flipping all peripheral arcs to reach a bridging triangulation. Let $T$ be a triangulation with $k \geq 1$ peripheral arcs. Then there exists a finite sequence of flips $\mu_{\alpha} := \mu_{\alpha_1}, \ldots, \mu_{\alpha_k}$, where $\alpha_i$ is a peripheral arc, for every $i$, and $\alpha_p$ is a bounding arc in the triangulation $\mu_{\alpha_{p-1}} \cdots \mu_{\alpha_1}T$. Note that this sequence is not necessarily unique, since we may have more than one bounding arc in our triangulation at any given time. Other sequences of flips (flipping non-bounding arcs, for example) may also result in a bridging triangulation, but flipping only bounding arcs will give us a minimal sequence of flips. However, our resulting triangulation is independent of the order in which we choose to mutate the bounding arcs. 

\begin{proposition}
Using the same notation as above, the bridging triangulation $\tilde{T} = \mu_\alpha T$ is uniquely determined (unique up to labeling of arcs).
\end{proposition}

\begin{proof}
Let $T$ be a a triangulation of $C_{p,q}$. We know by Lemma \ref{ann_bridge} that $T$ contains at least two bridging arcs. We claim that $\tilde{T}$ is determined by the bridging arcs in $T$. Every bridging arc is an edge of two triangles in $T$ where in both triangles one of the other two edges is also a bridging arc. Without loss of generality, let $d_i = [i_\partial, k_{\partial'}]$ and $d_j = [j_\partial, k_{\partial'}]$ be two bridging arcs in $T$ such that $d_i$ and $d_j$ are two sides of a triangle in $T$, and $i < j$. We consider the triangulation restricted to the polygon $P_k$ where the boundary of $P_k$ is made up of the arcs of $[i_\partial, k_{\partial'}]$ and $[j_\partial, k_{\partial'}]$, and the boundary segment $[i_\partial, j_{\partial}]$ of $C_{p,q}$. Then we can find a finite sequence of flips $\mu_{\alpha_k}$ such that after performing this sequence, all arcs in $P_k$ have an endpoint at $k_{\partial'}$ (note that if there are no internal arcs in $P_k$, then we have the empty sequence). Such a sequence exists because any two triangulations of a surface $S$ are related through a sequence of flips. Then we have a fan of bridging arcs $[(i+1)_\partial, k_{\partial'}] \cdots [(j-1)_\partial, k_{\partial'}]$ in our triangulation $\mu_{\alpha_k}T$ stemming from $k_{\partial'}$. We do this for every such triangle where two sides are bridging arcs of $T$. Our bridging triangulation $\tilde{T} = \mu_\alpha T$ is then comprised of fans of bridging arc originating at the endpoints where two bridging arcs of $T$ meet. 

\begin{center}
\begin{tikzpicture}[scale = .4]
		\tikzstyle{every node} = [font = \tiny]
		\foreach \x in {-12}
		{
			\foreach \y in {-8}
			{
				\draw[<-] (\x-5,\y+4) -- (\x+5,\y+4);
				\fill (\x-6,\y+4) node {\small{$\partial'$}};
				\draw[->] (\x-5,\y-0) -- (\x+5,\y-0);
				\fill (\x-6,\y-0) node {\small{$\partial$}};

				\draw[thick] (\x-0,\y+4) -- (\x-4,\y-0);
				\draw[thick] (\x-0,\y+4) -- (\x+4,\y-0);
				\draw[thick] (\x-4,\y+0) -- (\x+4,\y-0);
				
				\draw[gray] (\x-4,\y+4) -- (\x-4,\y-0);
				\draw[gray](\x+4,\y+4) -- (\x+4,\y-0);
				\draw[gray](\x-2,\y+4) -- (\x-4,\y-0);

				\draw[] (\x-4,\y-0) .. controls (\x-1.75,\y+2.25) and (\x+1.75,\y+2.25) .. (\x+4,\y-0);
				\draw[] (\x-4,\y-0) .. controls (\x-3,\y+1) and (\x-1,\y+1) .. (\x+0,\y+0);
				\draw[] (\x-0,\y+0) .. controls (\x+1,\y+1) and (\x+3,\y+1) .. (\x+4,\y+0);
				
				\draw[gray] (\x-0,\y+4) .. controls (\x+1,\y+3) and (\x+3,\y+3) .. (\x+4,\y+4);
				
				\foreach \t in {-4,-2,0,2,4}
				{
					\fill (\x+\t,\y-0) circle (.1);
					\fill (\x+\t,\y+4) circle (.1);
				}
					
				\fill (\x-0,\y+4) node [above] {$k_{\partial'}$};

				\fill (\x-4,\y-0) node [below] {$i_{\partial}$};
				\fill (\x+4,\y-0) node [below] {$j_{\partial}$};	
				
					\fill (\x-1.75,\y+2.5) node [above] {$d_{i}$};
				\fill (\x+2.75,\y+1.5) node [above] {$d_{j}$};
				
				\draw [->] (\x+7,\y+2) -- (\x+9,\y+2);
			}
			
		}
		
		\foreach \x in {6}
		{
			\foreach \y in {-8}
			{
				\draw[<-] (\x-5,\y+4) -- (\x+5,\y+4);
				\fill (\x-6,\y+4) node {\small{$\partial'$}};
				\draw[->] (\x-5,\y-0) -- (\x+5,\y-0);
				\fill (\x-6,\y-0) node {\small{$\partial$}};

				\draw[thick] (\x-0,\y+4) -- (\x-4,\y-0);
				\draw[thick] (\x-0,\y+4) -- (\x+4,\y-0);
				\draw[thick] (\x-4,\y+0) -- (\x+4,\y-0);
				
				\draw[gray] (\x-4,\y+4) -- (\x-4,\y-0);
				\draw[gray](\x+4,\y+4) -- (\x+4,\y-0);
				\draw[gray](\x-2,\y+4) -- (\x-4,\y-0);
				
				\draw[](\x+0,\y+4) -- (\x-2,\y-0);
				\draw[](\x+0,\y+4) -- (\x+0,\y-0);
				\draw[](\x+0,\y+4) -- (\x+2,\y-0);

				\draw[gray] (\x-0,\y+4) .. controls (\x+1,\y+3) and (\x+3,\y+3) .. (\x+4,\y+4);
			
				\foreach \t in {-4,-2,0,2,4}
				{
					\fill (\x+\t,\y-0) circle (.1);
					\fill (\x+\t,\y+4) circle (.1);
				}
					
				\fill (\x-0,\y+4) node [above] {$k_{\partial'}$};
				\fill (\x-4,\y-0) node [below] {$i_{\partial}$};
				\fill (\x+4,\y-0) node [below] {$j_{\partial}$};	
				
				\fill (\x-1.75,\y+2.5) node [above] {$d_{i}$};
				\fill (\x+2.75,\y+1.5) node [above] {$d_{j}$};
			}
			
		}
\end{tikzpicture}
\end{center}
\end{proof}

\begin{definition}
Let $T$ be a triangulation with peripheral arcs, and $\{\mu_{\alpha_i}\}_{i \in I}$ a finite sequence of flips, $\alpha_i$ peripheral, so that $\tilde{T} = \mu_\alpha T$ consists only of bridging arcs. Then the Coxeter transformation of $T$ is:
\[
\cox(T) = (\mu_\alpha)^{-1}\cox(\mu_\alpha T).
\]
\label{cox:def}
\end{definition}

\begin{theorem}
Let $T$ be a triangulation of $C_{p,q}$. 
\begin{itemize}
\item[(1)] We have $\mu D^+_z(T) = D^+_z(\mu T)$ for every arc flip $\mu$.
\end{itemize}
Let $\tilde{T} = \mu_\alpha T$ be a bridging triangulation of $C_{p,q}$, and let $m = \lcm(p,q)$. Then there exist $r,s \in \mathbb{N}$ such that $pr = m$ and $qs = m$. We have the following commutativity relations:
\begin{itemize}
\item[(2)] \,\, $\cox^m(\tilde{T}) = D^{r+s}_z(\tilde{T}),$
\item[(3)] \,\, $D^{r+s}_z(T) = \cox^m(T)$.
\end{itemize}
\label{th:comm}
\end{theorem}

\begin{proof}
\begin{enumerate}
\item Dehn twists do not change relative positions of arcs in a triangulation, so the arcs involved in a quadrilateral still form a quadrilateral after applying the Dehn twist. Thus the following diagram commutes for every arc flip $\mu$:

\begin{figure}[h!]
\centering
\begin{tikzpicture}[xscale = .175, yscale = .175]
\tikzstyle{every node} = [font = \tiny]
					\foreach \x in {0}
					{
						\foreach \y in {0}
						{
						
						\draw (\x-4, \y+4) -- (\x+4, \y+4);
						\draw (\x-4, \y) -- (\x+4, \y);
						
						\draw (\x-2, \y+4) -- (\x-2, \y);
						\draw (\x-2, \y) -- (\x+2, \y+4);
						\draw (\x+2, \y) -- (\x+2, \y+4);
						\draw (\x-2, \y) .. controls (\x-.5,\y+1) and (\x+.5,\y+1) .. (\x+2,\y);

							
						\draw (\x-4, \y-10) -- (\x+4, \y-10);
						\draw (\x-4, \y-14) -- (\x+4, \y-14);
						
						\draw (\x-2, \y-10) -- (\x-2, \y-14);
						\draw (\x-2, \y-14) -- (\x+2, \y-10);
						\draw (\x+2, \y-14) -- (\x+2, \y-10);
						\draw (\x, \y-14) -- (\x+2,\y-10);	
							
						\draw (\x+14, \y+4) -- (\x+22, \y+4);
						\draw (\x+14, \y) -- (\x+22, \y);
						
						\draw (\x+17, \y+4) -- (\x+15, \y);
						\draw (\x+15, \y) -- (\x+21, \y+4);
						\draw (\x+19, \y) -- (\x+21, \y+4);
						\draw (\x+15, \y) .. controls (\x+16.5,\y+1) and (\x+17.5,\y+1) .. (\x+19,\y);	
							
						\draw (\x+14, \y-10) -- (\x+22, \y-10);
						\draw (\x+14, \y-14) -- (\x+22, \y-14);
						
						\draw (\x+17, \y-10) -- (\x+15, \y-14);
						\draw (\x+15, \y-14) -- (\x+21, \y-10);
						\draw (\x+19, \y-14) -- (\x+21, \y-10);
						\draw (\x+17, \y-14) -- (\x+21,\y-10);	
							
						}
					}
					\draw[->] (8,2) -- (10,2);
					\node[above] at (9,2) {$D_z(T)$};
					\draw[->] (8,-12) -- (10,-12);
					\node[above] at (9,-12) {\tiny{$D_z(\mu T)$}};
					\draw[->] (0,-4) -- (0, -6);
					\node[left] at (0,-5) {$\tiny{\mu T}$};
					\draw[->] (18,-4)--(18,-6);
					\node[right] at (18,-5) {$\mu D_z(T)$};
	\end{tikzpicture}
\end{figure}
\medskip

\item One iteration of $\cox(\tilde{T})$ moves an arc $[i_\partial,j_{\partial'}] \mapsto [(i-1)_\partial,(j-1)_{\partial'}]$. Since $pr = m = qs$, $\cox^m(\tilde{T}) :[i_\partial,j_{\partial'}] \mapsto [(i-m)_\partial,(j-m)_{\partial'}] = [(i-pr)_\partial,(j-qs)_{\partial'}] $. So $\tilde{T}$ has shifted endpoints of arcs $r$ frames in the negative direction on boundary $\partial$, and $s$ frames in the negative direction on $\partial'$. In total, the triangulation now stretches $r+s$ frames, and thus $\cox^m(\tilde{T}) = D^{r+s}_z(\tilde{T})$.
\item We will use parts (1) \& (2) to prove (3). 
\[D_z^{r+s}(T)  = \mu_\alpha^{-1} \mu_\alpha D_z^{r+s}(T) \stackrel{(1)}{=} \mu_\alpha^{-1} D_z^{r+s} (\mu_\alpha T)   \stackrel{(2)}{=} \mu_\alpha^{-1} \cox^m(\mu_\alpha T) = \cox^m(T) \]
\end{enumerate}
\end{proof}

We define $\Cox := \cox^m$, where $m = \lcm(p,q)$. The endpoints of the arcs of a triangulation $T$ of $C_{p,q}$ are invariant under $\Cox(T)$.

These commutativity relations provide us with a dictionary to go between the topological and algebraic framework. This becomes useful when considering what happens to the quivers (or root systems) under the Dehn twist, and to see what happens to a triangulation when applying a Coxeter transformation. The Coxeter transformation for triangulations of the annulus can be defined for planar surfaces with several boundary components. This comes down to choosing appropriate ``boundaries" $\partial_1, \partial_2$.

\subsection{Coxeter transformations of surfaces with several boundary components}

Let $(S,M)$ be a marked planar surface with several boundary components, such that each boundary component has at least one marked point, and let $\tilde{T} = T$ be a bridging triangulation of $S$, i.e. a triangulation where the endpoints of each arc lie on different boundary components. We choose a simple, non-contractible loop $z$ in $S$. The \emph{interior} of the loop $z$ is to the right of $z$ when moving along $z$ in the clockwise direction. The exterior is then to the left of $z$. We use the following notation:

\begin{itemize}
\item Let $D(z) = \{d_{i_1}, \ldots, d_{i_d}\}$ denote the set of arcs of $T$ that intersect $z$.
\item Let $V_1(z)$ be the set of marked points in the interior of $z$ such that every marked point in $V_1$ is the endpoint of at least one $d_i \in D(z)$.
\item Let $V_2(z)$ be the set of marked points in the exterior of $z$ such that every marked point in $V_2$ is the endpoint of at least one $d_i \in D(z)$.
\end{itemize}

Then there exists a minimal cycle $c_1$ (not necessarily unique), formed by arcs in $T$ and boundary segments of $S$ (arcs may appear more than once in the cycle), connecting all $m \in V_1$. We set $\partial_1 = c_1$. There also exists a minimal cycle $c_2$ (not necessarily unique), formed by arcs in $T$ and boundary segments of $S$ (arcs may appear more than once in the cycle), connecting all $m \in V_2$. We set $\partial_2 = c_2$. 

We can then consider $T$ restricted to the region between $\partial_1$ and $\partial_2$. This is an annulus triangulated by $D(z)$, and we can apply the machinery from Section 5.1 to $T$.

\begin{example}
Consider the surface $S$ with four boundary components, drawn below.

\begin{tikzpicture}[scale = .6]
	\tikzstyle{every node} = [font = \small]
	 
	 \draw[] (-10,0) -- (-6.5,0);
	 \draw[] (-10,0) .. controls (-6.5,-1.75) and (-3.5,-2) ..  (-3.5,0);
	 \draw[] (-10,0) .. controls (-5,-2.75) and (-3.5,-2) ..  (-.5,0);
	 
	 \draw[] (0,3) .. controls (-5.5,2) and (-6.5,1.75) .. (-6.5,0);
	 \draw[] (0,3) .. controls (-2,2) and (-3,1.5) ..  (-3.5,0);
	 \draw[preaction={draw,green,-,double=green,double distance=2\pgflinewidth,}] 
	 		(0,3) .. controls (3,2) and (3.5,1.75) ..  (4.5,0);
	 \draw[] (0,3) .. controls (-.5,2) and (-.6,1.5) ..  (-.5,0);

	 \draw[] (10,0) .. controls (7,2.75) and (4.5,2) .. (4.5,0);
	\draw[] (10,0) ..controls (7,-3.25) and (-1,-3.25) .. (-.5,0);
	 \draw[preaction={draw,green,-,double=green,double distance=2\pgflinewidth,}] 
	 		(10,0) .. controls (8,-2) and (4.5,-2) ..  (4.5,0);	 
	 
	 \draw[preaction={draw,orange,-,double=orange,double distance=2\pgflinewidth,}] (-3.5,0) -- (-.5,0);
	 
	 \draw[] (-.5,0) .. controls (-.5,-1.75) and (1.5,-2.25) .. (4.5,0);
	 \draw[] (-.5,0) .. controls (-.5,2) and (1.5,2.5) ..  (4.5,0);	
	 
	 \draw[thick] (0,0) ellipse (10cm and 3cm);
	 \draw[thick,preaction={draw,orange,-,double=orange,double distance=2\pgflinewidth,}]
	 		 (1,0) ellipse (1.5cm and 1cm);
			 
	 \draw[thick,preaction={draw,orange,-,double=orange,double distance=2\pgflinewidth,}] 
	 		(-5,0) ellipse (1.5cm and 1cm);
			
	 \draw[thick] (6,0) ellipse (1.5cm and 1cm);
	 \draw[thick,blue,dashed] (-2, 0) ellipse (5.5cm and 1.7cm);
	 \fill (-0.5,-1.75) node [below,blue] {$z$};
	
	\draw[thick,preaction={draw,green,-,double=green,double distance=2\pgflinewidth,}]
		 (-10,0) arc (180:270:10cm and 3cm);
		
	\draw[thick,preaction={draw,green,-,double=green,double distance=2\pgflinewidth,}] 
		(-10,0) arc (180:90:10cm and 3cm);
	
	\draw[thick,preaction={draw,green,-,double=green,double distance=2\pgflinewidth,}] 
		(10,0) arc (0:-90:10cm and 3cm);
	
	 \fill (-10,0) circle (.15);
	 \fill (-10,0) node [left] {$1_{\partial}$};
	 \fill (0,3) circle (.15);
	 \fill (0,3) node [above] {$2_{\partial}$};
	\fill (10,0) circle (.15);
	\fill (10,0) node [right] {$3_{\partial}$};
	
	 \fill (-6.5,0) circle (.15);
	 \fill (-6.5,0) node [right] {\tiny{$1_{\partial'}$}};
	 \fill (-3.5,0) circle (.15);
	 \fill (-3.5,0) node [left] {\tiny{$2_{\partial'}$}};
	 \fill (-5,1) node [below] {\tiny{$\partial'$}};
	
	 \fill (-.5,0) circle (.15);
	 \fill (-.5,0) node [right] {\tiny{$1_{\partial''}$}};
	  \fill (1,1) node [below] {\tiny{$\partial''$}};
	
	 \fill (4.5,0) circle (.15);
	 \fill (4.5,0) node [right] {\tiny{$1_{\partial'''}$}};
	  \fill (6,1) node [below] {\tiny{$\partial'''$}};
	 
	 \fill (-8,0) node [above] {\tiny{$d_1$}};
	 \fill (-6,1.8) node [] {\tiny{$d_2$}};
	 \fill (-1.75,2.2) node [left] {\tiny{$d_3$}};
	 \fill (-.5,2.2) node [right] {\tiny{$d_4$}};
	 \fill (1.95,2.5) node [right] {\tiny{$d_5$}};
	 \fill (5,2) node [right] {\tiny{$d_6$}};
	 \fill (8.5,-1) node [above] {\tiny{$d_7$}};
	 \fill (4,-2) node [right] {\tiny{$d_8$}};
	 \fill (-4,-2) node [] {\tiny{$d_9$}};
	 \fill (-3.9,-.75) node [right] {\tiny{$d_{10}$}};
	 \fill (-2,0) node [above] {\tiny{$d_{11}$}};
	 \fill (2.1,1.5) node [right] {\tiny{$d_{12}$}};
	 \fill (2.05,-1.5) node [right] {\tiny{$d_{13}$}};
	 
\end{tikzpicture}

We want to perform a Coxeter transformation with respect to (the closed curve) $z$ as chosen. We have
\begin{align*}
D(z) &= \{d_1,d_2,d_3,d_4,d_8,d_9,d_{10},d_{12},d_{13}\}, \\ 
V_1(z) &= \{1_{\partial'},2_{\partial'}, 1_{\partial''}\},\\
V_2(z) &= \{1_{\partial},2_{\partial}, 1_{\partial'''}, 3_{\partial}\}.
\end{align*} 
We now find minimal cycles $c_1$ and $c_2$.
The cycle $c_1$ is marked in orange, and it is $1_{\partial'} \rightarrow 2_{\partial'} \rightarrow 1_{\partial''} \rightarrow 1_{\partial''} \rightarrow 2_{\partial'} \rightarrow 1_{\partial'}$. Note that $1_{\partial''}$ is repeated twice in a row. This is because the boundary component $\partial''$ cannot be contracted to a single point. This cycle now becomes our boundary $\partial_1$.
The cycle $c_2$ is marked in green, and it is $1_{\partial} \rightarrow 2_{\partial} \rightarrow 1_{\partial'''} \rightarrow 3_{\partial} \rightarrow 1_{\partial}$. This cycle now becomes our boundary $\partial_2$. We can then represent the part of the triangulation between $c_1$ and $c_2$ as a cylinder $Cyl_{4,5}$:

\begin{center}
\begin{tikzpicture}[scale = .5]
	\tikzstyle{every node} = [font = \small]
	
			\foreach \x in {0}
		{
			\foreach \y in {-8}
			{
				\draw[thick,preaction={draw,orange,-,double=orange,double distance=2\pgflinewidth,}] 
						(\x-6,\y+4) -- (\x+6,\y+4);
				\fill (\x-7,\y+4) node {$\partial_1$};
				\draw[thick,preaction={draw,green,-,double=green,double distance=2\pgflinewidth,}]
						(\x-6,\y-0) -- (\x+6,\y-0);
				\fill (\x-7,\y-0) node {$\partial_2$};

				\draw[] (\x-5,\y+4) -- (\x-5,\y-0);
				\draw[] (\x-5,\y+4) -- (\x-2.5,\y-0);
				\draw[] (\x-3,\y+4) -- (\x-2.5,\y-0);
				\draw[] (\x-1,\y+4) -- (\x-2.5,\y-0);	
				\draw[] (\x-1,\y+4) -- (\x-0,\y-0);
				\draw[] (\x+1,\y+4) -- (\x-0,\y-0);
				\draw[] (\x+1,\y+4) -- (\x+2.5,\y-0);
				\draw[] (\x+1,\y+4) -- (\x+5,\y-0);	
				\draw[] (\x+3,\y+4) -- (\x+5,\y-0);		
				\draw[] (\x+5,\y+4) -- (\x+5,\y-0);

				\foreach \t in {-5,-3,-1,1,3,5}
				{
					\fill (\x+\t,\y+4) circle (.1);
				}
				
				\foreach \t in {-5,-2.5,0,2.5,5}
				{
					\fill (\x+\t,\y-0) circle (.1);
				}

				\fill (\x-5,\y+4) node [above] {$1_{\partial'}$};
				\fill (\x-3,\y+4) node [above] {$2_{\partial'}$};
				\fill (\x-1,\y+4) node [above] {$1_{\partial''}$};
				\fill (\x+1,\y+4) node [above] {$1_{\partial''}$};
				\fill (\x+3,\y+4) node [above] {$2_{\partial'}$};
				\fill (\x+5,\y+4) node [above] {$1_{\partial'}$};

				\fill (\x-5,\y-0) node [below] {$1_{\partial}$};
				\fill (\x-2.5,\y-0) node [below] {$2_{\partial}$};
				\fill (\x+0,\y-0) node [below] {$1_{\partial'''}$};
				\fill (\x+2.5,\y-0) node [below] {$3_{\partial}$};
				\fill (\x+5,\y-0) node [below] {$1_{\partial}$};	
				
				 \fill (\x-5,\y+2) node [left] {\tiny{$d_1$}};
				 \fill (\x-3,\y+1) node [left] {\tiny{$d_2$}};
				 \fill (\x-2.75,\y+3) node [left] {\tiny{$d_3$}};
				 \fill (\x-1.25,\y+3) node [left] {\tiny{$d_4$}};
				  \fill (\x-0,\y+1) node [left] {\tiny{$d_{12}$}};
				 \fill (\x+.85,\y+3) node [left] {\tiny{$d_{13}$}};
				 \fill (\x+2.2,\y+1) node [left] {\tiny{$d_8$}};
				 \fill (\x+3.95,\y+1) node [left] {\tiny{$d_9$}};
				  \fill (\x+3.75,\y+3) node [left] {\tiny{$d_{10}$}};
				  \fill (\x+5.2,\y+2) node [left] {\tiny{$d_1$}};
			}
		}

\end{tikzpicture}	
\end{center}

From here, we perform a Coxeter transformation in the usual manner. We consider the associated quiver to find an admissible ordering, and then perform a sequence of flips. 

\end{example}

We can now reach the asymptotic triangulations using limits of Coxeter transformations. This is discussed in Section 6.

\section{Limits}

In this section we aim to show that the Coxeter transformation and the Dehn twist behave the same way in the limit. This allows us to use whichever process of obtaining an asymptotic triangulation that is the most useful in our setting. 

Recall from Definition \ref{cox:def} that if $T$ is a triangulation with peripheral arcs, and $\{\mu_{\alpha_i}\}_{i \in I}$ is a finite sequence of flips with $\alpha_i$ peripheral, so that $\tilde{T} = \mu_\alpha T$ consists only of bridging arcs, then the Coxeter transformation of $T$ is:
\[
\cox(T) = (\mu_\alpha)^{-1}\cox(\mu_\alpha T).
\]

Using the commutativity relations from Theorem \ref{th:comm}, we have that $\Cox(\tilde{T}) = \cox^m(\tilde{T}) = D^{r+s}_z(\tilde{T})$ for $r,s \in \mathbb{N}$ where $pr = qs = m = \lcm(p,q)$. We have the following proposition:

\begin{proposition} Let $T$ be a bridging triangulation of $C_{p,q}$. Then 
\[
\lim_{n\to\infty}\Cox^{n}(T) = D^{+\infty}_z(T).
\]
\label{cox_inf:thm}
\end{proposition}

The proof follows from the definitions of $\Cox$ and $D_z^{+\infty}$, cf. Section 4.

If $T$ is a bridging triangulation, we define $$\Cox^{+\infty}(T) = \lim_{n \to \infty}\Cox^n(T).$$ By Proposition \ref{cox_inf:thm}, $\Cox^{+\infty}(T)$ is an asymptotic triangulation of $C_{p,q}$.

\subsection{Quivers}

As described in Section 1, a Coxeter transformation on a quiver is a sequence of reflections from sources to sinks. 

We have already defined the limits of the Dehn twist and Coxeter transformation on a bridging triangulation $T $. Now we want to see what happens to the quiver $Q_T$ under these transformations. It is well known that a quiver is fixed under the action of the mapping class group on $C_{p,q}$. However, the quiver behaves differently in the limit. We saw that the quiver $Q_{D_z^{+\infty}(T)}$ becomes disconnected. We now give an algorithm to obtain the quiver  $Q_{\Cox^{+\infty}(T)}$ directly from $Q_T$ without passing through the triangulations involved.

Let $Q = Q_T$ be a quiver associated to a bridging triangulation $T$. Take two copies of $Q$, draw them as planar graphs (as un-oriented cycles) with the vertices of each quiver labeled in a clockwise manner. Consider a maximal counter-clockwise path $P$ in $Q$. For there to be such a path $P= i \to \cdots \to j$ in $Q$, the corresponding arcs $d_i, \ldots, d_j \in T$ must share an endpoint on $\partial$. In the limit $\Cox^{+\infty}(T)$, the arcs involved collapse to a single Pr\"ufer arc based on $\partial$ (cf. Figure \ref{contract_d}). We also consider maximal clockwise paths in $Q_T$.

\begin{algorithm}
\caption{Constructing $Q_\partial$, $Q_{\partial'}$ from $Q$.}
  \begin{algorithmic}[1]
  \STATE Replace every maximal counter-clockwise path $P = i \to \cdots \to j$ in $Q$ by a single vertex $w_{i,j}$. Denote the resulting quiver by $Q_\partial$.
  \STATE Replace every maximal clockwise path $P = r \to \cdots \to s$ in $Q$ by a single vertex $u_{r,s}$. Denote the resulting quiver by $Q_{\partial'}$.
  \end{algorithmic}
\label{algo1}
\end{algorithm}

By construction, $Q_\partial$ only has arrows forming a clockwise cycle, and $Q_{\partial'}$ only has arrows forming a counter-clockwise cycle.

\begin{figure}
\centering
\subfigure{
	\begin{tikzpicture}[scale = .3]
		\tikzstyle{every node} = [font = \small]
		\foreach \x in {0}
		{
			\foreach \y in {-8}
			{
				\draw[<-] (\x-5,\y+4) -- (\x+6,\y+4);
				\fill (\x-6,\y+4) node {$\partial'$};
				\draw[->] (\x-5,\y-4) -- (\x+6,\y-4);
				\fill (\x-6,\y-4) node {$\partial$};

				\foreach \t in {-4,-2.5,-1,0.5,2,3.5,5}
				{
					\fill (\x+\t,\y+4) circle (.1);
					\fill (\x+\t,\y-4) circle (.1);
				}

				\draw(\x-4,\y) -- (\x-4,\y-4);	
				\draw[gray](\x-2,\y) -- (\x-4,\y-4);
				\draw[gray] (\x+.75,\y-.5) -- (\x-4,\y-4);
				\draw(\x+3,\y-1.5) -- (\x-4,\y-4);
				
				\draw[->] (\x-3.75,\y-1) -- (\x-3,\y-1);
				\draw[->] [gray](\x-2.25,\y-1) -- (\x-1.25,\y-1.5);
				\draw[->] (\x-.65,\y-1.75) -- (\x-.25,\y-2.35);
				\fill (\x-4,\y-1) node [left] {\tiny{$d_i$}};
				\fill (\x,\y-3) node [right] {\tiny{$d_j$}};
				
				\fill (\x-4,\y-4) node [below] {\tiny{$k_{\partial}$}};
				
					
				\draw(\x+5,\y+4) -- (\x+3.5,\y+1);	
				\draw(\x+5,\y+4) -- (\x+1.5,\y+1.5);
				
				\draw[->] (\x+3.6,\y+1.7) -- (\x+2.65,\y+2);
				\fill (\x+2.65,\y+2.75) node [left] {\tiny{$d_s$}};
				\fill (\x+4.15,\y+2) node [right] {\tiny{$d_r$}};
				\fill (\x+5,\y+4) node [above] {\tiny{$t_{\partial'}$}};				
			}
		}
		
		\draw[->] (8,-8) -- (10,-8);
		\end{tikzpicture}
		}
		\quad
		\subfigure{
			\begin{tikzpicture}[scale = .3]
			\tikzstyle{every node} = [font = \small]
			\foreach \x in {0}
			{
				\foreach \y in {-8}
				{
					\draw[<-] (\x-5,\y+4) -- (\x+6,\y+4);
					\fill (\x-6,\y+4) node {$\partial'$};
					\draw[->] (\x-5,\y-4) -- (\x+6,\y-4);
					\fill (\x-6,\y-4) node {$\partial$};

					\foreach \t in {-4,-2.5,-1,0.5,2,3.5,5}
					{
						\fill (\x+\t,\y+4) circle (.1);
						\fill (\x+\t,\y-4) circle (.1);
					}

				\Prufer{\x-4}{\y-4}{9}{}
				\fill (\x+1,\y+2.75) node [below] {\tiny{$\pi_{t_{\partial'}}$}};
				\hPrufer{\x+5}{\y+4}{9}{}
				\fill (\x-1,\y-1.25) node [below] {\tiny{$\pi_{k_{\partial}}$}};

				\fill (\x+5,\y+4) node [above] {\tiny{$t_{\partial'}$}};
				\fill (\x-4,\y-4) node [below] {\tiny{$k_{\partial}$}};
				
			}
		}
		
		\end{tikzpicture}
		}
\caption{Coxeter transformation on a quiver}
\label{contract_d}
\end{figure}
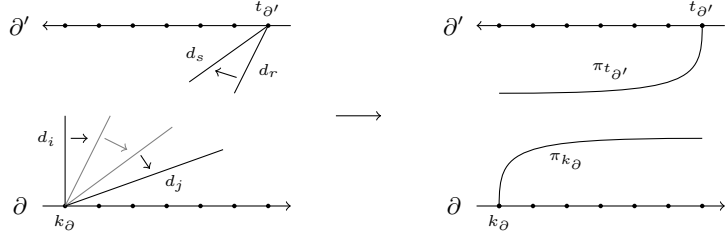

\begin{proposition}
Let $T$ be a bridging triangulation of $C_{p,q}$, $Q_T$ its associated quiver, and $Q_\partial$, $Q_{\partial'}$ as constructed above. Then $$Q_{D_z^{+\infty}(T)} \cong Q_\partial \sqcup Q_{\partial'}.$$
\end{proposition}

\begin{proof}
The quivers $Q_\partial$ and $Q_{\partial'}$ are constructed as above from a bridging triangulation of $C_{p,q}$. Thus the quiver $Q_\partial$ will be a clockwise cycle on $p$ vertices and the quiver $Q_{\partial'}$ will be a counter-clockwise cycle on $q$ vertices. 
Since $T$ is a bridging triangulation, $D^{+\infty}_z(T)$ will have $p$ Pr\"ufer arcs stemming from the outer boundary $\partial$, and $q$ Pr\"ufer arcs stemming from the inner boundary $\partial'$. So $Q_{D_z^{+\infty}(T)}$ has two connected components, one of which is a clockwise cycle on $p$ vertices, and one component is a counter-clockwise cycle on $q$ vertices. This is exactly $Q_\partial$ and $Q_{\partial'}$, and so we have the isomorphism $Q_{D_z^{+\infty}(T)} \cong Q_\partial \sqcup Q_{\partial'}.$
\end{proof}

\begin{corollary}
Let $T$ be a triangulation, and $Q_T$ its associated quiver. Then $Q_{D^{+\infty}(T)} \cong Q_{D^{-\infty}(T)}$.
\end{corollary}

\begin{proof}
Consider the two triangulations $D_z^{+\infty}(T)$ and $D_z^{-\infty}(T)$ obtained from a triangulation $T$. All the strictly asymptotic arcs of $D^{+\infty}(T)$ are Pr\"ufer arcs, and all the strictly asymptotic arcs of $D^{-\infty}(T)$ are adic arcs. Now consider the associated quivers. Recall that every arc in a triangulation corresponds to a vertex in the quiver, and there is an arrow between two vertices $i \to j$ in the quiver if the corresponding arc $d_i$ can be rotated clockwise to become an arc isotopic to the corresponding arc $d_j$. Strictly asymptotic arcs have one endpoint a marked point on the boundary of $C_{p,q}$, and the other endpoint spirals infinitely around $z$. Thus for a strictly asymptotic arc $d_j$ to be a clockwise rotation of $d_i$, the endpoint of $d_i$ that is rotated is the one on the boundary. All arrows will go from left to right between strictly asymptotic arcs on the upper boundary, and all arrows between strictly asymptotic arcs will go from right to left on the lower boundary (cf. Fig. \ref{pos_neg}), independent of whether the arcs spiral positively or negatively around $z$. Peripheral arcs are unaffected by the Dehn twist, and therefore the internal triangles of $T$ stay fixed, and the cycles in the quivers corresponding to the internal triangles will be the same for both $Q_{D_z^{+\infty}}(T)$ and $Q_{D_z^{-\infty}(T)}$.

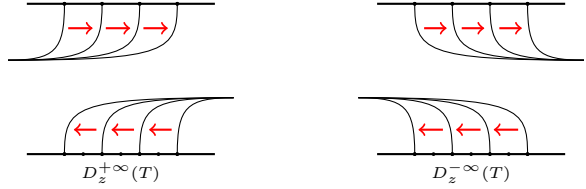
\begin{figure}[h!]
\subfigure{
\begin{tikzpicture}[xscale = .25, yscale = .25]
	\foreach \x in {0}
					{
						\foreach \y in {0}
						{
							\coordinate (H0) at (\x-5,\y+4);
							\coordinate (H1) at (\x-3,\y+4);
							\coordinate (H2) at (\x-1,\y+4);
							\coordinate (H3) at (\x+1,\y+4);
							\coordinate (H4) at (\x+3,\y+4);

							\draw[thick] (\x-5,\y+4) -- (\x+5,\y+4);
							\fill (H1) circle (.1);
							\fill (H2) circle (.1);
							\fill (H3) circle (.1);
							\fill (H4) circle (.1);

							\coordinate (B-1) at (\x-5,\y-4);
							\coordinate (B0) at (\x-4,\y-4);
							\coordinate (B1) at (\x-3,\y-4);
							\coordinate (B2) at (\x-2,\y-4);
							\coordinate (B3) at (\x-1,\y-4);
							\coordinate (B4) at (\x,\y-4);
							\coordinate (B5) at (\x+1,\y-4);
							\coordinate (B6) at (\x+2,\y-4);
							\coordinate (B7) at (\x+3,\y-4);

							\draw[thick] (\x-5,\y-4) -- (\x+5,\y-4);
							\fill (B1) circle (.1);
							\fill (B2) circle (.1);
							\fill (B3) circle (.1);
							\fill (B4) circle (.1);
							\fill (B5) circle (.1);
							\fill (B6) circle (.1);
							\fill (B7) circle (.1);
							
							\hPrufer{\x-3}{\y+4}{3}{}
							\hPrufer{\x-1}{\y+4}{5}{}
							\hPrufer{\x+1}{\y+4}{7}{}
							\hPrufer{\x+3}{\y+4}{9}{}

							\Prufer{\x-3}{\y-4}{9}{}
							\Prufer{\x-1}{\y-4}{7}{}
							\Prufer{\x+1}{\y-4}{5}{}
							\Prufer{\x+3}{\y-4}{3}{}

							\fill (\x,\y-5) node {\tiny{$D_z^{+\infty}(T)$}};
							
							\draw[red, thick, <-] (\x+2.5, \y+2.7) -- (\x+1.2, \y +2.7);
							\draw[red, thick, <-] (\x+.5, \y+2.7) -- (\x-.8, \y+2.7);
							\draw[red, thick, <-] (\x-1.5, \y+2.7) -- (\x-2.8, \y+2.7);
							
							\draw[red, thick, ->] (\x+2.7, \y-2.7) -- (\x+1.5, \y -2.7);
							\draw[red, thick, ->] (\x+.7, \y-2.7) -- (\x-.5, \y-2.7);
							\draw[red, thick, ->] (\x-1.3, \y-2.7) -- (\x-2.5, \y-2.7);
						}
					}

\end{tikzpicture}}
\quad \quad \quad \quad
\subfigure{
\begin{tikzpicture}[xscale = .25, yscale = .25]
	\foreach \x in {0}
					{
						\foreach \y in {0}
						{
							\coordinate (H0) at (\x-5,\y+4);
							\coordinate (H1) at (\x-3,\y+4);
							\coordinate (H2) at (\x-1,\y+4);
							\coordinate (H3) at (\x+1,\y+4);
							\coordinate (H4) at (\x+3,\y+4);

							\draw[thick] (\x-5,\y+4) -- (\x+5,\y+4);
							\fill (H1) circle (.1);
							\fill (H2) circle (.1);
							\fill (H3) circle (.1);
							\fill (H4) circle (.1);

							\coordinate (B-1) at (\x-5,\y-4);
							\coordinate (B0) at (\x-4,\y-4);
							\coordinate (B1) at (\x-3,\y-4);
							\coordinate (B2) at (\x-2,\y-4);
							\coordinate (B3) at (\x-1,\y-4);
							\coordinate (B4) at (\x,\y-4);
							\coordinate (B5) at (\x+1,\y-4);
							\coordinate (B6) at (\x+2,\y-4);
							\coordinate (B7) at (\x+3,\y-4);

							\draw[thick] (\x-5,\y-4) -- (\x+5,\y-4);
							\fill (B1) circle (.1);
							\fill (B2) circle (.1);
							\fill (B3) circle (.1);
							\fill (B4) circle (.1);
							\fill (B5) circle (.1);
							\fill (B6) circle (.1);
							\fill (B7) circle (.1);
							
							\hadic{\x-3}{\y+4}{9}{}
							\hadic{\x-1}{\y+4}{7}{}
							\hadic{\x+1}{\y+4}{5}{}
							\hadic{\x+3}{\y+4}{3}{}

							\adic{\x-3}{\y-4}{3}{}
							\adic{\x-1}{\y-4}{5}{}
							\adic{\x+1}{\y-4}{7}{}
							\adic{\x+3}{\y-4}{9}{}

							\fill (\x,\y-5) node {\tiny{$D_z^{-\infty} (T)$}};
							
							\draw[red, thick, <-] (\x+2.7, \y+2.7) -- (\x+1.5, \y +2.7);
							\draw[red, thick, <-] (\x+.7, \y+2.7) -- (\x-.5, \y+2.7);
							\draw[red, thick, <-] (\x-1.3, \y+2.7) -- (\x-2.5, \y+2.7);
							
							\draw[red, thick, ->] (\x+2.5, \y-2.7) -- (\x+1.2, \y -2.7);
							\draw[red, thick, ->] (\x+.5, \y-2.7) -- (\x-.8, \y-2.7);
							\draw[red, thick, ->] (\x-1.5, \y-2.7) -- (\x-2.8, \y-2.7);
						}
					
					}

\end{tikzpicture}
}
\caption{$Q_{D^{+\infty}(T)}$ and $Q_{D^{-\infty}(T)}$.}
\label{pos_neg}
\end{figure}
\label{Qpos_Qneg}
\end{proof}

To see an example of how these paths of arcs contract in the triangulation, see Figure \ref{contract_d}. The following example shows how we contract the paths in $Q$ to get two quivers $Q_\partial$, $Q_{\partial'}$. 

\begin{example}
Consider a frame of the following (bridging) triangulation $T$ and two copies of the associated quiver $Q_T$:

\begin{figure}[h!]
\centering
	\begin{tikzpicture}[xscale = .55, yscale = .35]
		\tikzstyle{every node} = [font = \small]
		\foreach \x in {0}
		{
			\foreach \y in {-8}
			{
				\draw[<-] (\x-7,\y+4) -- (\x+1,\y+4);
				\fill (\x-8,\y+4) node {$\partial'$};
				\draw[->] (\x-7,\y-3) -- (\x+1,\y-3);
				\fill (\x-8,\y-3) node {$\partial$};

				\foreach \t in {-6,-4,-2,0}
				{
					\fill (\x+\t,\y+4) circle (.1);
					\fill (\x+\t,\y-3) circle (.1);
				}
				
				\fill (\x-6,\y+4) node [above] {$0_{\partial'}$};
				\fill (\x-4,\y+4) node [above] {$2_{\partial'}$};
				\fill (\x-2,\y+4) node [above] {$1_{\partial'}$};
				\fill (\x+0,\y+4) node [above] {$0_{\partial'}$};

				\fill (\x-6,\y-3) node [below] {$0_{\partial}$};
				\fill (\x-4,\y-3) node [below] {$1_{\partial}$};
				\fill (\x-2,\y-3) node [below] {$2_{\partial}$};
				\fill (\x+0,\y-3) node [below] {$0_{\partial}$};

				\fill[] (\x-5.8,\y-1.5) node [left] {$d_1$};
				\fill[] (\x-4.3,\y-1.5) node [left] {$d_2$};
				\fill[] (\x-3.8,\y+2) node [left] {$d_3$};
				\fill[] (\x-2.3,\y-1.5) node [left] {$d_4$};
				\fill[] (\x-1.8,\y+2) node [left] {$d_5$};
				\fill[] (\x-0.5,\y+2) node [left] {$d_6$};
				\fill[] (\x+.2,\y-1.5) node [left] {$d_1$};

				\draw[] (\x-6,\y+4) -- (\x-6,\y-3);
				\draw[] (\x-6,\y+4) -- (\x-4,\y-3);
				\draw[] (\x-4,\y+4) -- (\x-4,\y-3);	
				\draw[] (\x-4,\y+4) -- (\x-2,\y-3);
				\draw[] (\x-2,\y+4) -- (\x-2,\y-3);
				\draw[] (\x-0,\y+4) -- (\x-2,\y-3);
				\draw[] (\x-0,\y+4) -- (\x-0,\y-3);

			}
		}
		\end{tikzpicture}
\end{figure}
\begin{figure}[h!]
\subfigure{
	\begin{tikzpicture}[scale = .25]
		\tikzstyle{every node} = [font = \small]
		\foreach \x in {0}
		{
			\foreach \y in {-8}
			{
			\fill (\x-10,\y+3) node {$Q_1:$};
			\fill (\x-6,\y) circle (.15);
			\fill (\x-6,\y) node [left] {$1$};
			\fill (\x-4,\y-3) circle (.15);
			\fill (\x-4,\y-3) node [below] {$6 $};
			\fill(\x,\y-3) circle (.15);
			\fill (\x,\y-3) node [below] {$ 5 $};
			\fill(\x+2,\y) circle (.15);
			\fill (\x+2,\y) node [right] {$4 $};
			\fill(\x,\y+3) circle (.15);
			\fill (\x,\y+3) node [above] {$ 3 $};
			\fill (\x-4,\y+3) circle (.15);
			\fill (\x-4,\y+3) node [above] {$2 $};

			\draw [->] (\x-6,\y)--(\x-4,\y-3); 
			\draw [->] (\x,\y-3) -- (\x-4,\y-3); 
			\draw [->] (\x+2,\y)--(\x,\y-3); 
			\draw [<-] (\x,\y+3) -- (\x+2,\y); 
			\draw [->] (\x-4,\y+3)--(\x,\y+3); 
			\draw [<-] (\x-6,\y) -- (\x-4,\y+3); 

			}
		}
	\end{tikzpicture}
	}
	\quad
	\subfigure{
		\begin{tikzpicture}[scale = .25]
		\tikzstyle{every node} = [font = \small]
		\foreach \x in {0}
		{
			\foreach \y in {-8}
			{
			\fill (\x-10,\y+3) node {$Q_2:$};
			\fill (\x-6,\y) circle (.15);
			\fill (\x-6,\y) node [left] {$1$};
			\fill (\x-4,\y-3) circle (.15);
			\fill (\x-4,\y-3) node [below] {$6 $};
			\fill(\x,\y-3) circle (.15);
			\fill (\x,\y-3) node [below] {$ 5 $};
			\fill(\x+2,\y) circle (.15);
			\fill (\x+2,\y) node [right] {$4 $};
			\fill(\x,\y+3) circle (.15);
			\fill (\x,\y+3) node [above] {$ 3 $};
			\fill (\x-4,\y+3) circle (.15);
			\fill (\x-4,\y+3) node [above] {$2 $};

			\draw [->] (\x-6,\y)--(\x-4,\y-3); 
			\draw [->] (\x,\y-3) -- (\x-4,\y-3); 
			\draw [->] (\x+2,\y)--(\x,\y-3); 
			\draw [<-] (\x,\y+3) -- (\x+2,\y); 
			\draw [->] (\x-4,\y+3)--(\x,\y+3); 
			\draw [<-] (\x-6,\y) -- (\x-4,\y+3); 

			}
		}
	\end{tikzpicture}
	}
	\end{figure} 
	
	\medskip
	
	As described, we consider all maximal clockwise paths in $Q_1$, and all maximal counter-clockwise paths in $Q_2$ and contract them to a single vertex:
	
	\begin{figure}[h!]
\centering
\subfigure{
	\begin{tikzpicture}[scale = .25]
		\tikzstyle{every node} = [font = \small]
		\foreach \x in {0}
		{
			\foreach \y in {-8}
			{
			\fill (\x-10,\y+3) node {$Q_1:$};
			\fill (\x-6,\y) circle (.15);
			\fill (\x-6,\y) node [left] {$1$};
			\fill (\x-4,\y-3) circle (.15);
			\fill (\x-4,\y-3) node [below] {$6 $};
			\fill(\x,\y-3) circle (.15);
			\fill (\x,\y-3) node [below] {$ 5 $};
			\fill(\x+2,\y) circle (.15);
			\fill (\x+2,\y) node [right] {$4 $};
			\fill(\x,\y+3) circle (.15);
			\fill (\x,\y+3) node [above] {$ 3 $};
			\fill (\x-4,\y+3) circle (.15);
			\fill (\x-4,\y+3) node [above] {$2 $};
			
			\draw [->] (\x-6,\y)--(\x-4,\y-3); 
			\draw [->] (\x,\y-3) -- (\x-4,\y-3); 
			\draw [->] (\x+2,\y)--(\x,\y-3); 
			\draw [<-] (\x,\y+3) -- (\x+2,\y); 
			\draw [->] (\x-4,\y+3)--(\x,\y+3); 
			\draw [<-] (\x-6,\y) -- (\x-4,\y+3); 

			\draw[red] (\x-2,\y+3) ellipse (2.5cm and .5cm);
			\draw[red] plot [smooth cycle] coordinates {(\x+2.35,\y+.25) (\x,\y-3.25) (\x-4.25,\y-3.25) (\x-4.25,\y-2.75) (\x-.15,\y-2.5) (\x+1.75,\y+.25)};
			
	}
		}
	\end{tikzpicture}
	}
	\quad
	\subfigure{
		\begin{tikzpicture}[scale = .25]
		\tikzstyle{every node} = [font = \small]
		\foreach \x in {0}
		{
			\foreach \y in {-8}
			{
			\fill (\x-10,\y+3) node {$Q_2:$};
			\fill (\x-6,\y) circle (.15);
			\fill (\x-6,\y) node [left] {$1$};
			\fill (\x-4,\y-3) circle (.15);
			\fill (\x-4,\y-3) node [below] {$6 $};
			\fill(\x,\y-3) circle (.15);
			\fill (\x,\y-3) node [below] {$ 5 $};
			\fill(\x+2,\y) circle (.15);
			\fill (\x+2,\y) node [right] {$4 $};
			\fill(\x,\y+3) circle (.15);
			\fill (\x,\y+3) node [above] {$ 3 $};
			\fill (\x-4,\y+3) circle (.15);
			\fill (\x-4,\y+3) node [above] {$2 $};

			\draw [->] (\x-6,\y)--(\x-4,\y-3); 
			\draw [->] (\x,\y-3) -- (\x-4,\y-3); 
			\draw [->] (\x+2,\y)--(\x,\y-3); 
			\draw [<-] (\x,\y+3) -- (\x+2,\y); 
			\draw [->] (\x-4,\y+3)--(\x,\y+3); 
			\draw [<-] (\x-6,\y) -- (\x-4,\y+3); 

			\draw[rotate around={-55:(\x+1,\y+1.5)},red] (\x+1,\y+1.5) ellipse (2.5cm and .5cm);
			\draw[red] plot [smooth cycle] coordinates {(\x-4.25,\y+3.25) (\x-6.25,\y) (\x-4.25,\y-3.25) (\x-3.75,\y-2.75) 					(\x-5.5,\y-.25) (\x-3.75,\y+2.75)};
			}
		}
	\end{tikzpicture}
	}
	\end{figure}

\begin{figure}[h!]
\centering
\subfigure{
	\begin{tikzpicture}[scale = .25]
		\tikzstyle{every node} = [font = \small]
		\foreach \x in {0}
		{
			\foreach \y in {-8}
			{
			\fill (\x-10,\y+4) node {$Q_{\partial'}:$};
			\fill (\x-6,\y) circle (.15);
			\fill (\x-6,\y) node [left] {${0_{\partial}} = 1$};
			\fill[red] (\x,\y-3) circle (.15);
			\fill (\x,\y-3) node [below] {$u_{4,6}$};
			\fill [red] (\x-2,\y+3) circle (.15);
			\fill (\x-2,\y+3) node [above] {$u_{2,3}$};

			\draw [->] (\x-6,\y) -- (\x,\y-3); 
			\draw [->] (\x,\y-3)--(\x-2,\y+3); 
			\draw [<-] (\x-6,\y) -- (\x-2,\y+3); 
			
			}
		}
	\end{tikzpicture}
	}
	\quad
	\subfigure{
		\begin{tikzpicture}[scale = .25]
		\tikzstyle{every node} = [font = \small]
		\foreach \x in {0}
		{
			\foreach \y in {-8}
			{
			\fill (\x-10,\y+5) node {$Q_\partial:$};
			\fill[red] (\x-6,\y) circle (.15);
			\fill (\x-6,\y) node [left] {$w_{2,6}$};
			\fill(\x,\y-3) circle (.15);
			\fill (\x,\y-3) node [right] {$5 = {1_{\partial'}}$};
			\fill[red](\x+1,\y+1.5) circle (.15);
			\fill (\x+1,\y+1.5) node [right] {$w_{4,3}$};

			\draw [->] (\x-6,\y)--(\x+1,\y+1.5); 
			\draw [->] (\x,\y-3) -- (\x-6,\y); 
			\draw [->] (\x+1,\y+1.5)--(\x,\y-3); 
			}
		}
	\end{tikzpicture}
	}
	\end{figure} 
	
We can check that this is indeed the quiver of the asymptotic triangulation $D_z^{+\infty}(T)$.

\begin{center}
	\begin{tikzpicture}[scale = .4][h!]
		\tikzstyle{every node} = [font = \small]
		\foreach \x in {0}
		{
			\foreach \y in {-8}
			{
				\draw[<-] (\x-8,\y+4) -- (\x+2,\y+4);
				\fill (\x-9,\y+4) node {$\partial'$};
				\draw[->] (\x-8,\y-3) -- (\x+2,\y-3);
				\fill (\x-9,\y-3) node {$\partial$};

				\foreach \t in {-6,-4,-2,0}
				{
					\fill (\x+\t,\y+4) circle (.1);
					\fill (\x+\t,\y-3) circle (.1);
				}
				
				\fill (\x-6,\y+4) node [above] {$0_{\partial'}$};
				\fill (\x-4,\y+4) node [above] {$2_{\partial'}$};
				\fill (\x-2,\y+4) node [above] {$1_{\partial'}$};
				\fill (\x+0,\y+4) node [above] {$0_{\partial'}$};

				\fill (\x-6,\y-3) node [below] {$0_{\partial}$};
				\fill (\x-4,\y-3) node [below] {$1_{\partial}$};
				\fill (\x-2,\y-3) node [below] {$2_{\partial}$};
				\fill (\x+0,\y-3) node [below] {$0_{\partial}$};
				
				\Prufer{\x-6}{\y-3}{8}{}
				\fill (\x-6,\y-1.5) node [right] {\tiny{$\pi_{0_{\partial}}$}};
				\Prufer{\x-4}{\y-3}{6}{}
				\fill (\x-4,\y-1.5) node [right] {\tiny{$\pi_{1_{\partial}}$}};
				\Prufer{\x-2}{\y-3}{4}{}
				\fill (\x-2,\y-1.5) node [right] {\tiny{$\pi_{2_{\partial}}$}};
				\Prufer{\x-0}{\y-3}{2}{}
				\fill (\x-0,\y-1.5) node [right] {\tiny{$\pi_{0_{\partial}}$}};
				\hPrufer{\x-6}{\y+4}{2}{}
				\fill (\x-6,\y+2.5) node [left] {\tiny{$\pi_{0_{\partial'}}$}};
				\hPrufer{\x-4}{\y+4}{4}{}
				\fill (\x-4,\y+2.5) node [left] {\tiny{$\pi_{2_{\partial'}}$}};
				\hPrufer{\x-2}{\y+4}{6}{}
				\fill (\x-2,\y+2.5) node [left] {\tiny{$\pi_{1_{\partial'}}$}};
				\hPrufer{\x-0}{\y+4}{8}{}
				\fill (\x-0,\y+2.5) node [left] {\tiny{$\pi_{0_{\partial'}}$}};				
			}
		}	
		\end{tikzpicture}
\end{center}
\end{example}

\begin{theorem}
Let $T$ be a bridging triangulation, and $Q_T$ the associated quiver. Then
$$Q_{\Cox^{+\infty}(T)} \cong Q_{D^{+\infty}_z(T)} \cong Q_{D^{-\infty}_z(T)}.$$
\end{theorem}

\begin{proof}
We have that $\Cox^{+\infty}(T) = D^{+\infty}_z(T)$ as triangulations, and thus $Q_{\Cox^{+\infty}(T)} \cong Q_{D^{+\infty}_z(T)}$. The second isomorphism is the result from Corollary \ref{Qpos_Qneg}.
\end{proof}

We have described an algorithm for obtaining a quiver from an asymptotic triangulation. A natural question to ask is whether the algorithm can be used on a quiver when we don't know the associated triangulation. To do this, we need to work with the \emph{shape} $Q^b$, where $Q^b$ is the full subquiver obtained by removing arrows that belong only to internal triangles of $Q_T$. We have the alternate algorithm:

\begin{algorithm}
\caption{Constructing $Q_\partial$, $Q_{\partial'}$ from $Q^b$.}
  \begin{algorithmic}[1]
  \STATE Draw two copies $Q_1, Q_2$ of $Q_T$.
  \STATE  Do 1 \& 2 as in Algorithm \ref{algo1} to paths in $Q^b \cap Q_1$ and $Q^b \cap Q_2$, respectively.
  \STATE Draw result, killing all subgraphs that share an edge with the contracted path $P$ in Step 2 above.
  \end{algorithmic}
\end{algorithm}

\begin{example}

Consider the quiver $Q_T$ and the full subquiver $Q^b$ of $Q_T$.

\begin{figure}[h!]
\centering
\subfigure{\begin{tikzpicture}[scale = .2]
		\tikzstyle{every node} = [font = \small]
		\foreach \x in {0}
		{
			\foreach \y in {-8}
			{
			\fill (\x-10,\y+3) node {$Q_T:$};
			\fill (\x-6,\y) circle (.15);
			\fill (\x-6,\y) node [left] {$2$};
			\fill (\x-3,\y+4) circle (.15);
			\fill (\x-3,\y+4) node [above] {$1 $};
			\fill(\x,\y-4) circle (.15);
			\fill (\x,\y-4) node [below] {$ 3 $};
			\fill(\x+6,\y) circle (.15);
			\fill (\x+6,\y) node [right] {$4 $};
			\fill(\x+3,\y+4) circle (.15);
			\fill (\x+3,\y+4) node [above] {$ 5 $};
			\fill (\x+2,\y-0.5) circle (.15);
			\fill (\x+1,\y-1) node [above] {$6 $};
			\fill (\x-5,\y-4) circle (.15);
			\fill (\x-5,\y-4) node [below] {$7 $};
			
			\draw [<-] (\x+2.7,\y+4)--(\x-2.7,\y+4); 
			\draw [->] (\x-3.3,\y+3.7) -- (\x-5.7,\y+.3); 
			\draw [<-] (\x+5.7,\y+.3)--(\x+3.3,\y+3.7); 
			\draw [->] (\x-.3,\y-3.7) -- (\x-5.7,\y-.3); 
			\draw [->] (\x+.3,\y-3.7)--(\x+5.7,\y-.3); 
			\draw [->] (\x-5.9,\y-.4) -- (\x-5,\y-3.7); 
			\draw [->] (\x-4.7,\y- 4) -- (\x-.3,\y-4); 
			\draw [->] (\x+5.7,\y) -- (\x+2.3,\y-.4); 
			\draw [->] (\x+1.7,\y-.7) -- (\x+.1,\y-3.6); 

			}
		}
	\end{tikzpicture}}
	\quad
	\subfigure{\begin{tikzpicture}[scale = .2]
		\tikzstyle{every node} = [font = \small]
		\foreach \x in {0}
		{
			\foreach \y in {-8}
			{
			\fill (\x-10,\y+3) node {$Q^b:$};
			\fill (\x-6,\y) circle (.15);
			\fill (\x-6,\y) node [left] {$2$};
			\fill (\x-3,\y+4) circle (.15);
			\fill (\x-3,\y+4) node [above] {$1 $};
			\fill(\x,\y-4) circle (.15);
			\fill (\x,\y-4) node [below] {$ 3 $};
			\fill(\x+6,\y) circle (.15);
			\fill (\x+6,\y) node [right] {$4 $};
			\fill(\x+3,\y+4) circle (.15);
			\fill (\x+3,\y+4) node [above] {$ 5 $};
						
			\draw [<-] (\x+2.7,\y+4)--(\x-2.7,\y+4); 
			\draw [->] (\x-3.3,\y+3.7) -- (\x-5.7,\y+.3); 
			\draw [<-] (\x+5.7,\y+.3)--(\x+3.3,\y+3.7); 
			\draw [->] (\x-.3,\y-3.7) -- (\x-5.7,\y-.3); 
			\draw [->] (\x+.3,\y-3.7)--(\x+5.7,\y-.3); 

			}
		}
	\end{tikzpicture}}
	\end{figure}
We first draw two copies $Q_1$ and $Q_2$ of $Q_T$:
	
\begin{figure}[h!]
\centering
\subfigure{\begin{tikzpicture}[scale = .25]
		\tikzstyle{every node} = [font = \small]
		\foreach \x in {0}
		{
			\foreach \y in {-8}
			{
			\fill (\x-10,\y+3) node {$Q_1:$};
			\fill (\x-6,\y) circle (.15);
			\fill (\x-6,\y) node [left] {$2$};
			\fill (\x-3,\y+4) circle (.15);
			\fill (\x-3,\y+4) node [above] {$1 $};
			\fill(\x,\y-4) circle (.15);
			\fill (\x,\y-4) node [below] {$ 3 $};
			\fill(\x+6,\y) circle (.15);
			\fill (\x+6,\y) node [right] {$4 $};
			\fill(\x+3,\y+4) circle (.15);
			\fill (\x+3,\y+4) node [above] {$ 5 $};
			\fill (\x+2,\y-0.5) circle (.15);
			\fill (\x+1,\y-1) node [above] {$6 $};
			\fill (\x-5,\y-4) circle (.15);
			\fill (\x-5,\y-4) node [below] {$7 $};
			
			\draw [<-] (\x+2.7,\y+4)--(\x-2.7,\y+4); 
			\draw [->] (\x-3.3,\y+3.7) -- (\x-5.7,\y+.3); 
			\draw [<-] (\x+5.7,\y+.3)--(\x+3.3,\y+3.7); 
			\draw [->] (\x-.3,\y-3.7) -- (\x-5.7,\y-.3); 
			\draw [->] (\x+.3,\y-3.7)--(\x+5.7,\y-.3); 
			\draw [->] (\x-5.9,\y-.4) -- (\x-5,\y-3.7); 
			\draw [->] (\x-4.7,\y- 4) -- (\x-.3,\y-4); 
			\draw [->] (\x+5.7,\y) -- (\x+2.3,\y-.4); 
			\draw [->] (\x+1.7,\y-.7) -- (\x+.1,\y-3.6); 

			}
		}
	\end{tikzpicture}}
	\quad
	\subfigure{\begin{tikzpicture}[scale = .25]
		\tikzstyle{every node} = [font = \small]
		\foreach \x in {0}
		{
			\foreach \y in {-8}
			{
			\fill (\x-10,\y+3) node {$Q_2:$};
			\fill (\x-6,\y) circle (.15);
			\fill (\x-6,\y) node [left] {$2$};
			\fill (\x-3,\y+4) circle (.15);
			\fill (\x-3,\y+4) node [above] {$1 $};
			\fill(\x,\y-4) circle (.15);
			\fill (\x,\y-4) node [below] {$ 3 $};
			\fill(\x+6,\y) circle (.15);
			\fill (\x+6,\y) node [right] {$4 $};
			\fill(\x+3,\y+4) circle (.15);
			\fill (\x+3,\y+4) node [above] {$ 5 $};
			\fill (\x+2,\y-0.5) circle (.15);
			\fill (\x+1,\y-1) node [above] {$6 $};
			\fill (\x-5,\y-4) circle (.15);
			\fill (\x-5,\y-4) node [below] {$7 $};
			
			\draw [<-] (\x+2.7,\y+4)--(\x-2.7,\y+4); 
			\draw [->] (\x-3.3,\y+3.7) -- (\x-5.7,\y+.3); 
			\draw [<-] (\x+5.7,\y+.3)--(\x+3.3,\y+3.7); 
			\draw [->] (\x-.3,\y-3.7) -- (\x-5.7,\y-.3); 
			\draw [->] (\x+.3,\y-3.7)--(\x+5.7,\y-.3); 
			\draw [->] (\x-5.9,\y-.4) -- (\x-5,\y-3.7); 
			\draw [->] (\x-4.7,\y- 4) -- (\x-.3,\y-4); 
			\draw [->] (\x+5.7,\y) -- (\x+2.3,\y-.4); 
			\draw [->] (\x+1.7,\y-.7) -- (\x+.1,\y-3.6); 

			}
		}
	\end{tikzpicture}}
	\end{figure}
	
	Then we apply Algorithm  \ref{algo1}  to $Q_1 \cap Q^b$ and $Q_2 \cap Q^b$ and kill subgraphs:
	\begin{figure}[h!]
	\medskip
	\centering
	\subfigure{	
		\begin{tikzpicture}[scale = .25]
		\tikzstyle{every node} = [font = \small]
		\foreach \x in {0}
		{
			\foreach \y in {-8}
			{
			\fill (\x-10,\y+3) node {$Q_1:$};
			\fill (\x-6,\y) circle (.15);
			\fill (\x-6,\y) node [left] {$2$};
			\fill (\x-3,\y+4) circle (.15);
			\fill (\x-3,\y+4) node [above] {$1 $};
			\fill(\x,\y-4) circle (.15);
			\fill (\x,\y-4) node [below] {$ 3 $};
			\fill(\x+6,\y) circle (.15);
			\fill (\x+6,\y) node [right] {$4 $};
			\fill(\x+3,\y+4) circle (.15);
			\fill (\x+3,\y+4) node [above] {$ 5 $};
			\fill (\x+2,\y-0.5) circle (.15);
			\fill (\x+1,\y-1) node [above] {$6 $};
			\fill (\x-5,\y-4) circle (.15);
			\fill (\x-5,\y-4) node [below] {$7 $};
			
			\draw [ultra thick, <-] (\x+2.7,\y+4)--(\x-2.7,\y+4); 
			\draw [->] (\x-3.3,\y+3.7) -- (\x-5.7,\y+.3); 
			\draw [ultra thick,<-] (\x+5.7,\y+.3)--(\x+3.3,\y+3.7); 
			\draw [ultra thick,->] (\x-.3,\y-3.7) -- (\x-5.7,\y-.3); 
			\draw [->] (\x+.3,\y-3.7)--(\x+5.7,\y-.3); 
			\draw [->] (\x-5.9,\y-.4) -- (\x-5,\y-3.7); 
			\draw [->] (\x-4.7,\y- 4) -- (\x-.3,\y-4); 
			\draw [->] (\x+5.7,\y) -- (\x+2.3,\y-.4); 
			\draw [->] (\x+1.7,\y-.7) -- (\x+.1,\y-3.6); 

			\draw[red] (\x-5, \y-4) circle (1.2);
			\draw[red] (\x-6.1, \y-2.7) -- (\x-3.8, \y-5.2);
			\draw[red] (\x-6.1, \y-5.2) -- (\x-3.8, \y-2.7);

			}
		}
	\end{tikzpicture}}
	\quad
	\subfigure{\begin{tikzpicture}[scale = .25]
		\tikzstyle{every node} = [font = \small]
		\foreach \x in {0}
		{
			\foreach \y in {-8}
			{
			\fill (\x-10,\y+3) node {$Q_2$};
			\fill (\x-6,\y) circle (.15);
			\fill (\x-6,\y) node [left] {$2$};
			\fill (\x-3,\y+4) circle (.15);
			\fill (\x-3,\y+4) node [above] {$1 $};
			\fill(\x,\y-4) circle (.15);
			\fill (\x,\y-4) node [below] {$ 3 $};
			\fill(\x+6,\y) circle (.15);
			\fill (\x+6,\y) node [right] {$4 $};
			\fill(\x+3,\y+4) circle (.15);
			\fill (\x+3,\y+4) node [above] {$ 5 $};
			\fill (\x+2,\y-0.5) circle (.15);
			\fill (\x+1,\y-1) node [above] {$6 $};
			\fill (\x-5,\y-4) circle (.15);
			\fill (\x-5,\y-4) node [below] {$7 $};
			
			\draw [,<-] (\x+2.7,\y+4)--(\x-2.7,\y+4); 
			\draw [ultra thick,->] (\x-3.3,\y+3.7) -- (\x-5.7,\y+.3); 
			\draw [<-] (\x+5.7,\y+.3)--(\x+3.3,\y+3.7); 
			\draw [->] (\x-.3,\y-3.7) -- (\x-5.7,\y-.3); 
			\draw [ultra thick,->] (\x+.3,\y-3.7)--(\x+5.7,\y-.3); 
			\draw [->] (\x-5.9,\y-.4) -- (\x-5,\y-3.7); 
			\draw [->] (\x-4.7,\y- 4) -- (\x-.3,\y-4); 
			\draw [->] (\x+5.7,\y) -- (\x+2.3,\y-.4); 
			\draw [->] (\x+1.7,\y-.7) -- (\x+.1,\y-3.6); 
			
			\draw[red] (\x+2, \y-.5) circle (1.2);
			\draw[red] (\x+.8, \y+.6) -- (\x+3.2, \y-1.6);
			\draw[red] (\x+.8, \y-1.6) -- (\x+3.2, \y+.6);

			}
		}
	\end{tikzpicture}}
	\end{figure}
	
	The resulting quivers are:
	\begin{figure}[h!]	
	\centering
	\subfigure{
	
			\begin{tikzpicture}[scale = .25]
		\tikzstyle{every node} = [font = \small]
		\foreach \x in {0}
		{
			\foreach \y in {-8}
			{
			\fill (\x-10,\y+3) node {$Q_{\partial'}:$};
			\fill (\x-2,\y+4) circle (.15);
			\fill (\x-2,\y+4) node [above] {$u_{1,4} $};
			\fill(\x-2,\y-4) circle (.15);
			\fill (\x-2,\y-4) node [below] {$ u_{3,2} $};
			\fill (\x-2,\y-0) circle (.15);
			\fill (\x-2,\y-0) node [left] {$6 $};
			
			\draw [->] (\x-2, \y+3.6) -- (\x-2, \y+.4); 
			\draw [->] (\x-2,\y-.4)--(\x-2,\y-3.6); 
			\draw [->] (\x-2.5,\y+3.9) .. controls (\x-5.5, \y +2) and (\x-5.5, \y-2) .. (\x-2.5,\y-3.9); 
			\draw [<-] (\x-1.5,\y+3.9) .. controls (\x+1.5, \y +2) and (\x+1.5, \y-2) .. (\x-1.5,\y-3.9); 

			}
		}
	\end{tikzpicture}}
	\quad
	\subfigure{\begin{tikzpicture}[scale = .25]
		\tikzstyle{every node} = [font = \small]
		\foreach \x in {0}
		{
			\foreach \y in {-8}
			{
			\fill (\x-10,\y+3) node {$Q_{\partial}:$};

			\fill (\x-4,\y+4) circle (.15);
			\fill (\x-4,\y+4) node [above] {$w_{1,2} $};
			\fill (\x-4,\y-4) circle (.15);
			\fill (\x-4,\y-4) node [below] {$7 $};
			\fill (\x+4,\y+4) circle (.15);
			\fill (\x+4,\y+4) node [above] {$5 $};
			\fill (\x+4,\y-4) circle (.15);
			\fill (\x+4,\y-4) node [below] {$w_{3,4}$};
			
			\draw [->] (\x-3.7, \y+4) -- (\x +3.7, \y+4); 
			\draw [->] (\x-4, \y+3.7)-- (\x-4, \y -3.7); 
			\draw [->] (\x -3.7, \y-4) -- (\x +3.7, \y-4); 
			\draw [->] (\x +4, \y+3.7) -- (\x +4, \y-3.7); 
			\draw [<-] (\x-3.8, \y+3.8) -- (\x +3.8, \y-3.8); 
			
			}
		}
	\end{tikzpicture}}
	\end{figure}
	
	An example of a triangulation associated to this quiver is:

	\begin{center}
		\begin{tikzpicture}[xscale = .55, yscale = .35]
		\tikzstyle{every node} = [font = \small]
		\foreach \x in {0}
		{
			\foreach \y in {-8}
			{
				\draw[<-] (\x-7,\y+4) -- (\x+1,\y+4);
				\fill (\x-8,\y+4) node {$\partial'$};
				\draw[->] (\x-7,\y-3) -- (\x+1,\y-3);
				\fill (\x-8,\y-3) node {$\partial$};

				\foreach \t in {-6,-4,-2,0}
				{
					\fill (\x+\t,\y+4) circle (.1);
				}
				
				\fill (\x-6,\y-3) circle (.1);
				\fill (\x-4.5,\y-3) circle (.1);
				\fill (\x-3,\y-3) circle (.1);
				\fill (\x-1.5,\y-3) circle (.1);
				\fill (\x-0,\y-3) circle (.1);
				
				\fill (\x-6,\y+4) node [above] {$0_{\partial'}$};
				\fill (\x-4,\y+4) node [above] {$2_{\partial'}$};
				\fill (\x-2,\y+4) node [above] {$1_{\partial'}$};
				\fill (\x+0,\y+4) node [above] {$0_{\partial'}$};

				\fill (\x-6,\y-3) node [below] {$0_{\partial}$};
				\fill (\x-4.5,\y-3) node [below] {$1_{\partial}$};
				\fill (\x-3,\y-3) node [below] {$2_{\partial}$};
				\fill (\x-1.5,\y-3) node [below] {$3_{\partial}$};
				\fill (\x+0,\y-3) node [below] {$0_{\partial}$};

				\fill[] (\x-5.8,\y+.5) node [left] {\tiny{$d_1$}};
				\fill[] (\x-4.5,\y+2.5) node [left] {\tiny{$d_2$}};
				\fill[] (\x-3.5,\y+0.5) node [right] {\tiny{$d_3$}};
				\fill[] (\x-4.5,\y-1.9) node [above] {\tiny{$d_7$}};

				\fill[] (\x-2.6,\y-2) node [right] {\tiny{$d_4$}};
				\fill[] (\x-2,\y+3) node [below] {\tiny{$d_6$}};
				\fill[] (\x-0.2,\y-1.5) node [left] {\tiny{$d_5$}};
				\fill[] (\x+0,\y+.5) node [right] {\tiny{$d_1$}};

				\draw[] (\x-6,\y+4) -- (\x-6,\y-3);
				\draw[] (\x-4,\y+4) -- (\x-6,\y-3);
				\draw[] (\x-4,\y+4) -- (\x-3,\y-3);
				\draw[] (\x-0,\y+4) -- (\x-3,\y-3);
				\draw[] (\x-0,\y+4) -- (\x-1.5,\y-3);
				\draw[] (\x-0,\y+4) -- (\x-0,\y-3);
				
				\draw[] (\x-6,\y-3) .. controls (\x-5 ,\y-1.5) and (\x-4 , \y-1.5) .. (\x-3, \y-3);
				\draw[] (\x-4,\y+4) .. controls (\x-3, \y+2.5) and (\x-1,\y+2.5) .. (\x-0,\y+4);

			}
		}
		\end{tikzpicture}
		\end{center}
	
\end{example}

It is possible to distinguish the vertices that correspond to bridging (resp. strictly asymptotic) arcs in the associated triangulation from those that correspond to peripheral arcs. Vertices that correspond to bridging (strictly asymptotic) arcs form an un-oriented (clockwise-oriented) cycle $C$ in the quiver. This cycle actually gives us the full subquiver $Q^b$. Vertices that correspond to peripheral arcs lie in counter-clockwise oriented cycles in the quiver, and these counter-clockwise oriented cycle shares an edge with $C$, that is, an edge between two vertices corresponding to bridging (strictly asymptotic) arcs. Bastian, in \cite{B}, denotes these vertices by $z_\alpha$, and describes the quivers $Q_\alpha$ that branch off from $C$.

Using the full subquiver allows us to construct an ``asymptotic quiver" without knowing the triangulation. There are restrictions on the types of quivers for which these algorithms work. The quivers need to be associated to a triangulation of a surface described in this paper. 

\appendix
\section{Quivers with potentials}

In this appendix, we describe an alternate way of performing quiver mutation by using quivers with potentials. The authors Derksen, Weymann, and Zelevinsky developed a mutation theory of quivers in \cite{DWZ} using potentials, which lifts quiver mutation from the combinatorial level to the algebraic level. This provides a representation-theoretic interpretation of quiver mutation and ultimately leads to the notion of mutation of representations of quivers with potentials. For the convenience of the reader, we will recall the necessary background needed (cf. \cite{DWZ}). The definitions and notation in this appendix are taken from \cite{DLF} and \cite{DWZ}.

Let $Q$ be a quiver. For each vertex $i \in Q_0$ we have the path of length 0, denoted by $e_i$. $A^l$ denotes the $\mathbb{C}$-vector space with basis the set of paths of length $l \geq 0$. We use the notation $R = A^0$ and $A = A^1$. Note that $R$ is the vector space with basis the set of length-0 paths ($\dim R = |Q_0|$), and $A$ is the vector space with basis the set of arrows of $Q$. If we define $e_ie_j = \delta_{ij}e_i$, then $R$ becomes a commutative $\mathbb{C}$-algebra. If we define $e_i\alpha = \delta_{i,h(\alpha)}\alpha$ and $\alpha e_i = \delta_{i,t(\alpha)}\alpha$ then $A^{l > 0}$ becomes an $R$-$R$-bimodule for every $l > 0$.

\begin{definition}
The \emph{path algebra of $Q$} is the $\mathbb{C}$-vector space $$R\langle Q \rangle = \bigoplus_{l = 0}^\infty A^l.$$ 
\end{definition}

The path algebra can also be defined as the (graded) tensor algebra, and for each $i,j \in Q_0$, the component $R\langle Q \rangle_{i,j} = e_i R\langle Q \rangle e_j$ is called the \emph{space of paths from $j$ to $i$}.

\begin{definition}
The \emph{complete path algebra of $Q$} is the $\mathbb{C}$-vector space $R\langle \langle Q \rangle \rangle $ consisting of all possibly infinite linear combinations of paths in $Q$, that is: $$R\langle \langle Q \rangle \rangle = \prod_{l = 0}^{\infty} A^l.$$
\end{definition}

$R\langle Q \rangle $ has multiplication induced by concatenation of paths and this multiplication extends naturally to $R\langle \langle Q \rangle  \rangle $. $R\langle Q \rangle $ is a dense subalgebra of $R\langle \langle Q \rangle \rangle $ under the $\mathfrak{m}$-adic topology for $\mathfrak{m}$ the two-sided ideal of $R\langle \langle Q \rangle \rangle $ generated by the arrows of $Q$ . The fundamental system of open neighborhoods of this topology around 0 is given by the powers of the ideal $\mathfrak{m}$.

\subsection{Quiver mutation}

For a quiver $Q$, an $l$-cycle in $Q$ is a path $\alpha_1\alpha_2 \ldots \alpha_l$ with $l > 0$ such that $h(\alpha_1) = t(\alpha_l)$. If $\alpha_1 \alpha_2 \ldots \alpha_l$ is an $l$-cycle in $Q$, then so is $\alpha_i \alpha_{i+1} \ldots \alpha_{i-2} \alpha_{i-1}$ for $i = 2, \ldots, l$ (reducing indices $ \modm l$). We say that $\alpha_i \alpha_{i+1} \ldots \alpha_{i-1}$ can be obtained from $\alpha_1 \alpha_2 \ldots \alpha_l$ by \emph{rotation}.

\begin{definition}
Let $Q$ be a quiver. An element $W$ of $R \langle \langle Q \rangle \rangle$ is called a \emph{potential} if it is a possibly infinite linear combination of cycles of $Q$ such that no two cycles appearing in $W$ with non-zero coefficients can be obtain from each other by rotation. If $W$ is a potential on $Q$, we say that the pair $(Q,W)$ is a quiver with potential, or a QP.
\end{definition}

\begin{definition}
Let $Q, Q'$ be quivers with the same vertex set $Q_0 = Q_0'$. 
\begin{enumerate}
\item Two potentials $W$ and $W'$ on $Q$ are \emph{cyclically equivalent} if $W-W'$ lies in the closure of the vector subspace of $R \langle \langle Q \rangle \rangle$ spanned by all the elements of the form $\alpha_1 \ldots \alpha_l - \alpha_2 \ldots \alpha_l \alpha_1$ with $\alpha_1 \ldots \alpha_l$ a cycle of positive length.

\item We say that two QPs $(Q,W)$ and $(Q',W')$ are \emph{right-equivalent} if there exists a $\mathbb{C}$-algebra isomorphism $\phi : R \langle \langle Q \rangle \rangle \rightarrow R \langle \langle Q' \rangle \rangle $ satisfying $\phi(e_i) = e_i$ $\forall i \in Q_0 = Q_0'$ and such that $\phi(W)$ is cyclically-equivalent to $W'$.

\item For each arrow $\alpha \in Q_1$ and each cycle $\alpha_1 \ldots \alpha_l$ in $Q$ we define the \emph{cyclic derivative} $$\partial_\alpha(\alpha_1 \ldots \alpha_l) = \sum_{k=1}^{l} \delta_{\alpha,\alpha_k} \alpha_{k+1} \cdots \alpha_l \,\alpha_1 \cdots \alpha_{k-1}$$ and extend $\partial_\alpha$ by linearity and continuity so that $\partial_\alpha(W)$ is defined for every potential $W$.

\item The \emph{Jacobian ideal} $J(W)$ is the topological closure of the two-sided ideal of $R \langle \langle Q \rangle \rangle$ generated by $\{ \partial_\alpha(W) | \alpha \in Q_1 \}$, and the \emph{Jacobian algebra} $P(Q,W)$ is the quotient algebra $R \langle \langle Q \rangle \rangle / J(W)$.

\item A $QP$ is \emph{trivial} if $W \in A^2$ and $\{ \partial_\alpha(W) | \alpha \in Q_1 \}$ spans $A$ as a $\mathbb{C}$-vector space.

\item A QP is \emph{reduced} if the degree-2 component of $W$ is 0, that is, if the expression of $W$ involves no 2-cycles.

\item The \emph{direct sum} $Q \oplus Q'$ is the quiver whose vertex set is $Q_0 = Q_0'$ and whose arrow set is the disjoint union $Q_1 \sqcup Q_1'$.

\item The \emph{direct sum of two QPs} $(Q,W)$ and $(Q',W')$ is the QP $(Q,W) \oplus (Q',W') = (Q \oplus Q', W+W')$.
\end{enumerate}
\end{definition}

The following proposition then follows:

\begin{proposition}
If $\varphi: R \langle \langle Q \rangle \rangle \rightarrow R \langle \langle Q' \rangle \rangle$ is a right-equivalence between $(Q,W)$ and $(Q',S')$, then $\varphi$ sends $J(W)$ onto $J(W')$ and therefore induces an algebra isomorphism $P(Q,W) \rightarrow P(Q',W')$.
\end{proposition}

\begin{theorem}[Splitting theorem \cite{DWZ}]
For every QP $(Q,W)$ there exist a trivial QP $(Q_{triv},W_{triv})$ and a reduced QP $(Q_{red}, W_{red})$ such that $(Q,W)$ is right-equivalent to the direct sum $(Q_{triv},W_{triv}) \oplus (Q_{red}, W_{red})$. The right-equivalence class of each of the QPs $(Q_{triv},W_{triv})$ and $(Q_{red}, W_{red})$ is determined by the right-equivalence class of $(Q,W)$.
\end{theorem}

We will now discuss mutations of quivers with potentials. Let $(Q,W)$ be a QP on the vertex set $Q_0$, and let $i \in Q_0$. We have no restrictions on $Q$, so it is possible that $Q$ has a loop or 2-cycle incident to $i$. We can replace $W$ with a cyclically equivalent potential, where none of the cyclic paths of length greater than 1 in the expression of $W$ begin at $i$. We can now define the potential $[W]$ on $Q$ as the potential obtained from $W$ by replacing every length-2 path $\alpha \beta$ passing through $i$ with the arrow $[\alpha \beta]$. We also define $\Delta_i(Q) = \sum \beta^*\alpha^*[\alpha \beta]$, where the sum runs over all length-2 paths $\alpha \beta$ through $i$. Now we set $\tilde{\mu}_i(W) = [W] + \Delta_i(Q)$, which is a potential on $\tilde{\mu}_i(Q)$, the quiver obtained by applying the first two steps of quiver mutation as in Def. \ref{def:mutation}.

\begin{definition} The mutation $\mu_i(Q,W)$ of $(Q,W)$ with respect to $i$ is defined as the reduced part of the QP $\tilde{\mu}_i(Q, W) = (\tilde{\mu}_i(Q), \tilde{\mu}_i(W))$.
\end{definition}

It's important to note that the underlying quiver of a mutated QP is not necessarily 2-acyclic. The potential determines whether or not we keep 2-cycles.

\subsection{Potential of a triangulation}

If we have two triangulations related by a flip, we know that the associated quivers are related by the corresponding quiver mutation. We want to lift this to the level of QPs, and see if the associated QPs are also related by a QP-mutation.

Let $T$ be a triangulation of a marked surface $(S,M)$, possible with punctures. Then the associated quiver has two types of oriented cycles: cycles arising from internal triangles of $T$, and simple cycles (cycles without repeated arrows) surrounding punctures. As before, we only consider cyclic equivalence classes of cycles.

In this generality (i.e. allowing punctures), Labardini-Fragoso \cite{DLF} provides the following definition of a potential associated to a triangulation.

\begin{definition}
Let $T$ be a triangulation of a marked surface $(S,M)$. The potential $W_T$ associated to $T$ is the potential on $Q_T$ that results from adding all the 3-cycles that arise from internal triangles of $T$, and all the simple cycles that surround the punctures of $(S,M)$.
\label{def:pot}
\end{definition}

In our situation with asymptotic triangulations of the annulus, we have no punctures, so the definition of a potential looks as follows:
\[
W_T = \Sigma \mbox{ internal 3-cycles}.
\]
Note that cycles between strictly asymptotic arcs do not appear in the potential.

\begin{example}
Let $T$ be the following triangulation of the punctured disk $D_4$:
\begin{center}
\begin{tikzpicture}[scale = 1]
		\tikzstyle{every node} = [font = \tiny]
		\foreach \x in {0}
		{
			\foreach \y in {0}
			{
			\draw[thick] (\x+0,\y+0) circle (1.5);
			\fill(\x,\y+0) circle (.07);		

			\fill(\x+1.5,\y) circle (.05);
			\draw[thick] (\x,\y) -- (\x+1.5,\y);		
			
			\fill(\x-1.5,\y) circle (.05);
			\draw[thick] (\x,\y) -- (\x-1.5,\y);

			\fill(\x+0,\y+1.5) circle (.05);
			\draw[thick] (\x,\y) -- (\x,\y+1.5);
			
			\fill(\x+0,\y-1.5) circle (.05);
			\draw[thick] (\x-1.5,\y) ..controls ( \x-.5, \y-1.25) and (\x+.5, \y-1.25) ..(\x+1.5,\y);

			\draw[->] (\x-.75, \y + .05) -- (\x - .05, \y+.75);
			\fill (\x-.5,\y+.5) node [] {$\rho$};
			
			\draw[<-] (\x+.75, \y + .05) -- (\x + .05, \y+.75);
			\fill (\x+.5,\y+.5) node [] {$\beta$};
			
			\draw[->] (\x-.75, \y - .05) -- (\x - .05, \y-.85);
			\fill (\x-.5,\y-.6) node [] {$\gamma$};
			
			\draw[<-] (\x+.75, \y - .05) -- (\x + .05, \y-.85);
			\fill (\x+.5,\y-.6) node [] {$\delta$};
			
			\draw[->] (\x+.65, \y - .05) .. controls (\x+.25, \y - .3) and (\x-.25, \y-.3) .. (\x - .65, \y-.05);
			\fill (\x,\y-.35) node [] {$\alpha$};
						}
		}
	\end{tikzpicture}
	\end{center}

Then the potential $W_T$ is $W_T = \alpha \gamma \beta + \alpha \beta \rho$. 
\end{example}

\begin{theorem} \cite{DLF}
Let $T$ and $T'$ be two triangulations of a marked surface $(S,M)$. If $T'$ is obtained from $T$ by flipping an arc $d_i$, then the QPs $(Q_{T'}, W_{T'})$ is obtained from the QP $(Q_{T},W_{T})$ via the QP mutation $\mu_i$.
\end{theorem}

\subsection{QPs of asymptotic triangulations}
Let $T = T_\partial \sqcup T_{\partial'}$ be an asymptotic triangulation. We now look at QPs and QP-mutation of the associated quiver $Q_T = Q_{\partial} \sqcup Q_{\partial'}$.

Consider a partial asymptotic triangulation $T_\beta$ of $C_{p,q}$ and its associated quiver $Q_\beta$. Following Definition \ref{def:pot}, the potential $W_\beta = W_{T_\beta}$ associated to $T_\beta$ results from adding all the 3-cycles in the quiver $Q_\beta$. 

As stated earlier, the 2-cyclicity of a quiver relies heavily on the potential. When we work with the triangulation and quiver side-by-side, it's easy to determine the potential. However, if we're given a quiver $Q_T$, we want to be able to perform QP mutation without seeing what happens in $T$. We can consider two types of quivers that we associate to $T$. The first is as described in Section 3.1 with framing vertices. There we can start with any framing quiver associated to an asymptotic triangulation, read off the potential directly from the quiver, and perform the QP mutation.

Here we describe how to define mutation on the ``classical" quiver (as in Def \ref{def:QT}) via QPs. Let $T = T_\partial \sqcup T_{\partial'}$ be an asymptotic triangulation. As mentioned earlier, the quiver $Q_T = Q_\partial \sqcup Q_{\partial'}$ may contain 2-cycles or loops. Furthermore, both $Q_\partial$ and  $Q_{\partial'}$ contain a negatively-oriented cycle around the meridian $z$. In particular, if all arcs are strictly asymptotic, the quiver $Q_\partial$ and $Q_{\partial'}$ are both simple cycles $\alpha_1 \ldots \alpha_p$ and $\alpha_{p+1} \ldots \alpha_{p+q}$. This is the type of quiver where framing vertices are identified. In this quiver model, we need to start with a strictly asymptotic triangulation (all arcs of the triangulation $T_\beta$ are strictly asymptotic), and because we have no internal triangles in $T_\beta$ our potential $W_\beta = 0$. We make this specification because our quiver may show 3-cycles that arise from going around the meridian, and we do not want this to be included in our potential. By starting with a potential $W_\beta = 0$, we can now work with quiver and QP mutation, and we will be able to keep certain loops and 2-cycles in $Q_\beta$, while eliminating others.

\begin{example}
Consider the potential $W = 0$ on the quiver
		\begin{center}
		\begin{tikzpicture}[scale = .25]
		\tikzstyle{every node} = [font = \tiny]
		\foreach \x in {0}
		{
			\foreach \y in {-8}
			{
			\fill (\x-3,\y+17) circle (.15);
			\fill(\x,\y+19.5) circle (.15);
			\fill(\x+3,\y+17) circle (.15);
			
			\fill[] (\x-3,\y+17) node [left] {$1$};
			\fill[] (\x,\y+19.5) node [above] {$2$};
			\fill[] (\x+3,\y+17) node [right] {$3$};
			
			\fill[] (\x-1,\y+19) node [left] {$c$};
			\fill[] (\x+1,\y+19) node [right] {$b$};
			\fill[] (\x-0,\y+16.8) node [above] {$a$};
		
			\draw [->,thick] (\x-2.75,\y+17.2)--(\x-.2,\y+19.4); 
			\draw [<-,thick] (\x-2.75,\y+16.8) -- (\x+2.75,\y+16.8); 

			\draw [->,thick] (\x+0.25,\y+19.4) -- (\x+2.8,\y+17.2);

			}
		}
		\end{tikzpicture}
		\end{center}

If we perform the \emph{premutation} $\tilde{\mu}_2$ on $(Q,W)$, we get $(\widetilde{Q},\widetilde{W})$ where $\tilde{Q}$ is the arrow span of the quiver

		\begin{center}
		\begin{tikzpicture}[scale = .25]
		\tikzstyle{every node} = [font = \tiny]
		\foreach \x in {0}
		{
			\foreach \y in {-8}
			{
			\fill (\x-3,\y+17) circle (.15);
			\fill(\x,\y+19.5) circle (.15);
			\fill(\x+3,\y+17) circle (.15);
			
			\fill[] (\x-3,\y+17) node [left] {$1$};
			\fill[] (\x,\y+19.5) node [above] {$2'$};
			\fill[] (\x+3,\y+17) node [right] {$3$};
			
			\fill[] (\x-1,\y+19) node [left] {$c^*$};
			\fill[] (\x+1,\y+19) node [right] {$b^*$};
			\fill[] (\x-0,\y+16.8) node [above] {$a$};
			\fill[] (\x-0,\y+16.8) node [below] {$[bc]$};
		
			\draw [<-,thick] (\x-2.75,\y+17.2)--(\x-.2,\y+19.4); 
			\draw [<-,thick] (\x+0.25,\y+19.4) -- (\x+2.8,\y+17.2);
			\draw [<-,thick] (\x-2.75,\y+16.8) -- (\x+2.75,\y+16.8); 
			\draw [->,thick] (\x-2.75,\y+16.5) -- (\x+2.75,\y+16.5); 
			}
		}
		\end{tikzpicture}
		\end{center}

and $\widetilde{W} = c^*b^*[bc]$. Then $\tilde{\mu}_2(Q,W) = \mu_2(\widetilde{Q},\widetilde{W}) = (Q',W')$. Now if we want to mutate at vertex 3, we perform the same steps. First we have the premutation $\tilde{\mu}_3$ on $(Q',W')$: 
	\begin{center}
		\begin{tikzpicture}[scale = .25]
		\tikzstyle{every node} = [font = \tiny]
		\foreach \x in {0}
		{
			\foreach \y in {-8}
			{
			\fill (\x-3,\y+17) circle (.15);
			\fill(\x,\y+19.5) circle (.15);
			\fill(\x+3,\y+17) circle (.15);
			
			\fill[] (\x-3,\y+17) node [left] {$1$};
			\fill[] (\x,\y+19.5) node [above] {$2'$};
			\fill[] (\x+3,\y+17) node [right] {$3$};
			
			\fill[] (\x-1.3,\y+19) node [left] {$[b^*[bc]]$};
			\fill[] (\x-.9,\y+18.9) node [below] {$c^*$};
			\fill[] (\x+1,\y+19) node [right] {$b^{**}=b$};
			\fill[] (\x+.2,\y+16.8) node [above] {$a^*$};
			\fill[] (\x+.2,\y+16.8) node [below] {$[bc]^*$};
			\fill[] (\x-4.5, \y+17) node [left] {$[[bc]a]$};
		
			\draw [->,thick] (\x-3,\y+17.5)--(\x-.6,\y+19.4); 
			\draw [<-,thick] (\x-2.73,\y+17.2)--(\x-.2,\y+19.2); 
			\draw [->,thick] (\x+0.25,\y+19.4) -- (\x+2.8,\y+17.2);
			\draw [->,thick] (\x-2.75,\y+16.8) -- (\x+2.75,\y+16.8); 
			\draw [<-,thick] (\x-2.75,\y+16.5) -- (\x+2.75,\y+16.5); 
			\draw [<-,thick] (\x-3.2,\y+16.5) .. controls (\x-5, \y+15) and (\x-5, \y+19) .. (\x-3.2,\y+17.5); 
			}
		}
		\end{tikzpicture}
		\end{center}

and our potential is $$\widetilde{W'} = c^* [b^*[bc]] + [bc]^*b[b^*[bc]] + a^*[bc]^*[[bc]a] = (c^* + [bc]^*b)[b^*[bc]] + a^*[bc]^*[[bc]a]. $$ Now we can check that $(c^* + [bc]^*b)$ is right-equivalent to an arrow $D: 2' \to 1$, which gives the potential $$\widetilde{W}' = D[b^*[bc]] + a^*[bc]^*[[bc]a] ,$$ and our Jacobian algebra is $P(\widetilde{Q}',\widetilde{W}') = R\langle\langle \widetilde{Q}' \rangle \rangle / J(\widetilde{W}')$, where our Jacobian ideal $J(\widetilde{W}')$ gives us the relations $D =0$ and $[b^*[bc]] = 0$, along with other relations. Thus our mutated potential $\mu_3(\widetilde{Q}',\widetilde{W}') = (Q'',W'')$ where $Q''$ is the arrow span of the quiver

		\begin{center}
		\begin{tikzpicture}[scale = .25]
		\tikzstyle{every node} = [font = \tiny]
		\foreach \x in {0}
		{
			\foreach \y in {-8}
			{
			\fill (\x-3,\y+17) circle (.15);
			\fill(\x,\y+19.5) circle (.15);
			\fill(\x+3,\y+17) circle (.15);
			
			\fill[] (\x-3,\y+17) node [left] {$1$};
			\fill[] (\x,\y+19.5) node [above] {$2'$};
			\fill[] (\x+3,\y+17) node [right] {$3'$};
			
			\fill[] (\x+1,\y+19) node [right] {$b$};
			\fill[] (\x+.2,\y+16.8) node [above] {$a^*$};
			\fill[] (\x+.2,\y+16.8) node [below] {$[bc]^*$};
			\fill[] (\x-4.5, \y+17) node [left] {$[[bc]a]$};
		
			\draw [<-,thick] (\x+0.25,\y+19.4) -- (\x+2.8,\y+17.2);
			\draw [->,thick] (\x-2.75,\y+16.8) -- (\x+2.75,\y+16.8); 
			\draw [<-,thick] (\x-2.75,\y+16.5) -- (\x+2.75,\y+16.5); 
			\draw [<-,thick] (\x-3.2,\y+16.5) .. controls (\x-5, \y+15) and (\x-5, \y+19) .. (\x-3.2,\y+17.5); 
			}
		}
		\end{tikzpicture}
	\end{center}
 
\end{example}

\bigskip

\def\t{\widetilde}

\appendix
\setcounter{section}{1}
\section{Cluster structure on asymptotic triangulations}
\begin{center}
{by Anna~Felikson and  Pavel~Tumarkin\footnote[1]{Department of Mathematical Sciences, Durham University, Science Laboratories, South Road, Durham, DH1 3LE, UK. email: \tt{anna.felikson@durham.ac.uk}, \tt{pavel.tumarkin@durham.ac.uk}}} 
\end{center}

\vspace{10pt} 

The aim of this short addendum is to introduce an alternative  cluster structure on asymptotic triangulations and provide a geometric interpretation.

\subsection{Double cover and double quiver}
As it is shown in Section~3, the most natural way to build a quiver from an asymptotic triangulation (i.e.  adjacency quiver) leads to loops and 2-cycles. To avoid this, consider a double cover $\widetilde C_{p,0}$ of the annulus $C_{p,0}$.
An asymptotic triangulation $T$ on $C_{p,0}$ induces an asymptotic triangulation $\t T$  on $\t C_{p,0}$, and the signed adjacency  quiver $ Q(\t T)$ of $\t T$ is free of loops and 2-cycles, so one can mutate it applying usual rules.

A flip of an arc  $d_i\in  T$ lifts as a composition of two commuting flips of arcs $d_i^1$ and $d_i^2$ in $\t T$, so, it has the same effect as a composition of two commuting mutations of $Q(\t T)$.

\subsection{Variables}  
To the arcs $d_1,\dots d_n$ of $T$ we assign independent variables $x_1,\dots, x_n$. Lifting this to the double cover results in pairs of identical variables 
$x_i^1,x_i^2$. To mutate the variables we use the usual exchange relations provided by the quiver $Q(\t T)$.
Since the initial quiver  $Q(\t T)$ is symmetric (with symmetrically assigned initial variables) and each mutation is a composition of the symmetric commuting mutations,
the symmetric structure on $Q(\t T)$ is preserved under mutations. 
We can also consider an exchange graph $\Gamma$ of  $(Q(\t T), \{ x_1,\dots, x_n \}) $ consisting of seeds obtained by composite mutations preserving the initial symmetry. 

\begin{figure}[h!]
\begin{center}
\psfrag{a}{\small $a$}
\psfrag{b}{\small $b$}
\psfrag{1}{\scriptsize 1}
\psfrag{2}{\scriptsize 2}
\psfrag{f1}{$(x,y)$}
\psfrag{f2}{$(\frac{y}{x},y)$}
\psfrag{f3}{$(\frac{y}{x},\frac{a+b}{x})$}
\psfrag{f4}{$(\frac{(a+b)^2}{y},\frac{a+b}{x})$}
\psfrag{f5}{$(\frac{(a+b)^2}{y},\frac{x(a+b)}{y})$}
\psfrag{f6}{$(x,\frac{x(a+b)}{y})$}
\includegraphics[width=.89\linewidth]{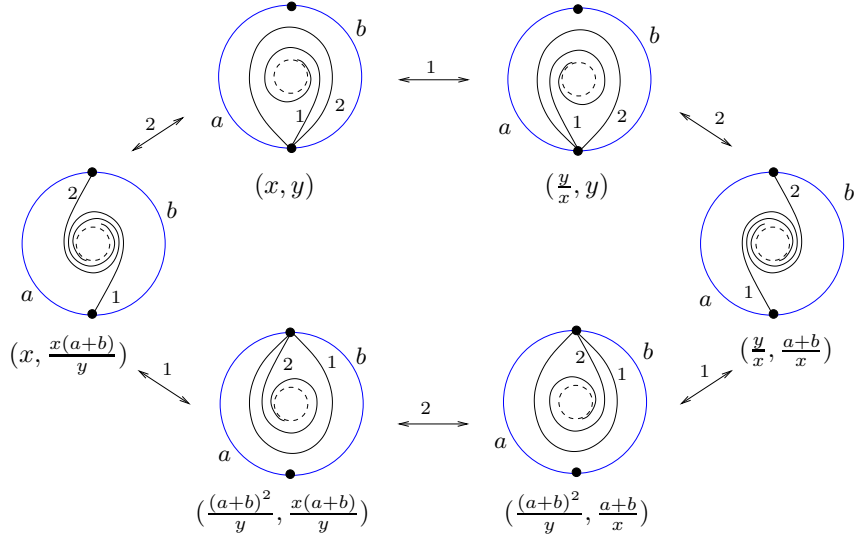}
\caption{Lambda lengths of the curves on a connected component of   $D_z^{\infty}(C_{2,q})$.}
\label{ex1}
\end{center}
\end{figure}

\subsection{Geometric interpretation of variables}

Consider the annulus $C_{p,q}$  as a surface with hyperbolic metric. While applying  Dehn twists in a closed curve $z$,
we can renormalise the metric on  $D_z^{\infty}(C_{p,q})$ so that the limit of the length of  $z$ is equal to $0$. 
Then  we can consider $D_z^{\infty}(C_{p,q})$ as a disjoint union of two hyperbolic punctured discs $\mathcal C_p$ and $\mathcal C_q$.  
Hence we are able to measure lambda lengths of the curves of the asymptotic triangulation (including strictly asymptotic arcs). 
Combinatorially, $z$  becomes a puncture, Pr\"ufer arcs  are tagged plane, adic arcs are tagged notched. 
See Fig.~\ref{ex1} for an example of an exchange graph and corresponding lambda lengths on $\mathcal C_{2}$ obtained in this way.

Now, assign to $x_1,\dots,x_n$  the values equal to the lambda lengths of the lifts of the arcs $d_1,\dots,d_n$ on  the double cover of  $D_z^{\infty}(C_{p,q})$. Then the cluster variables in the exchange graph $\Gamma$ 
will model the lambda lengths of arcs of asymptotic triangulations of  $D_z^{\infty}(C_{p,q})$.
More precisely, assuming (without loss of generality) that the initial asymptotic triangulation $T$ contained no adic arcs, the lambda lengths of finite arcs and Pr\"ufer arcs will be exactly equal to the values of the corresponding cluster variables, and the lambda lengths of the adic arcs will be halves of the corresponding variables
(this is caused by the fact that the length of the corresponding horosphere around the limit of $z$ is doubled in the cover 
$\t D_z^{\infty}(C_{p,q})$ of $D_z^{\infty}(C_{p,q})$).
See Fig.~\ref{ex2} for an example.

\begin{figure}[h!]
\begin{center}
\psfrag{1}{$1$}
\psfrag{2}{$2$}
\psfrag{x}{\scriptsize $1$}
\psfrag{x1}{\scriptsize $1$}
\psfrag{x2}{\scriptsize $1$}
\psfrag{y}{\scriptsize $2$}
\psfrag{y1}{\scriptsize $2$}
\psfrag{y2}{\scriptsize $2$}
\psfrag{a}{\scriptsize $a$}
\psfrag{b}{\scriptsize $b$}
\psfrag{f1}{$(x_1,x_2)=(x,y)$}
\psfrag{f2}{$(x_1,x_2)=(\frac{2y}{x},y)$}
\psfrag{f3}{$(x_1,x_2)=(\frac{2y}{x},\frac{2(a+b)}{x})$}
\psfrag{f4}{$(x_1,x_2)=(\frac{(a+b)^2}{y},\frac{2(a+b)}{x})$}
\psfrag{f5}{$(x_1,x_2)=(\frac{(a+b)^2}{y},\frac{x(a+b)}{y})$}
\psfrag{f6}{$(x_1,x_2)=(x,\frac{x(a+b)}{y})$}
\includegraphics[width=.99\linewidth]{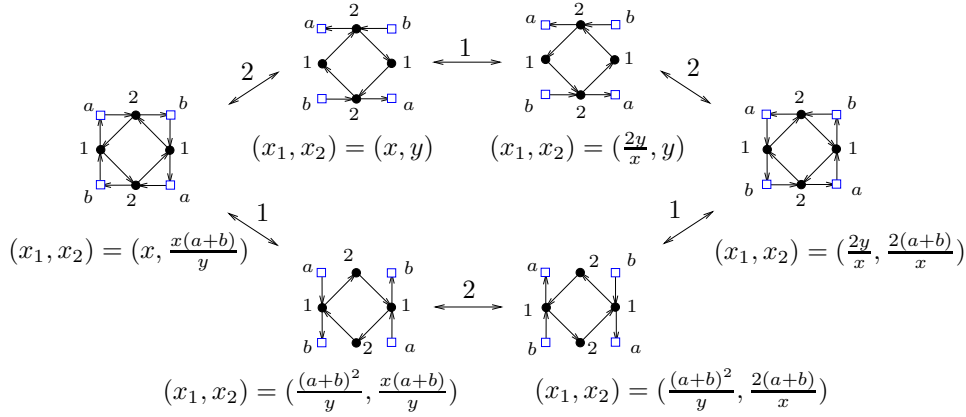}
\caption{Exchange graph $\Gamma$ and cluster variables for a connected component of $\t D^\infty_z(C_{2,q})$.}
\label{ex2}
\end{center}
\end{figure}

\subsection{From annulus to general hyperbolic surface}
The same procedure as described above for an annulus can be done for any triangulated hyperbolic surface $S$:
one can choose any simple closed curve $z\subset S$ and apply a sequence of Dehn twists in $z$, so that $z$ becomes shorter and shorter in a renormalised metric and turns into a cusp in the limit. A triangulation $T$ of $S$ turns into an asymptotic triangulation of $D_z^{\infty}(S)$. 
If in addition there exists a double cover of $S$ such that the curve $z$ is covered by one (twice longer) curve, then we can build the quiver of the asymptotic triangulation and realise corresponding variables as lambda lengths.

\medskip
\noindent
{\bf Acknowledgements.} We are grateful to Karin Baur and Michael Shapiro for inspiring discussions.

\end{document}